\documentclass[12,reqno]{amsart}
\usepackage{amsfonts} 
\usepackage{amsmath,amssymb,amsthm,bm}
\usepackage[active]{srcltx}
\usepackage{mathtools}
\usepackage{shuffle}
\usepackage[shortlabels]{enumitem}
\usepackage[alphabetic]{amsrefs}
\usepackage{empheq}
\usepackage{wasysym}
\usepackage{verbatim}
\usepackage{graphicx}
\usepackage[
bookmarks=true,         
bookmarksnumbered=true, 
colorlinks=true, pdfstartview=FitV, linkcolor=blue, citecolor=blue,
urlcolor=blue]{hyperref}
\usepackage{cancel}

\usepackage{graphicx}
\usepackage{amsfonts}
\usepackage{mathrsfs}
\usepackage{fancyhdr}
\usepackage{amsthm}
\usepackage{cases}
\usepackage{amsmath,amssymb,bm}
\usepackage{enumitem}
\usepackage{xcolor}
\usepackage{upgreek}

\usepackage[leftcaption]{sidecap} 
\usepackage{booktabs}
\usepackage{lipsum}
\usepackage{subfigure}
\usepackage{graphicx}
\usepackage{bm}
\usepackage{caption}
\usepackage{mathrsfs}
\usepackage{amsmath}
\usepackage{amsfonts}
\usepackage{caption} 
\usepackage{floatrow}
\newfloatcommand{capbtabbox}{table}[][\FBwidth]
\usepackage{threeparttable,booktabs} 
\PassOptionsToPackage{normalem}{ulem}
\usepackage{ulem}

\providecolor{added}{rgb}{0,0,1}
\providecolor{deleted}{rgb}{1,0,0}

\def\RR{\mathbb{R}}

\def\EE{\mathbb{E}}
\def\bT{\mathbb{T}}
\def\bL{\mathbb{L}}
\def\bN{\mathbb{N}}

\def\cA{\mathcal{A}}

\def\cF{\mathcal{F}}
\def\cC{\mathcal{C}}
\def\cD{\mathcal{D}}
\def\cH{\mathcal{H}}
\def\cI{\mathcal{I}}
\def\cJ{\mathcal{J}}
\def\cK{\mathcal{K}}

\def\cS{\mathcal{S}}
\def\cZ{\mathcal{Z}}
\def\cN{\mathcal{N}}
\def\sB{\mathscr{B}}
\def\LL{[-L,L]}
\def\TT{\mathbf{T}}
\newcommand\bP{\mathbf{P}}
\newcommand\HH{\mathfrak H}
\newcommand\diam { {\rm{diam\,}}}
\def\sB{\mathscr{B}}
\def\ls{{\lesssim}}

\def\es{{\simeq}}

\def\al{{\alpha}}

\def\tcg{\textcolor{green}}

\newcommand{\ep}{\varepsilon}

\newcommand{\si}{\sigma}
\newcommand{\affine}{{\rm add}}

\newcommand{\1}{{\bf 1}}

\newcommand{\blc}{\big(}
\newcommand{\brc}{\big)}
\newcommand{\Blc}{\Big(}
\newcommand{\Brc}{\Big)}
\newcommand{\blk}{\big[}
\newcommand{\brk}{\big]}
\newcommand{\Blk}{\Big[}
\newcommand{\Brk}{\Big]}
\newcommand{\lc}{\left(}
\newcommand{\rc}{\right)}
\newcommand{\lk}{\left[}
\newcommand{\rk}{\right]}
\newcommand{\lt}{\left }
\newcommand{\rt}{\right}

\newtheorem{theorem}{Theorem}[section]
\newtheorem{lemma}[theorem]{Lemma}
\newtheorem{proposition}[theorem]{Proposition}
\newtheorem{remark}[theorem]{Remark}
\newtheorem{corollary}[theorem]{Corollary}

\newtheorem{definition}[theorem]{Definition}
\theoremstyle{definition}

\numberwithin{equation}{section}

\begin{document}
 
\title{Stochastic Heat Equation with general 
rough   noise  
}
\thanks{Supported by   an NSERC discovery grant   and a startup fund from   University of Alberta at Edmonton.}

\author{Yaozhong Hu} 
\address{Department of Mathematical and Statistical Sciences, University of Alberta, Edmonton, AB T6G 2G1, Canada}
\email{yaozhong@ualberta.ca}
 

\author{Xiong Wang}
\address{Department of Mathematical and Statistical Sciences, University of Alberta, Edmonton, AB T6G 2G1, Canada}
\email{xiongwang@ualberta.ca}




\maketitle

\begin{abstract}
We study the well-posedness of a nonlinear one    dimensional stochastic heat equation   driven by  Gaussian noise:
$\frac{\partial u }{\partial t}=\frac{\partial^2 u }{\partial x^2}+\sigma(u )\dot{W} $,   where  $\dot{W} $  is white in time and     fractional in space with Hurst parameter $H\in(\frac 14,\frac 12)$.    In a
 recent paper 
 \cite{HHLNT2017}  by Hu, Huang, L\^{e}, Nualart and Tindel a technical and unusual condition 
 of $\sigma(0)=0$ was assumed which is critical in their approach. 
The main effort of   this paper is to  remove this     condition.    The idea is to work on a weighted space $\cZ_{\lambda,T}^p$ for some power decay weight
 $\lambda(x)=c_H(1+|x|^2)^{H-1}$.    
In addition,  when $\si(u)=1$ we obtain 
the exact asympotics
  of  the solution $u_{\affine}(t,x)$  as $t$ and $x$ go  to infinity.
In particular, we find the exact  growth of $\sup_{|x|\leq L}{|u_{\affine}(t,x)|}$ and  the sharp growth rate for the H\"older
coefficients, namely,   $\sup_{|x|\leq L} \frac{|  u_{\affine}(t,x+h)-u_{\affine}(t,x)|}{|h|^\beta}$ and $\sup_{|x|\leq L} \frac{| u_{\affine}(t+\tau,x)-u_{\affine}(t,x)|}{\tau^\al}$.  
\end{abstract}
 
 \begin{abstract}
 Nous \'etudions une \'equation de chaleur stochastique \'a une dimension spatiale non lin\'eaire entr\^an\'ee par le bruit gaussien:
$ \frac {\partial u} {\partial t} = \frac{\partial ^ 2 u} {\partial x ^ 2} + \sigma (u) \dot {W} $, o\`u $ \dot {W} $ est blanc dans le temps et   {  fractionnaire dans le espace avec le param\`etre Hurst 
$ H \in (\frac 14, \frac 12) $.}     Dans un
 article r\'ecent
 \cite{HHLNT2017} par Hu, Huang, L\^{e}, Nualart et Tindel une condition technique et inhabituelle
 de $ \sigma (0) = 0 $ a \'et\'e suppos\'e, 
 ce qui est critique dans leur approche.
Le principal effort de ce document est de supprimer cette condition. L'id\'ee est de travailler sur un espace pond\'er\'e $ \cZ _ {\lambda, T} ^ p $ pour un certain poids de d\'ecroissance de puissance
 $ \lambda (x) = c_H (1+ | x | ^ 2) ^ {H-1} $.
Lorsque $ \si (u) = 1 $ nous obtenons 
les asympotiques exacts
  de la solution $ u_ {\affine} (t, x) $ as $ t $ et $ x $ vont à l'infini. 
En particulier, nous trouvons la croissance exacte de 
$ \sup_{| x | \leq L}  | u _{\affine} (t, x) |  $ et la croissance exacte
  des   
coefficients de H\"older, c'est-\`a-dire,  $ \sup_{| x | \leq L} \frac {| u _{\affine} (t, x + h) -u _ {\affine} (t, x) |} {| h | ^ \beta} $ et $ \sup_ {| x | \leq L} \frac {| u _ {\affine} (t + \tau, x) -u _ {\affine} (t, x) |} {\tau ^ \al} $.
 \end{abstract}

\section{Introduction  and main results}
In this paper, we consider the following one   dimensional
(in space variable) nonlinear  stochastic   heat  equation driven by  the  Gaussian noise which is white in time and fractional in space:
\begin{equation}\label{SHE}
  \frac{\partial u(t,x)}{\partial t}=\frac{\partial^2 u(t,x)}{\partial x^2}+\sigma(t,x,u(t,x))\dot{W}(t,x),\quad   t\geqslant 0,\quad
   x\in\RR\,,
\end{equation}
where $W(t,x)$ is a centered Gaussian process with covariance given by
\begin{equation}\label{CovW}
  \EE[W(t,x)W(s,y)]=\frac 12 (s\wedge t) \blc |x|^{2H}+|y|^{2H}-|x-y|^{2H} \brc 
\end{equation}
and where $\frac 14<H<\frac 12$ and $\dot{W}(t,x)=\frac{\partial^2  }{\partial t\partial x}
W(t,x)$.


 There has been a lot of  work on stochastic heat equations driven by general Gaussian noises.  We refer to \cite{hu19} for a short survey and for more references.  The main feature of this work is that the noise is rough  (e.g. $\frac 14<H<\frac 12$) in  { the} space variable.
 We mention three  works that are directly related to this  specific Gaussian noise structure.   The first two  are
  \cite{BJQ2015}  and \cite{HHLNT2019}, where  the authors study the existence, uniqueness and some properties such as moment bounds of the   mild solution   when   the diffusion coefficient $\sigma$  is affine (i.e. ${  \si(t,x,u)=}\sigma(u)=au+b$)  in \cite{BJQ2015} or  linear
  (i.e. $\sigma(u)=au $)  in  \cite{HHLNT2019}.  After these works researchers tried to study \eqref{SHE} for general nonlinear $\sigma$.   
   However, the  method effective for affine (and linear) equations   cannot  no longer work.  
{     One difficulty
 for general nonlinear diffusion coefficient $\sigma$
  is    that we cannot  no longer bound
 $|\sigma(u_1)-\sigma(u_2)-\sigma(v_1)+\sigma(v_2)|$ by  a multiple of
 $ |u_1 - u_2 - v_1+v_2 |$  (which is possible only in the   affine case).}
A breakthrough was made in   \cite{HHLNT2017}.
 However, to solve    equation \eqref{SHE}   the authors in \cite{HHLNT2017}
  have to  assume  that  $\sigma(0)=0$, 
%
%
which  does not even  cover the affine case studied in \cite{BJQ2015}!   The main motivation of this paper is to remove the condition $\sigma(0)=0$
assumed in \cite{HHLNT2017}.   To this end we need to understand   why this condition is so crucial there.  We first find out that this condition    $\sigma(0)=0$   can ensure the solution lives in the space 
 $\cZ_T^p$    
  (see \cite{HHLNT2017} or \eqref{ZNorm} in Section \ref{weak} of this paper with $\lambda(x)=1$).  As we shall see that even in  the simplest case $\sigma(u)\equiv 1$ (of the case $\si(0)\not=0$), 
  the solution  
  is {  no longer}  in  $\cZ_T^p$    (see e.g. Theorem \ref{LUEESup}   and Proposition \ref{Prop.SupN}). Moreover, the initial condition 
  in \cite{HHLNT2017}  must be integrable to guarantee the solution belongs to $\cZ_T^p$, which means $u_0(x)=1$ is excluded.
Thus,  {  to remove the restriction $\si(0)=0$,  
we must find another appropriate  solution space. }  Our  idea is to  introduce a decay weight (as the spatial variable $x$ goes to infinity)  to {  enlarge }  the  solution space  $\cZ_T^p$  to a weighted space
$\cZ_{\lambda, T}^p$ for some suitable power decay function $\lambda(x)$.
This weight function  will have to be chosen appropriately (not too fast and not too slow.
See Section 2 for details).

The introduction of  the weight makes all the tools used in    \cite{HHLNT2017}   {collapse}.  As we can see we shall need  a whole set of new understandings of the heat kernel to complete our program. People may wonder whether   one can still just use  $\cZ_T^\infty $ 
for our solution space. This question is natural since we work on the whole real line  $\RR$ for the space variable. A constant function is in $L^\infty(\RR)$ but not in $L^p(\RR)$ for any finite $p$.
{  If it happens to be   possible to use $\cZ_T^\infty $
(without weight), then many  computations in \cite{HHLNT2017}  will still be valid and the problem becomes greatly simplified. }

To see if  this is possible or not we consider the solution $u_\affine(t,x)$ to the equation with additive noise, which is  the solution to \eqref{SHE} with $\sigma(u)=1$ and with initial condition $u_0(x)=0$.  {  This is the simplest case that $\si(0)\not=0$.  
To find out}  if $u_\affine(t,x)$ is in $\cZ_T^\infty $ or not (or to see if
 the introduction of decay weight $\lambda$ is necessary  or  not), we shall find the sharp bound of the      solution $u_\affine(t,x)$     as $x$ goes to infinity.   In other words,  we shall find the exact explosion rate of
$\sup_{|x|\le L} |u_\affine(t,x)|$ as $L$ goes to infinity.  This problem has a great value of its own. To study  the supremum of a family of random variables,  there are two    powerful  tools: one is to use
 the independence and  the other one is to use the
 martingale inequalities.   However,
$u_\affine(t,x)$ is not a martingale with respect to the  spatial variable
$x$ (nor it is a  {martingale}  with respect to the time variable $t$)   and since the noise $\dot W$ is not independent in the spatial variable either, the application of independence may be much more involved (We refer,  however, to \cite{Chen2016, CHNT2017, DMD2013, DMDS2013}   for  some successful applications  of the independence in the stochastic heat equation \eqref{SHE}). In this work, we  shall use instead  the idea of   majorizing measure to obtain  sharp   growth of
$\sup_{|x|\le L} |u_\affine(t,x)|$ 
and $   \sup_{0\le t\le T,|x|\le L} |u_\affine(t,x)|$, as $L$ and $T$ go to infinity, both in terms of expectation and  almost surely.
 More precisely, we have

\begin{theorem}\label{LUEESup}
    Let the Gaussian field $u_\affine(t,x)$    be the solution to \eqref{SHE} with $\sigma(t,x, u)\\=1$ and $u_0(x)=0$.    Then, we have the following statements.
    \begin{enumerate}
    \item[(1)] There are two  positive
    constants $c_H$ and $ C_H$, independent of $T$ and $L$,   such that
\begin{equation}
c_H \, \Uppsi(T, L)\le
\EE \lk\sup_{{0\le t\le T\atop  -L\le x\le L}}  u_\affine (t,x) \rk\le \EE \lk\sup_{{0\le t\le T\atop  -L\le x\le L}} |u_\affine (t,x)|\rk\le C_H  \, \Uppsi(T, L)\,,
 \label{SupEEasmp}
    \end{equation}
    where $\Uppsi_0(T, L):=1+\sqrt{\log_2\lc L/\sqrt{T} \rc}\,, L\ge \sqrt{T}$ and 
    \begin{equation}
 \Uppsi(T, L):= \begin{cases} T^{\frac{H}{2}}\Uppsi_0(T, L) & \text{if $L\geq \sqrt{T}$},  \\
      T^{\frac{H}{2}} & \text{if $L < \sqrt{T}$}\,.
    \end{cases}\label{e.def_rho}
    \end{equation}
\item[(2)]
    There are two strictly positive random constants $c_H$ and $C_H$, independent of $T$ and $L$,  such that almost surely
    \begin{align}\label{Supasmp}
    c_H\, T^{\frac{H}{2}}\Uppsi_0(T, L)  &\le   \sup_{(t,x)\in\Upupsilon(T,L)}  u_\affine (t,x)   \\
    &\leq  \sup_{(t,x)\in\Upupsilon(T,L)} |u_\affine (t,x)|\le C_H\,
        T^{\frac{H}{2}}\Uppsi_0(T, L)   \,, \nonumber
    \end{align}
  where $\Upupsilon(T,L)
  =\{(t,x)\in[0,T]\times[-L,L]~:~L\geq \sqrt{T}\}$\,.
    \end{enumerate}
\end{theorem}
{  Let us point out that
Theorem  \ref{LUEESup}  is an extension of  Theorem 1.2 of  \cite{DMD2013} and Theorem 2.3 of \cite{DMDS2013} to spatial rough   noise.  }

It is well-known that the   solution to   equation
\eqref{SHE}, if exists,  is usually H\"older continuous on any bounded domain. But    usually  it is not   H\"older continuous on the whole space. An interesting question to ask is
how the H\"older coefficient depends on the size of the domain. 
Since the  additive solution $u_\affine(t,x)$ is a Gaussian random field we will be able to obtain  sharp dependence
on  the size of the domain of  the H\"older coefficient.
In the following  theorem 
we   state our result on the H\"older continuity in spatial variable over unbounded domain. 
\begin{theorem}\label{LUHolderEESup}
	Let $u_\affine (t,x)$      be the solution to \eqref{SHE} with $\sigma(t,x, u)=1$ and $u_0(x)=0$   and denote
	\[
	\Delta _h u_\affine (t,x):=u_\affine (t,x+h)-u_\affine (t,x)\,.
	\]
  Let $0<\theta<H$ be given. Then, there are  positive constants $c$, $c_H$
	and $C_{H, \theta}$ such that    the following inequalities hold  true:
\begin{align}
  c_{H}\,|h|^{H} \Uppsi_0(t, L)&\leq \EE \lk\sup_{-L\le x\le L} \Delta _h u_\affine (t,x) \rk \label{SupHolderEEasmp}\\
       & \leq \EE \lk\sup_{-L\le x\le L} |\Delta _h u_\affine (t,x)|\rk\leq C_{H,\theta}\, t^{\frac{H-\theta}{2}} |h|^{\theta} \Uppsi_0(t, L)  \nonumber
	\end{align}
for all  
	$L\geq \sqrt{t}>0$ and    $0<|h|\leq c(\sqrt{t}\wedge1)$. Moreover,  there are two (strictly) positive  random constants $c_{H}$ and
    $C_{H,\theta} $ 
\begin{align}
	c_{H}\, |h|^{H} \Uppsi_0(t, L)
	 &\leq \sup_{-L\le x\le L}   \Delta _h u_\affine (t,x)
\label{SupasmpHolder}  \\
	 &\leq \sup_{-L\le x\le L} \left| \Delta _h u_\affine (t,x) \right| \leq
    C_{H,\theta}\, t^{\frac{H-\theta}{2}} |h|^{\theta} \Uppsi_0(t, L)
\nonumber
	\end{align}
	for all  
	$L\geq \sqrt{t}>0$ and    $0<|h|\leq c(\sqrt{t}\wedge1)$.
\end{theorem}

Next,  we study the H\"older continuity     in time over the unbounded domain. We state  the following.
\begin{theorem}\label{LUtHolderEESup}
Let $u_\affine (t,x)$      be the solution to \eqref{SHE} with $\sigma(t,x, u)=1$ and $u_0(x)=0$   and denote
	\begin{equation*}
	\Delta_\tau u_\affine(t,x):=u_\affine (t+\tau,x)-u_\affine(t,x) \,.
	\end{equation*}
Let $0<\theta<H/2$.   Then,  
there are  positive constants $c$, $c_H$
	and $C_{H, \theta}$ such that 
\begin{align}	
 c_{H}\, \tau^{\frac H2} \Uppsi_0(t, L)&\leq \EE \lk\sup_{-L\le x\le L}  \Delta_\tau u_\affine(t,x) \rk\label{SuptHolderEEasmp}\\
    &\leq \EE \lk\sup_{-L\le x\le L} |\Delta_\tau u_\affine(t,x)|\rk\leq
     C_{H,\theta}\, t^{\frac{H}{2}-\theta} \tau^{\theta} \Uppsi_0(t, L) \nonumber
\end{align}
for all  
	$L\geq \sqrt{t}>0$ and    $0<\tau\leq c( {t}\wedge1)$.
We also have the almost sure version of the above result. 
\begin{align}		
c_{H}\, \tau^{\frac H2} \Uppsi_0(t, L) &\leq \sup_{-L\le x\le L} \Delta_\tau u_\affine(t,x)
\label{SuptasmpHolder}\\
&\leq \sup_{-L\le x\le L} |\Delta_\tau u_\affine(t,x) |\leq
    C_{H,\theta}\, t^{\frac{H}{2}-\theta} \tau^{\theta}\Uppsi_0(t, L)   \nonumber
\end{align}
for all  
	$L\geq \sqrt{t}>0$ and    $0<\tau\leq c( {t}\wedge1)$.  Now, $c_H$
	and $C_{H, \theta}$ are random. 
\end{theorem}

The above Theorems \ref{LUEESup}-\ref{LUtHolderEESup} are proved in Section 3.
Now let us return to the equation \eqref{SHE}.
To make things precise we give  here   the  definitions of strong and weak solutions.
\begin{definition}\label{DefMildSol}
 Let $\left\{u(t,x)\,, t\ge 0\,, x\in \RR\right\}$ be a real-valued adapted  stochastic process such that for all $t\in[0,T]$ and $x\in \RR$ the process $\{G_{t-s}(x-y)\sigma(s,y,u(s,y))\1_{[0,t]}(s)\}$ is integrable with respect to $W$ (see Definition \ref{Def_SI_Gene}), where $G_{t}(x):=\frac{1}{\sqrt{4\pi t}}\exp\lc-\frac{x^2}{4t}\rc$ is the heat kernel on $\RR$ associated with
 the  Laplacian operator $\Delta$.
\begin{enumerate}
\item[(i)]\     We say that $u(t,x)$ is a  \emph{strong  (mild)   solution}  to \eqref{SHE} if for all $t\in[0,T]$ and $x\in\RR$ we have
   \begin{equation}\label{MildSol}
     u(t,x)=G_t\ast u_0(x)+\int_{0}^{t}\int_{\RR} G_{t-s}(x-y)\sigma(s,y,u(s,y))W(ds,dy)
   \end{equation}
 almost surely,    where the stochastic integral is understood in the sense of Definition \ref{Def_SI_Gene}.
\item[(ii)]\      We say   \eqref{SHE}   has a \emph{weak solution}   if there exists a probability space with a filtration $(\widetilde{\Omega},\widetilde{\cF},\widetilde{\bP},\widetilde{\cF}_t)$,  a Gaussian random field   $\widetilde{W}$ identical to $W$ in law,  and an  adapted  stochastic process
   $\left\{u(t,x)\,, t\ge 0\,, x\in \RR\right\}$ on this probability space  $(\widetilde{\Omega},\widetilde{\cF},\widetilde{\bP},\widetilde{\cF}_t)$ such that $u(t,x)   $ is a mild solution to   \eqref{SHE}  with respect to
   $(\widetilde{\Omega},\widetilde{\cF},\widetilde{\bP},\widetilde{\cF}_t)$  and  $\widetilde{W}$.
   \end{enumerate}
\end{definition}

Before stating our    theorem
on the existence of a weak solution, we make the following assumption.
  \begin{enumerate}
  \item[\textbf{(H1)}]\  $\sigma(t,x,u)$ is jointly continuous
  over {  $[0, T]\times \RR^2$}  and is at most of  linear growth in $u$ uniformly in $t$ and $x$.  This means
  \begin{equation}\label{LinGrowSigam}
    \sup\limits_{t\in[0,T],x\in\RR}|\sigma(t,x,u)|\leq C (|u|+1)
  \end{equation}
for some positive constant $C $.  We also assume that  it is uniformly  Lipschitzian  in $u$, namely,
\  $\forall\  u,v\in\RR$
  \begin{equation}\label{LipSigam}
    \sup\limits_{t\in[0,T],x\in\RR}
    |\sigma(t,x,u)-\sigma(t,x,v)|\leq C |u-v|\,,
  \end{equation}
  for some constant $C >0$.
\end{enumerate}
We can now state our main theorem of the paper.
\begin{theorem}\label{WeakSol}  Let $ \lambda(x) = c_H(1+|x|^2)^{H-1} $  satisfy  $\int_\RR \lambda(x)dx=1$.
   Assume   that
      $\sigma(t,x,u)$   satisfies hypothesis  \textbf{(H1)}    and that  the initial data $u_0$ belongs to $\cZ_{\lambda,0}^p$
 for some  $p>\frac{6}{4H-1}$ (see Section 4.1 for the definition of $\cZ_{\lambda,T}^p$).
   Then,   there exists a   weak solution to \eqref{SHE} with sample paths   in $\cC([0, T]\times\RR)$ almost surely. In addition,     for any  $\gamma<H-\frac 3p$,    the process $u(\cdot,\cdot)$
    is almost surely   H\"{o}lder continuous   on any compact sets in $[0, T]\times \RR$ of H\"older  exponent $\gamma/2$ with respect to
    the time variable $t$ and of  H\"older  exponent $\gamma $ with respect to
    the spatial variable $x$.
 \end{theorem}
From Theorem  \ref{LUEESup} we see that when $\si(0)\not=0$ we  expect
that the solution will not be  in the   space $\cZ_T^p$. We 
enlarge it to the weighted space $\cZ_{\lambda,T}^p$ in the above theorem.
As we said earlier that 
the introduction of the weight $\lambda$ makes the computations
in \cite{HHLNT2017} no longer applicable.  For example, now we need to control, roughly speaking,   a certain norm  of  $
\lambda(\cdot) \int_0^*\int_{\RR} G_{*-s}(\cdot-y)  u(s,y) dW(s,y)
$ and its fractional derivative (with respect to spatial variable)
 by the similar  norm of $\lambda(\cdot) u(*, \cdot)$ and its fractional derivative. 
 This would require us to study the delicate properties of   
$
 \lambda(x)   G_t(x-y)  \lambda^{-1} (y) 
$. 
Thus,   we need some very subtle and  very sharp bounds on the heat kernel $G_t(x-y)$  with respect to the weight
 function   $\lambda(x)$,
 which are  of interest in their  own.  This is done in Section \ref{lemmas}.  After these preparations, we shall show the above theorem  in Section \ref{weak}.  Although the techniques of \cite{HHLNT2017}  {are} no longer effective
in our new  situation  we still follow    the {same}  spirit there. 

It is always interesting   to have    existence and uniqueness of the  strong solution. {  As we said earlier,}  due to the roughness of the noise   we need to handle, as in \cite{HHLNT2017},    the square  increment  $|\sigma(u_1)-\sigma(u_2)-\sigma(v_1)+\sigma(v_2)|$.    It seems too complicated   for  the weighted space.        So,  to show the   existence and uniqueness of strong solution we assume that the derivative of the diffusion coefficient in \eqref{SHE}  possesses  a decay itself   as  $x\rightarrow \infty$.  More precisely, we   make the following assumptions.
\begin{enumerate}
  \item[\textbf{(H2)}]   Assume  that   $\sigma(t,x,u)\in\cC^{0,1,1}([0, T]\times\RR^2)$ satisfies the following conditions:
  $|\sigma'_u(t,x,u)|$ and $|\sigma''_{xu}(t,x, u)|$ are uniformly bounded, i.e. there is a constant  $C>0$ such that
  \begin{align}
  \sup_{t\in[0,T],x\in\RR,u\in\RR} |\sigma'_u(t,x, u)| &\leq C\,; \label{DuSigam}\\
  \sup_{t\in[0,T],x\in\RR,u\in\RR} |\sigma''_{xu}(t,x, u)| &\leq C   \,. \label{DxuSigam}
  \end{align}
Moreover, we assume {  that for some $p>\frac{6}{4H-1}$,} 
  \begin{equation}\label{DuSigamAdd}
   \sup_{t\in[0,T],x\in\RR}\lambda^{-\frac{1}{p}}(x)\left| \sigma'_u(t,x,u_1)-\sigma'_u(t,x, u_2)\right| \le C  |u_1-u_2| \,.
  \end{equation}

\end{enumerate}


\begin{theorem}\label{ModUniq}
Let  $\sigma$ satisfy the above  hypothesis  \textbf{(H2)} and assume that for some $p>\frac{6}{4H-1}$,    $u_0\in\cZ_{\lambda,0}^p$. Then \eqref{SHE} has a unique strong solution with sample paths   in $\cC([0, T]\times\RR)$ almost surely. Moreover, the process $u(\cdot,\cdot)$ is uniformly H\"{o}lder continuous almost surely  on any compact subset  of  $[0, T]\times\RR$ with the same temporal and spatial H\"{o}lder exponents as those  in Theorem \ref{WeakSol}.
\end{theorem}
This theorem will be proved in Section \ref{strong}.  Let us  point out that if  $\sigma(u)$ is affine,  then it satisfies the assumption \textbf{(H2)}.


\section{Auxiliary Lemmas}\label{lemmas}
 In this section, we shall obtain some estimates  about  the heat kernel   $G_t(x)=\frac{1}{\sqrt{ 4\pi t}  }e^{-\frac{|x-y|^2}{4t}}$  associated with the Laplacian $\Delta$ combined  with the decay  weight $\lambda(x)$. These estimates are the key ingredients  to establish our results.

\subsection{Covariance structure}
We start  by recalling  some notations used in \cite{HHLNT2017}.    Denote by  $ \cD=\cD(\RR) $ the space of smooth functions
on $\RR$  with compact support, and by $\cD'$  the   dual of $\cD$ with respect to the $L^2(\RR, dx)$. The Fourier transform of a function $f\in\cD$ is defined as
 \[
\hat f(\xi)=  \cF f(\xi)=\int_{\RR} e^{-i\xi x}f(x) dx,
 \]
 and the inverse Fourier transform is  then  given by $\cF^{-1}g(x)=\frac{1}{2\pi}\cF g(-x)$.

 Let   $(\Omega,\cF,\bP)$ be a complete probability space
 and let  $H\in (\frac14 , \frac 12)$ be given and fixed.
Our noise $\dot W$ is  a zero-mean Gaussian family $\{W(\phi), \phi\in\cD(\RR_{+}\times \RR)\}$  with covariance structure given by
 \begin{equation}\label{CovStru}
   \EE\lk W(\phi)W(\psi)\rk=c_{1,H}\int_{\RR_{+}\times \RR} \cF \phi(s,\xi) \overline{\cF \psi(s,\xi)} |\xi|^{1-2H}d\xi ds,
 \end{equation}
 where $c_{1,H} $  is given below by \eqref{e.c1}
 and $\cF \phi(s,\xi)$ is the Fourier transform with respect to the spatial variable $x$ of the function $\phi(s,x)$. Let $\cF_t$ be the filtration generated by $W$.  This  means
\[
 \cF_t=\sigma\{W(\phi(x)\1_{[0,r]}(s)):r\in[0,t],\phi(x)\in \cD(\RR)\}.
\]
Equation \eqref{CovStru}  defines   a Hilbert scalar product on $\cD(\RR_{+}\times \RR) $.  To express this product without the use of Fourier   transform, we recall   the Marchaud fractional derivative $D_{-}^{\beta}$ of order $\beta\in(0,1)$. For a function $\phi: \RR_{+}\times \RR\rightarrow \RR$, the Marchaud fractional derivative  $D_{-}^{\beta}$ is defined as:
 \begin{equation}\label{MFDer}
   D_{-}^{\beta} \phi(t,x)=\lim_{\ep\downarrow 0}D_{-,\ep}^{\beta} \phi(t,x)=\lim_{\ep\downarrow 0}\frac{\beta}{\Gamma(1-\beta)}\int_{\ep}^{\infty}\frac{\phi(t,x)-\phi(t,x+y)}{y^{1+\beta}} dy.
 \end{equation}
 We also define the Riemann-Liouville fractional integral of order $\beta$ of a function $\phi$  by 
 \begin{equation*}
   I^{\beta}_{-} \phi(t,x)=\frac{1}{\Gamma(\beta)}\int_{x}^{\infty} \phi(t,y)(y-x)^{\beta-1}dy.
 \end{equation*}
%
%
Set
 \begin{equation}\label{hSpace}
   \HH=\{\phi:\RR_{+}\times \RR\to \RR~ \,; \ ~\exists \psi\in L^2(\RR_{+}\times \RR)~s.t.~ \phi(t,x)=I^{\frac 12-H}_{-} \psi(t,x)\}.
 \end{equation}
 With these  notations we can express  the Hilbert space obtained by completing $\cD(\RR_{+}\times \RR) $ with respect to the scalar  product given by \eqref{CovStru} in
the following proposition
(see e.g.
 \cite{PT2000} for a proof).

 \begin{proposition}\label{hSpaceProp}
   The function space $\HH$ is a Hilbert space equipped with the scalar  product
\begin{align}
     \langle\phi,\psi\rangle_{\HH}
     =&c_{1,H}\int_{\RR_{+}\times \RR} \cF \phi(s,\xi) \overline{\cF \psi(s,\xi)} |\xi|^{1-2H}d\xi ds \label{hinner.1}
     \\
     =&c_{2,H}\int_{\RR_{+}\times \RR} D_{-}^{\frac 12 -H} \phi(t,x)D_{-}^{\frac 12 -H} \psi(t,x)dxdt\label{hinner.2}\\
     =& {c_{3,H}  \int_{\RR_{+}}\int_{\RR^2}[\phi(t,x+y)-\phi(t,x)][\psi(t,x+y)-\psi(t,x)] |y|^{2H-2} dxdydt}\,,\label{hinner.3}
   \end{align}
   where
   \begin{align}
   c_{1,H}=&\frac{1}{2\pi}\Gamma(2H+1)\sin(\pi H)\,;   \label{e.c1}\\
   c_{2,H}=&\lk\Gamma\Blc H+\frac 12\Brc\rk^2\left(\int_{0}^{\infty} \Blk(1+t)^{H-\frac 12}-t^{H-\frac 12}\Brk^2 dt+\frac{1}{2H}\right)^{-1}\,;\\
    c_{3,H}   =&\sqrt{H (\frac 12-H) }\, c_{2,H}^{-1/2}\,.  \label{e.c3}
   \end{align}
 The space  $\cD(\RR_{+}\times \RR)$ is dense in $\HH$.
 \end{proposition}

 The Gaussian  space $\HH$ is the same as  the homogeneous Sobolev space $\dot{\cH}^{\beta}$ for $\beta=\frac12-H\in(0,\frac 12)$ in harmonic analysis (\cite{BCD}).
The Gaussian family $W=\{W(\phi),~\phi\in\cD(\RR_{+}\times \RR)\}$  can be extended to an isonormal Gaussian process $W=\{W(\phi),~\phi\in\HH\}$ indexed by the Hilbert space $\HH$. It is easy to see that $\phi(t,x)=\chi_{\{[0, t]\times[0, x]\}}$, $t\in \RR_+$ and $x\in \RR$,  is in $\HH$
(we set $\chi_{\{[0, t]\times[0, x]\}}=-\chi_{\{[0, t]\times[x, 0]\}}$
if $x$ is negative).
We denote $W(t,x)=W(\chi_{\{[0, t]\times[0, x]\}})$.

 \subsection{Stochastic integration}
We first define stochastic integral for elementary integrands and then extend it to general ones.

 \begin{definition}\label{ElemP}
An elementary process $g$ is a process of the following form
   \[
    g(t,x)=\sum_{i=1}^{n}\sum_{j=1}^{m} X_{i,j}\1_{(a_i,b_i]}(t)\1_{(h_j,l_j]}(x),
   \]
   where $n$ and $m$ are finite positive integers, $-\infty<a_1<b_1<\cdots<a_n<b_n<\infty$, $h_j<l_j$ and $X_{i,j}$ are $\cF_{a_i}$-measurable random variables for $i=1,\dots,n$. The stochastic integral of such an elementary  process with respect to $W$ is defined as
   \begin{equation}\label{SI_Ele}
     \begin{split}
         &\int_{\RR_{+}}\int_{\RR} g(t,x)W(dt,dx)
     = \sum_{i=1}^{n}\sum_{j=1}^{m}  X_{i,j}W(\1_{(a_i,b_i]}\otimes \1_{(h_j,l_j]}) \\
     &\qquad\qquad =\sum_{i=1}^{n}\sum_{j=1}^{m}  X_{i,j}\big[W(b_i,l_j)-W(a_i,l_j)-W(b_i,h_j)+W(a_i,h_j)\big].
     \end{split}
   \end{equation}
 \end{definition}
%
 \begin{proposition}\label{SI_Gene}
   Let $\Lambda_{H}$ be the space of predictable processes $g$ defined on $\RR_{+}\times\RR$ such that almost surely $g\in\HH$ and $\EE[\|g\|_{\HH}^2]<\infty$. Then, the space of elementary processes
    defined in Definition \ref{ElemP} is dense in $\Lambda_{H}$.
  \end{proposition}

\begin{definition}\label{Def_SI_Gene} For $g\in\Lambda_{H}$, the stochastic integral $\int_{\RR_{+}\times\RR} g(t,x)W(dt,dx)$ is defined as the $L^2(\Omega)$  limit of stochastic integrals of the  elementary processes approximating $g(t,x)$
in $\Lambda_H$, and we have the following isometry equality
    \begin{equation}\label{Isometry}
     \EE\lk\lc\int_{\RR_{+}\times\RR} g(t,x)W(dt,dx)\rc^2\rk=\EE\lk\|g\|_{\HH}^2\rk.
    \end{equation}
    \end{definition}



\subsection{Auxiliary Lemmas}
We shall find a  solution to   equation
\eqref{SHE} in the space $\cZ_{\lambda, T}^p$. To deal with weight $\lambda$  we need   a  few technical results concerning the interaction between  the weight $\lambda(x)$  and the Green's function $G_t(x-y)$.
 \begin{lemma}\label{TechLemma1}
   For any $\lambda \in \RR$, $\lambda(x)=\frac{1}{(1+|x|^2)^{\lambda}}$ and $T>0$, we have
   \begin{equation}\label{Glamd}
     \sup_{0\leq t\leq T}\sup_{x\in\RR} \frac{1}{\lambda(x)}\int_{\RR} G_{t}(x-y) \lambda(y) dy<\infty.
   \end{equation}
 \end{lemma}
 \begin{remark} To avoid using too many notations we use the symbol
 $\lambda$   for a real number and the function induced. Apparently, there will be no confusion.
 \end{remark}
 \begin{proof}
   Let us rewrite (\ref{Glamd}) as
   \begin{equation*}
     \sup_{0\leq t\leq T}\sup_{x\in\RR} \int_{\RR} G_{t}(y)\frac{\lambda(y+x)}{\lambda(x)}dy\leq \sup_{0\leq t\leq T} \int_{\RR} G_{t}(y)\sup_{x\in\RR}\frac{\lambda(y+x)}{\lambda(x)}dy.
   \end{equation*}
We discuss the cases $\lambda\ge 0$ and $\lambda <0$ separately.
When $\lambda\geq0$, we have
   \[
    \sup_{x\in\RR} \frac{\lambda(y+x)}{\lambda(x)}\leq C_{\lambda} \sup_{x\in\RR} \lc\frac{1+|x|}{1+|x+y|}\rc^{2\lambda}\leq C_{\lambda}(1+|y|)^{2\lambda}\,.
   \]
On the other hand   when $\lambda<0$  we have
   \[
    \sup_{x\in\RR} \lc\frac{1+|x+y|^2}{1+|x|^2}\rc^{-\lambda}\leq C_{\lambda} \sup_{x\in\RR} \lc\frac{1+|x+y|}{1+|x|}\rc^{-2\lambda}\leq C_{\lambda} (1+|y|)^{-2\lambda}\,.
    \]
In both cases we see
\[
 \sup_{0\leq t\leq T} \int_{\RR} G_{t}(y)\sup_{x\in\RR}\frac{\lambda(y+x)}{\lambda(x)}dy\leq C_{\lambda} \sup_{t\in[0,T]}\int_{\RR} G_t(y)(1+|y|)^{2|\lambda|} dy< \infty.
   \]
This  finishes the proof.
 \end{proof}

 \begin{lemma}\label{TechLemma2}
Denote  $J(x):=\int_{0}^{\infty}e^{-{\eta^2}}\eta^{\beta}cos(x\eta)d\eta$,  where  $\beta>-1$.  We have
   \begin{equation}\label{J1bdd}
     |J(x)|\leq C_{\beta} \lc1\wedge \frac{1}{|x|^{\beta+1}}\rc.
   \end{equation}
 \end{lemma}
 \begin{proof}
Notice that this is  to estimate the decay rate of the Fourier transform of $e^{-{\eta^2}}\eta^{\beta}$   when $|x|$ is large.  Since $J(-x)=J(x)$  and since we are only concerned with the large $x$ behaviour we may  assume $x\geq 1$.
We split the integral $J(x)$ into two parts:
   \[
    J(x)=  \int_{0}^{s(x)}e^{-{\eta^2}}\eta^{\beta}\cos(x\eta)d\eta+\int_{s(x)}^{\infty}e^{-{\eta^2}}\eta^{\beta}\cos(x\eta)d\eta
    :=J_1(x)+J_2(x),
   \]
   where $s(x)>0$ is a function to be determined shortly.

  First, it is easy to see
   \[
    |J_1(x)|\leq \int_{0}^{s(x)} \eta^{\beta}d\eta\leq C_{\beta} [s(x)]^{\beta+1}.
   \]
   For $J_2(x)$,   an  integration  by parts implies
   \begin{align*}
     |J_2(x)| &= \Big|\int_{s(x)}^{\infty}e^{-{\eta^2}}\eta^{\beta}\cos(x\eta)d\eta\Big| \\
       &=\Big|\frac 1x \int_{s(x)}^{\infty}e^{-{\eta^2}}\eta^{\beta} d \sin(x\eta) \Big| \\
       &\leq C_{\beta}\frac {[s(x)]^{\beta}}{x} +\frac {C_{\beta}}{x} \Big|\int_{s(x)}^{\infty}\eta^{\beta-1}e^{-{\eta^2}}\sin(x\eta)d\eta \Big|\\
       &\qquad +\frac {C_{\beta}}{x} \Big|\int_{s(x)}^{\infty}\eta^{\beta+1}e^{-{\eta^2}}\sin(x\eta)d\eta \Big|.
   \end{align*}
Let $k=\lceil\beta\rceil $  denote
   the least integer greater than or equal to $\beta$. Continuing the above  application of     integration by parts another   $k $ times yields
   \begin{equation*}
     |J_2(x)| \leq \frac{C_{\beta}}{x^{k+1}}+ C_{\beta}\sum_{j=0}^{k}\frac{[s(x)]^{\beta-j}+[s(x)]^{\beta+j}}{x^{j+1}}. \\
   \end{equation*}
   Combining the estimates of $J_1(x)$ and $J_2(x)$ we have
  \begin{equation*}
     |J (x)| \leq C_{\beta}[s(x)]^{\beta+1}+ \frac{C_{\beta}}{x^{k+1}}+ C_{\beta}\sum_{j=0}^{k}\frac{[s(x)]^{\beta-j}+[s(x)]^{\beta+j}}{x^{j+1}}. \\
   \end{equation*}
The lemmas follows with the choice of  $s(x)=\frac 1x$.
 \end{proof}

 Let us associate  two increments  related to the Green
 function $G_t(x)$, given as follows.
The first one is a first order difference:
 \begin{equation}\label{FirDiff}
   D_t(x,h):=G_t(x+h)-G_t(x)\,.
 \end{equation}
Denote $D(x,h) =\sqrt{\pi} D_{1/4}(x,h)=e^{-(x+h)^2}-e^{-x^2}$ .  The second one is a second order difference:
 \begin{equation}\label{SecDiff}
      \Box_{t}(x,y,h) :=G_t(x+y+h)-G_t(x+y)-G_t(x+h)+G_t(x)\,.
 \end{equation}
 As above, we denote $\Box(x,y, h) =\sqrt{\pi} \Box_{1/4}(x,y, h)$:
\begin{equation}
 \Box(x,y,h) =e^{-(x+y+h)^2}-e^{-(x+h)^2}-e^{-(x+y)^2}+e^{-x^2}\,.
 \label{e.square_def}
\end{equation}

For these two increments, we have  the following identities which are  needed   later.
 \begin{lemma}\label{NGreen}
   For any $\alpha,\beta\in(0,1)$, we have
   \begin{equation}\label{FirDiffIneq}
    \int_{\RR^2}|D_t(x,h)|^2|h|^{-1-2\beta}dhdx = \frac{C_{\beta}}{t^{\frac 12+\beta}}  
   \end{equation}
   and
   \begin{equation}\label{SecDiffIneq}
     \int_{\RR^3}|\Box_t(x,y,h)|^2|h|^{-1-2\alpha}|y|^{-1-2\beta}dydhdx = \frac{C_{\alpha,\beta}}{t^{\frac 12+\alpha+\beta}}.
   \end{equation}
 \end{lemma}
 \begin{proof} With   change  of variables,  it suffices  to show
   \begin{equation}
     \begin{split}
        &\int_{\RR^2}|D(x,h)|^2|h|^{-1-2\beta}dhdx <\infty\,;  \\
        &\int_{\RR^3}|\Box(x,y,h)|^2|h|^{-1-2\alpha}|y|^{-1-2\beta}dydhdx < \infty.\label{e.Box}
     \end{split}
   \end{equation}
 The above  two inequalities  will  be derived from
   Plancherel's identity.
The Fourier  transforms with respect to the variable $x$ of
$D(x,h)$ and $\Box(x,y,h)$ are, respectively,
\[
        \hat{D}(\xi,h) =\cF[D(\cdot,h)](\xi)=\sqrt{\pi}e^{-\frac{\xi^2}{4}}[e^{ih\xi}-1]
        \]
        and
        \[
            \hat{\Box}(\xi,y,h) =\cF[\Box(\cdot,y,h)](\xi)=\sqrt{\pi} e^{-\frac{\xi^2}{4}}[e^{iy\xi}-1][e^{ih\xi}-1]
\,.
\]
Thus,  we have
   \begin{equation*}
     \begin{split}
        \int_{\RR}|D(x,h)|^2dx &=\int_{\RR}|\hat{D}(\xi,h)|^2 d\xi =4\pi\int_{\RR}e^{-\frac{\xi^2}{2}}[1-\cos(h\xi)]d\xi
     \end{split}
   \end{equation*}
   and \begin{equation*}
     \begin{split}
        \int_{\RR}|\Box(x,y,h)|^2dx &=\int_{\RR} |\hat{\Box}(\xi,y,h)|^2 d\xi =4\pi\int_{\RR}e^{-\frac{\xi^2}{2}}[1-\cos(h\xi)][1-\cos(y\xi)]d\xi.
     \end{split}
   \end{equation*}
By Fubini's theorem
 \begin{equation}
   \begin{split}
      \int_{\RR^2}|D(x,h)|^2|h|^{-1-2\beta}dhdx &=
        C\int_{\RR}  e^{-\frac{\xi^2}{2}} d\xi \int_\RR [1-\cos(h\xi)] |h|^{-1-2\beta}dh \\
        &= C\int_{\RR}  e^{-\frac{\xi^2}{2}} |\xi|^{2\beta}d\xi \int_\RR [1-\cos(h )] |h|^{-1-2\beta}dh <\infty \label{CosEq}
   \end{split}
 \end{equation}
since $     \int_{0}^{\infty} \frac{1-\cos(t)}{t^{\theta}}dt$ is finite
   for all $\theta\in(1,3)$ which requires  $ \alpha,\beta\in (0, 1)$.
   This proves the first inequality in \eqref{e.Box}.
   Same argument shows the second inequality in \eqref{e.Box} under the condition of the lemma.
 \end{proof}
 \begin{remark}
   In the rest of our paper, we shall use the lemma for
    $\alpha=\beta=\frac 12-H\in (0, 1/4)$.
 \end{remark}


 \begin{lemma}\label{LemDGcontra}
   For $D(x,h)$ and $D_t(x,h)$ defined in \eqref{FirDiff},  we have
   \begin{equation}\label{DGcontra}
     F(x):=\int_{\RR} |D(x,h)|^2 |h|^{2H-2} dh\leq C_{H} \lc1\wedge |x| ^{2H-2}\rc\,,
   \end{equation}
   and when $t>0$
   \begin{equation}
   F_t(x):=\int_{\RR} |D_t(x,h)|^2 |h|^{2H-2} dh\leq C_{H} \lc t^{  H-\frac32 } \wedge\frac{|x|^{2H-2}}{\sqrt{t}}\rc\,,  \label{DGcontra_t}
   \end{equation}
   where $0<H<\frac 12$.
 \end{lemma}
 \begin{proof}
   The   assertion  \eqref{DGcontra_t} is an easy consequence of     \eqref{DGcontra}  by   change of variables so we only need to provide  a  proof   for  \eqref{DGcontra}.

Recall that the Fourier transform of $  {D}(x,h)$  (as a function of $x$)
is
   \[
    \hat{D}(\eta,h)=\cF[D(\cdot,h)](\eta)=\sqrt{\pi}e^{-\frac{\eta^2}{4}}[e^{ih\eta}-1]\,.
   \]
By the inverse Fourier transformation  $D(x,h) $ can also be written as
\[
D(x,h)=\frac1{2\pi} \int_\RR  \hat{D}(\eta,h)e^{ix\eta} d\eta=\frac1{2
\sqrt{\pi}}\int_\RR e^{-\frac{\eta^2}{4}}[e^{ih\eta}-1]e^{ix\eta} d\eta\,.
\]
Therefore,  we  can write
\begin{align*}
    F(x)
&=C_H  \pi^2 \int_{\RR^2}   e^{-\frac{\eta_1^2+\eta_2^2} {4}}\int_\RR [e^{ih\eta_1}-1]
 \overline{[e^{ih\eta_2}-1]} |h|^{2H-2} dh \   e^{ix(\eta_1-\eta_2)}d\eta_1 d\eta_2 \\
    &=C_H   \int_{\RR^2} e^{-\frac{\eta^2_1+\eta^2_2}{4}}H(\eta_1,\eta_2)e^{ix(\eta_1-\eta_2)}d\eta_1d\eta_2\,,
   \end{align*}
where similar to \eqref{CosEq}, we have
 \begin{align}
 H(\eta_1,\eta_2)
 &=C_H  \int_\RR [e^{ih\eta_1}-1]  \overline{[e^{ih\eta_2}-1]} |h|^{2H-2} dh \nonumber\\
 &=C_H \lc |\eta_1|^{1-2H}+|\eta_2|^{1-2H}-|\eta_1-\eta_2|^{1-2H}\rc  \,.
 \label{HEqv}
 \end{align}
It is easy to see that $\sup_{x\in \RR} |F(x)|\le C<\infty$.   Now, we want to get the desired decay estimate when $x$ goes to infinity.  We have
\begin{align*}
     F(x)
     &\leq C_H  \left|\int_{\RR^2} e^{-\frac{\eta^2_1+ \eta^2_2}{4}}|\eta_2|^{1-2H} e^{ix(\eta_1-\eta_2)}  d\eta_1d\eta_2\right|
     \nonumber\\
     &\qquad  +C_H \left|\int_{\RR^2} e^{-\frac{\eta^2_1+\eta^2_2}{4}}|\eta_1-\eta_2|^{1-2H}e^{ix(\eta_1-\eta_2)}d\eta_1d\eta_2\right|
     \\
    &\leq C_H e^{-x^2}\lt|\int_{\RR} e^{-\frac{\eta_2^2}{4}}|\eta_2|^{1-2H}e^{-ix\eta_2}d\eta_2\rt|\nonumber\\
    &\qquad
      + C_H \lt|\int_{\RR} \left[ |\eta|^{1-2H}e^{-ix\eta}  \int_\RR e^{-\frac{|\eta_2|^2+|\eta_2+\eta|^2}{4}}d\eta_2\right]d\eta \rt| \\
    &\leq C_H e^{-x^2}\lt|\int_{\RR_+} e^{-\frac{\eta^2}{4}}|\eta|^{1-2H}cos(x\eta)d\eta\rt|
      +C_H \lt|\int_{\RR_+}e^{-\frac{\eta^2}{8}}|\eta|^{1-2H}cos(x\eta)d\eta\rt|
   \end{align*}
since
\[
\int_\RR e^{-\frac{|\eta_2|^2+|\eta_2+\eta|^2}{4}}d\eta_2=C e^{-\frac{|\eta|^2}{8}}\,.
\]
 Now the inequality    \eqref{DGcontra}  follows from    Lemma \ref{TechLemma2}.
 \end{proof}

 \begin{lemma}\label{LemBGcontra}
   Recall that $\Box_t(x,y,h)$ and $\Box (x,y,h)$ are defined  by  \eqref{SecDiff}
   and \eqref{e.square_def}. We have
   \begin{equation}\label{BGcontra}
     F(x):=\int_{\RR^2} |\Box(x,y,h)|^2 |h|^{2H-2}|y|^{2H-2} dydh\leq C_{H} \lc 1\wedge |x|  ^{2H-2}\rc\,.
   \end{equation}
Moreover,  for any $t>0$ we have
   \begin{equation}
   F_t(x):=\int_{\RR^2} |\Box_t(x,y,h)|^2 |h|^{2H-2}|y|^{2H-2} dydh\leq C_{H} \lc t^{2H-2}\wedge\frac{|x|^{2H-2}}{t^{1-H}}\rc\,.
   \end{equation}
 \end{lemma}
 \begin{proof} As in the proof of  Lemma
 \ref{LemDGcontra} we only need
 to prove \eqref{BGcontra} and last inequality can be derived from \eqref{BGcontra}  by a change of variable.

   The proof of \eqref{BGcontra} is similar to that of Lemma \ref{LemDGcontra}. Recall the Fourier transform of $\Box(x,y,h)$ as a function of $x$:
\[
\hat{\Box}(\eta,y,h) =\cF[\Box(\cdot,y,h)](\eta)=\sqrt{\pi} e^{-\frac{\eta^2}{4}}[e^{iy\eta}-1][e^{ih\eta}-1] \,.
\]
This means
\[
 {\Box}(x,y,h) = \sqrt{\pi} \int_\RR e^{-\frac{\eta^2}{4}}[e^{iy\eta}-1][e^{ih\eta}-1] e^{ix\eta} d\eta  \,.
\]
Thus, we have
\begin{align}
F(x)&=\int_{\RR^4}
e^{-\frac{\eta_1^2+\eta_2^2}{4}}[e^{iy \eta_1 }-1][e^{ih \eta_1}-1]\cdot\overline{[e^{iy \eta_2 }-1]}\nonumber\\
&\quad \overline{[e^{ih \eta_2}-1]} |h|^{2H-2}|y|^{2H-2} e^{ix(\eta_1-\eta_2)}   dydh d\eta_1 d\eta_2
\nonumber\\
&=2\pi^2 \int_{\RR^2} e^{-\frac{\eta^2_1+\eta^2_2}{4}}H^2(\eta_1,\eta_2)e^{ix(\eta_1-\eta_2)}d\eta_1d\eta_2\,,
\end{align}
where $H(\eta_1,\eta_2)$ is defined in \eqref{HEqv} or
   \begin{align*}
   {H^2(\eta_1,\eta_2)}= C_H^2\bigg(  & |\eta_1|^{2-4H}+|\eta_2|^{2-4H}+|\eta_1|^{1-2H}|\eta_2|^{1-2H}+|\eta_1-\eta_2|^{2-4H} \\
     & -|\eta_1|^{1-2H}|\eta_1-\eta_2|^{1-2H}-|\eta_2|^{1-2H}|\eta_1-\eta_2|^{1-2H}\bigg)\,.
   \end{align*}
   It is easy to see that  $\sup_{x\in \RR} |F(x)|\le C<\infty$.
Now we want to show the desired decay rate as $x\rightarrow \infty$. By the symmetry $F(-x)=F(x)$, we can and will assume $x\ge 1$.
The argument in the proof of    Lemma \ref{LemDGcontra} can be used to obtain the desired bound for each of the above terms  {except} the terms $|\eta_1-\eta_2|^{2-4H}$ and $|\eta_1|^{1-2H}|\eta_1-\eta_2|^{1-2H}$ (and $|\eta_2|^{1-2H}|\eta_1-\eta_2|^{1-2H}$,    which can be handled analogously).

   For term $|\eta_1-\eta_2|^{2-4H}$, letting $\xi_1=\eta_1-\eta_2$ and $\xi_2=\eta_1+\eta_2$ implies
   \begin{align*}
     &\int_{\RR^2} e^{-\frac{\eta^2_1+\eta^2_2}{4}} |\eta_1-\eta_2|^{2-4H} e^{ix(\eta_1-\eta_2)}d\eta_1 d\eta_2 \\
    =& C \int_{\RR^2} e^{-\frac{\xi^2_1+\xi^2_2}{8}} |\xi_1|^{2-4H}e^{ix\xi_1}d\xi_1d\xi_2
    =C \int_{\RR_{+}}e^{-\frac{\xi^2}{8}}|\xi|^{2-4H} \cos(x\xi)d\xi.
   \end{align*}
   Then using Lemma \ref{TechLemma2}, we see that  this term is bounded by $1\wedge |x|^{4H-3}\lesssim 1\wedge |x|^{2H-2}$ for $\frac 14<H<\frac 12$.

   In order to deal with the second term
   $|\eta_1|^{1-2H}|\eta_1-\eta_2|^{1-2H}$, we make  the substitution  $\xi =\eta_1$ and $\eta=\frac12(\eta_1-\eta_2)$
   to obtain
   \begin{align*}
     J(x):=&\int_{\RR^2} e^{-\frac{\eta^2_1+\eta^2_2}{4}} |\eta_1|^{1-2H}|\eta_1-\eta_2|^{1-2H} e^{ix(\eta_1-\eta_2)}d\eta_1 d\eta_2 \\
    =&C \int_{\RR^2} \exp\lc-\frac{(\xi-\eta)^2}{2}\rc \exp\lc-\frac{\eta^2}{2}\rc |\xi|^{1-2H}|\eta|^{1-2H} e^{i2x\eta} d\xi d\eta\,.
   \end{align*}
Denote
   \[
    E(\eta):=\int_{\RR}  \exp\lc-\frac{(\xi-\eta)^2}{2}\rc |\xi|^{1-2H}d\xi\,.
   \]
We need to show a similar inequality to  that  in Lemma \ref{TechLemma2}:
   \[
    |J(x)| =\lt|\int_{0}^{\infty}e^{-\frac{\eta^2}{2}}\eta^{1-2H}E(\eta)\cos(2x\eta)d\eta\rt| \leq C_H \lc1\wedge |x|^{2H-2}\rc.
   \]
First,   we  observe that $|E(\eta)|\leq C_H (1+|\eta|^{1-2H})$   and
    both $|E'(\eta)|$ and $|E''(\eta)|$ can be bounded by a multiple of
   \[
    \int_{\RR} \exp\lc-\frac{(\xi-\eta)^2}{4}\rc |\xi|^{1-2H}d\xi\leq C_H \lc1+|\eta|^{1-2H}\rc\,.
   \]

We only need to care the case when $x$ is large. Let us split $ J(x) $ into two parts of which one integrates from $0$ to $s(x)$, denoted by $J_1(x)$,
 and the other integrates from $s(x)$ to infinity, denoted by $J_2(x)$,   such that   $s(x)\rightarrow 0$ as $x$ goes to infinity  and  whose precise form will be given   later. For the first part
   \[
    |J_1(x)|\leq [s(x)]^{1-2H}\int_{0}^{s(x)}|E(\eta)| d\eta \leq C_H \lc [s(x)]^{2-2H}+[s(x)]^{3-4H} \rc.
   \]
   For $J_2(x)$, an integration  by parts  yields
   \begin{align*}
     &|J_2(x)| = \Big|\int_{s(x)}^{\infty}e^{-{\frac{\eta^2}{2}}}\eta^{1-2H}E(\eta)\cos(2x\eta)d\eta\Big| \\
       =&C\Big|\frac 1x \int_{s(x)}^{\infty}e^{-{\frac{\eta^2}{2}}}\eta^{1-2H}E(\eta) d \sin(2x\eta) \Big| \\
       \leq& C_H\frac {[s(x)]^{1-2H}}{x} e^{-{\frac{[s(x)]^2}{2}}}|E(s(x))|+\frac {C_H}{x} \Big|\int_{s(x)}^{\infty}\eta^{-2H}e^{-{\frac{\eta^2}{2}}}\sin(2x\eta)E(\eta)d\eta \Big|\\
       +&\frac {C_H}{x} \Big|\int_{s(x)}^{\infty}\eta^{2-2H}e^{-{\frac{\eta^2}{2}}}\sin(2x\eta)E(\eta)d\eta \Big| +\frac {C_H}{x} \Big|\int_{s(x)}^{\infty}\eta^{1-2H}e^{-{\frac{\eta^2}{2}}}\sin(2x\eta)E'(\eta)d\eta \Big|\\
       =:& J_{21}+J_{22} +J_{23}+J_{24} \,.
   \end{align*}
The first term is  bounded   by
\[
J_{21}(x)\leq C_H \frac 1x [s(x)]^{1-2H}
\,.
\]
As  for $J_{22}(x)$ an integration by parts yields
   \begin{align*}
    J_{22}(x):=&\frac 1x \Big|\int_{s(x)}^{\infty}\eta^{-2H}e^{-{\frac{\eta^2}{2}}}\sin(2x\eta)E(\eta)d\eta \Big| \\
     \leq& C\frac{E(s(x))}{x^2} [s(x)]^{-2H}+\frac{C}{x^2}\int_{s(x)}^{\infty} \lt|\frac{d}{d \eta}\lt[\eta^{-2H}E(\eta)e^{-\frac{\eta^2}{2}}\rt]\rt| d\eta\\
     \leq& \frac{C_H}{x^2} [s(x)]^{-2H}+\frac{C_H}{x^2} [s(x)]^{1-4H}+\frac{C_H}{x^2}\,.
   \end{align*}
In the same way we can bound     $J_{23}(x)$ as follows.
   \begin{align*}
    J_{23}(x):=& \frac 1x \Big|\int_{s(x)}^{\infty}\eta^{2-2H}e^{-{\frac{\eta^2}{2}}}\sin(2x\eta)E(\eta)d\eta \Big| \\
     \leq& C\frac{E(s(x))}{x^2}[s(x)]^{2-2H} +\frac{C}{x^2}\int_{s(x)}^{\infty} \lt|\frac{d}{d \eta}\lt[\eta^{2-2H}E(\eta)e^{-{\frac{\eta^2}{2}}}\rt]\rt|d\eta\\
     \leq& \frac{C_H}{x^2} [s(x)]^{2-2H}+\frac{C_H}{x^2} [s(x)]^{3-4H}+\frac{C_H}{x^2}\,.
   \end{align*}
The term  $J_{24}(x)$ satisfies
   \begin{align*}
    J_{24}(x):=& \frac 1x \Big|\int_{s(x)}^{\infty}\eta^{1-2H}e^{-{\frac{\eta^2}{2}}}\sin(x\eta)E'(\eta)d\eta \Big| \\
     \leq& C\frac{E'(s(x))}{x^2}[s(x)]^{1-2H} +\frac{C}{x^2}\int_{s(x)}^{\infty} \lt|\frac{d}{d \eta}\lt[\eta^{1-2H}E'(\eta)e^{-{\frac{\eta^2}{2}}}\rt]\rt|d\eta\\
     \leq& \frac{C_H}{x^2}[s(x)]^{1-2H}+\frac{C_H}{x^2}[s(x)]^{2-4H}+\frac{C_H}{x^2}\,.
   \end{align*}
Noticing that $\frac 14<H<\frac 12$, and   taking $s(x)=\frac 1x$   imply our result.
 \end{proof}

 \begin{lemma}\label{LemDGLcontra}
Denote  $\lambda(x)=\frac{1}{(1+|x|^2)^{1-H}}$  and recall $D_t(x,h)$ defined by \eqref{FirDiff} and  $\Box_t(x,y,h)$ defined by \eqref{SecDiff}.  We have
   \begin{equation}\label{DGLcontra}
    \begin{split}
       &\int_{\RR^2} |D_t(x,h)|^2|h|^{2H-2} \lambda(z-x)dxdh\leq C_{T,H} t^{H-1} \lambda(z),\\
       &\int_{\RR^3} |\Box_t(x,y,h)|^2|h|^{2H-2}|y|^{2H-2} \lambda(z-x)dxdydh\leq C_{T,H} t^{2H-\frac 32} \lambda(z).
    \end{split}
   \end{equation}
 \end{lemma}
 \begin{proof}
   Set
   \[
    R(x,z)=\frac{\lambda(z-x)}{\lambda(z)}\simeq \lc\frac{1+|z|}{1+|x-z|} \rc^{2-2H}\,,
   \]
where and throughout the paper for two functions $f$ and $g$, notation $f\simeq g$ means   that there exist two positive constants $c_H$ and $C_H$ such that
 $c_H g\leq f\leq C_H g$.
By Lemma \ref{NGreen}, we have by change of variables $x\to {x}{\sqrt{t}}$, $h\to   {h}{\sqrt{t}}$ and $z\to {z}{\sqrt{t}}$
   \begin{align}
      &\int_{\RR^2} |D_t(x,h)|^2|h|^{2H-2} R(x,z)dxdh \nonumber \\
     \leq&C_H t^{H-1} \int_{\RR^2} |D(x,h)|^2 |h|^{2H-2}R(\sqrt{t}x,\sqrt{t}z)dxdh  \nonumber \\
     \leq& C_H t^{H-1} \int_{\RR} \lc 1\wedge |x|^{2H-2} \rc R(\sqrt{t}x,\sqrt{t}z)dx.\label{e.3.16}
   \end{align}
   Similarly,  making substitutions    $x\to {x}{\sqrt{t}}$, $y\to  {y}{\sqrt{t}}$, $h\to   {h}{\sqrt{t}}$ and $z\to {z}{\sqrt{t}}$  we can get rid of the $t$ in $\Box_t$: 
   \begin{align}
      &\int_{\RR^3} |\Box_t(x,y,h)|^2|h|^{2H-2}|y|^{2H-2} R(x,z)dxdydh
       \nonumber  \\
     =&C_H t^{2H-\frac 32} \int_{\RR^3} |\Box(x,y,h)|^2 |h|^{2H-2}|y|^{2H-2} R(\sqrt{t}x,\sqrt{t}z)dxdydh  \nonumber \\
     \leq&C_H t^{2H-\frac 32} \int_{\RR} \lc 1\wedge |x|^{2H-2} \rc R(\sqrt{t}x,\sqrt{t}z)dx.\label{e.3.17}
   \end{align}
Notice that the  above change  of variable with respect to  $z$ is not essential because we will take its  supremum   over  $\RR$. But it will be  convenient for us  to split the intervals. 
From \eqref{e.3.16} and \eqref{e.3.17} to show our lemma it is sufficient to show
\begin{equation}
    \sup_{t\in[0,T]}\sup_{z\in \RR}\int_{\RR} \lc 1\wedge |x|^{2H-2} \rc R(\sqrt{t}x,\sqrt{t}z)dx <\infty.\label{e.3.18}
   \end{equation}
Notice that we assume that $t\in [0, T]$ is bounded.
If $z$ is bounded,  then $R(\sqrt{t}x,\sqrt{t}z)$ is also bounded.
Then, we have
\begin{equation}
 \sup_{t\in[0,T]}\sup_{|z|\le 2}\int_{\RR} \lc 1\wedge |x|^{2H-2} \rc R(\sqrt{t}x,\sqrt{t}z)dx\ \leq \
 C_{T,H} \int_{\RR}   1\wedge |x|^{2H-2}  dx<\infty.
 \label{e.3.19}
\end{equation}
 This means that  we only need to consider the case $|z|\ge 2$. Due to the symmetry $R(-\sqrt{t}x,-\sqrt{t}z)=R(\sqrt{t}x,\sqrt{t}z)$,    we can assume $z\ge 2$.

Next we  split  the integral    domain into two parts.

\noindent (i) The domain $ x \le z/2$ or $ x \ge 2z$.
 On this domain $R(\sqrt{t}x,\sqrt{t}z)$ is bounded. Thus
 \begin{equation}
\sup_{t\in[0,T]}\sup_{|z|\ge 1}\int_{ \left\{x \le z/2 \atop\  x \ge 2z\right\} } \lc 1\wedge |x|^{2H-2} \rc R(\sqrt{t}x,\sqrt{t}z)dx\ \leq \
 C_{T } \int_{\RR}   1\wedge |x|^{2H-2}  dx<\infty.
 \label{e.3.20}
 \end{equation}

\noindent (ii) The domain $  z/2\le  x \le 2z$.
 On this domain  we have $x\ge z/2\ge (z+1)/3\ge 1$  and then
 \[
 1\wedge |x|^{2H-2}\le |x|^{2H-2}\le \frac{3^{2-2H}}{(1+z)^{2-2H}}\,.
 \]
 Thus,
     \begin{align*}
     I:= &\int_{  \frac z2<x<2z } (1\wedge|x|^{2H-2}) R(\sqrt{t}x,\sqrt{t}z)dx \\
     \leq& C_H \lc\frac{1+\sqrt{t}z}{1+z}\rc^{2-2H} \int_{0}^{2z}\frac{1}{(1+\sqrt{t}|x-z|)^{2-2H}} dx  
     \,. 
   \end{align*}
By the symmetry    of the above integrand   we know that 
the integrals  $\int_0^z$ and 
$\int_z^{2z}$ are the same.  Hence, we have 
  \begin{align*}
     I 
     \leq& C_H \frac{ 1+(\sqrt{t}z)^{2-2H}}{1+z^{2-2H}}
     \int_{z}^{2z}\frac{1}{ \left(1+\sqrt{t}|x-z|\right)^{2-2H}} dx\\
      = & C_H \frac{ 1+(\sqrt{t}z)^{2-2H}}{\sqrt t\left(1+z^{2-2H}\right)}
   \left[1 - (1+\sqrt{t}z)^{2H -1} \right]\\
    \leq &C_H T^{\frac12-H}\frac{ 1+(\sqrt{t}z)^{2-2H}}{ \sqrt t z\left( 1+(\sqrt t z)^{1-2H} \right)  }
   \left[1 - (1+\sqrt{t}z)^{2H -1} \right]\,.
   \end{align*}
Consider now the function
\[
f(u)=\frac{ 1+u^{2-2H}}{ u(1+u^{1-2H})  }
   \left[1 - (1+u)^{2H -1} \right]\,, \quad u>0\,.
   \]
This is a continuous function on   $(0, \infty)$. When $u\rightarrow 0$
and when $u\rightarrow \infty$ we have
\[
\lim_{u\rightarrow 0+}f(u)=  1-2H\,,\quad \lim_{u\rightarrow \infty}f(u)=1\,.
\]
Thus,  $f(u)$ is bounded on $(0, \infty)$  and this   in turn proves
\begin{equation}
\sup_{t\in[0,T]}\sup_{z\ge 1}\int_{z/2\le  x \le   2z }
 \lc 1\wedge |x|^{2H-2} \rc R(\sqrt{t}x,\sqrt{t}z)dx\  <\infty.
 \label{e.3.21}
  \end{equation}
Combining    \eqref{e.3.19}-\eqref{e.3.20} together with our above symmetry argument  we prove  \eqref{e.3.18} and hence we complete   the proof of the lemma.
 \end{proof}
 \begin{remark}\label{RateLam}
   From this lemma,  we see why we choose 
    the above    decay rate for  our weight
   function. If we consider $\lambda(x)= (1+|x|^2)^{-{\lambda}}$ with $\lambda>1-H$, then for $|z|$ sufficiently large  one has
   \[
    \int_{\RR} \lc 1\wedge |x|^{2H-2} \rc R(x,z)dx\gtrsim \int_{|x-z|<1} |x|^{2H-2}R(x,z)dx\gtrsim \frac{(1+|z|)^{{\lambda}}}{|z|^{2-2H}},
   \]
   which diverges as $|z|\rightarrow \infty$. This elementary fact suggests  us   that $\lambda$ must be in
   $(\frac 12, 1-H]$, and it is obvious $L^p_{\lambda}(\RR)$ is the largest space when $\lambda=1-H$.
 \end{remark}

\section{Additive noise}\label{s.additive}
When the diffusion coefficient  $\sigma(t,x,u) =1$
(or  a general constant), the noise  is
additive and the  solution    to  
\eqref{SHE} can be written explicitly as
\begin{equation}
  u(t,x)= \int_\RR
  G_{t }(x-y) u_0(y) dy+\int_{0}^{t}\int_{\RR} G_{t-s}(x-y) W(ds,dy)\,, 
\end{equation}
where $G_t(x)=\frac{1}{\sqrt{4\pi t}} \exp\lt(-\frac{x^2}{4t}\rt)$ is the heat kernel. To focus on the stochastic part we assume $u_0=0$. Thus,  the  resulted   solution is written as
\begin{equation}\label{LinearMildSol}
  u_\affine (t,x)=  \int_{0}^{t}\int_{\RR} G_{t-s}(x-y) W(ds,dy)\,.
\end{equation}
This solution  $u_\affine (t,x)$ defines a (symmetric) centered Gaussian process.  We shall study how it grows as the parameters $t$ and $x$ go to infinity.  It is expected  that $u_\affine (t,x)$
is H\"older continuous in $t$ and $x$.
More precisely,  for any positive constants $\gamma<H$, $T, L\in (0, \infty)$, there is a constant $C_{T, L, \gamma}$, depending only on $T$,  $L$ and $\gamma$,
such that
\[
\sup_{0\le s, t\le T\,, |x|, |y|\le L}|u_\affine (s,x)-u_\affine (t,y)|\le C_{T, L, \gamma}\left(|t-s|^{\gamma/2}+|x-y|^{\gamma}\right)\,.
\]
We want to consider   the H\"older continuity of $u_\affine (t,x)$ on the whole space $\RR$.  Namely, we  want to know how the constant
$C_{T, L, \gamma}$ grows as  $T$ and $L$ go to infinity (for any fixed $\gamma$).

\subsection{Majorizing measure theorem}
To find the sharp bound for $C_{T, L, \gamma}$ we shall  utilize   Talagrand's majorizing measure theorem which we  recall   below.
%
%
 \begin{theorem} {\rm (Majorizing Measure Theorem,  see  e.g.
   \cite[Theorem 2.4.2]{Talagrand2014}).} \label{Talagrand}
 Let $T$ be a given set and 	let $\{X_t,t\in T\}$ be a centered Gaussian process indexed by $T$.  Denote
 	 $d(t,s)=(\EE|X_t-X_s|^2)^{\frac 12}$,
 	     the associated natural metric on $T$.
%
  Then
 	\begin{equation}
 	\EE\Blk \sup_{t\in T} X_t \Brk\asymp \gamma_2(T,d):=\inf_\cA \sup_{t\in T} \sum_{n\geq 0}2^{n/2} \diam(A_n(t))\,,
 	\end{equation}
 	where   the infimum is taken over all increasing sequence $\cA:=\{\cA_n, n=1, 2, \cdots\}  $ of partitions of $T$ such that $\#\cA_n  \leq 2^{2^n}$ ($\#A$ denotes the number of elements in the  set $A$),   $A_n(t)$ denotes the unique element of $\cA_n$ that contains $t$, and $\diam(A_n(t))$ is the diameter (with respect to the natural   distance $d$)  of $A_n(t)$.
 \end{theorem}
This theorem provides a powerful general principle for the study of the supremum of Gaussian process.

\begin{remark}
	The natural metric $d(t,s)$ is actually only a pseudo-metric because $d(t,s)=0$ does not necessarily imply $t=s$   (e.g. $X_t\equiv1$).  It is also   { called} the  canonical metric.
\end{remark}
It is more convenient for us  to use the following theorem to obtain the lower bound.
	\begin{theorem}  {\rm (Sudakov minoration  theorem,  see   e.g.  \cite[Lemma 2.4.2]{Talagrand2014}).} \label{Sudakov}
		Let $\{X_{t_i},i=1,\cdots,L\}$ be  { a} centered Gaussian family with natural distance $d$ and assume
		\[
		 \forall p,q \leq L,~p\neq q \Rightarrow d(t_p,t_q)\geq \delta.
		\]
		Then,    we have
		\begin{equation}
		\EE\Blc\sup_{1\leq i\leq L} X_{t_i} \Brc \geq \frac{\delta}{C} \sqrt{\log_2(L)},
		\end{equation}
		where $C$ is a    universal constant.
	\end{theorem}
The following ``concentration of measure" type theorem  allows  us to obtain deviation inequalities for the supremum of  { a} Gaussian family.
 \begin{theorem} {\rm (Borell, see   e.g.   \cite[Theorem 2.1]{adler}).}   \label{Borell} { Let  $\{X_t,t\in T\}$  be a centered separable   Gaussian process   on some topological index set $T$
 with almost surely bounded sample paths.}   Then    $\EE\Blc\sup_{t\in  T} X_t\Brc<\infty$,  and for all $\lambda>0$
	\begin{equation}
	\bP\lt\{\lt|\sup_{t\in T} X_t-\EE\Blc\sup_{t\in T} X_t\Brc\rt|>\lambda\rt\}\leq 2\exp\lc-\frac {\lambda^2}{2\sigma_T^2}\rc,
	\end{equation}
	where $\sigma_T^2:=\sup_{t\in T}\EE(X_t^2)$.
 \end{theorem}


We have the following observation which can be deduced immediately from    \cite[Lemma 2.2.1]{Talagrand2014}. This simple fact tells us $\EE[\sup_{t\in T} |X_{t}|]\simeq \EE[\sup_{t\in T} X_{t} ]$. So, 
 we only need to consider $\EE[\sup_{t\in T} X_{t} ]$.
\begin{lemma}\label{equiv}
  If the process $\{X_{t},t\in T\}$ is symmetric,   then we have
  \begin{equation}
    \EE\big[\sup_{t\in T} |X_{t}|\big]\leqslant 2\EE\big[\sup_{t\in T} X_{t} \big]+\inf_{t_0\in T}\EE\big[|X_{t_0}|\big]\,.
  \end{equation}
\end{lemma}

\subsection{Asymptotics of the  Gaussian solution}
For the mild solution $u_\affine (t,x)$ to \eqref{SHE} with additive noise (e.g. $\sigma(t,x,u)=1$),  defined by \eqref{LinearMildSol},  we  shall first obtain the sharp upper and lower bounds for its associated natural metric:
\begin{equation}\label{NaturalMetric}
   d_1((t,x),(s,y))= \sqrt{\EE|u_\affine (t,x)-u
   _\affine (s,y)|^2}\,,
\end{equation}

The following lemma  gives a sharp bounds for   this induced natural metric for the Gaussian solution $u_\affine(t,x)$.
 \begin{lemma}\label{LemMetricProp}
Let  $d_1((t,x),(s,y))$  be    the natural metric  defined by  \eqref{NaturalMetric}. Then, there are   positive  constants $c_{H},C_H $  such that
   \begin{equation}\label{LUofNaturalMetric}
   \begin{split}
      c_{H}(|x-y|^H\wedge (t\wedge s)^{\frac H2}+|t-s|^{\frac H2})&\leq d_1((t,x),(s,y)) \\
        & \leq C_{H}(|x-y|^H\wedge (t\wedge s)^{\frac H2}+|t-s|^{\frac H2})
   \end{split}
   \end{equation}
 for any  $(t,x),(s,y)\in \RR_+\times\RR$.
 \end{lemma}
 \begin{remark} The above property of  the natural metric can also be written as
   \begin{equation}\label{EqvMetric}
   	d_1((t,x),(s,y))\asymp d_{1,H}((t,x), (s, y)):=|x-y|^H\wedge (t\wedge s)^{\frac H2}+|t-s|^{\frac H2}.
   \end{equation}
   $d_{1,H}((t,x), (s, y))$ is no longer a distance but it is very convenient for us to obtain   the  desired results.
 \end{remark}

\begin{proof} Without loss of generality, let us assume $t>s$. Plancherel's identity and the independence of the    stochastic integrals over the time intervals $[0, s]$ and $[s, t]$  give
\begin{align}
    d_1^2((t,x),&(s,y))= \EE|u_\affine (t,x)-u_\affine (s,y)|^2 \nonumber\\
    =& \EE\lt| \int_{0}^{s}\int_{\RR}[G_{t-r}(x-z)-G_{s-r}(y-z)] W(dr,dz) \rt|^2  \nonumber\\
    &+ \EE\lt| \int_{s}^{t}\int_{\RR}G_{t-r}(x-z) W(dr,dz) \rt|^2 \nonumber\\
    =& \int_{\RR_+} [1-\exp(-2s\xi^2)][1+\exp(-2(t-s)\xi^2) \label{NMPlancherel}\\
    &\qquad -2\exp(-(t-s)\xi^2)\cos(|x-y|\xi)]\cdot \xi^{-1-2H}d\xi
    + 2^{H-1}\kappa_H (t-s)^{H}\,, \nonumber
\end{align}
where $\kappa_H=H^{-1} \Gamma(1-H)$ is a positive constant.
We start to obtain  the upper bound of \eqref{LUofNaturalMetric}.  The triangular inequality gives
	\begin{equation}\label{UConstant}
     d_1 ((t,x), (s,y))\le   d_1 ((t,x), (s,x))+  d_1 ((s,x), (s,y))\,.
	\end{equation}

	Let us deal with the two terms on the right hand side of the above inequality  separately. For the first term, Plancherel's identity
    \eqref{NMPlancherel} implies
	\begin{equation*}
		\begin{split}
		d_1^2((t,x) ,(s,x))
				&=\kappa_H\lk 2^{H-1}t^{H}+2^{H-1}s^{H}-(t+s)^{H}\rk+(2^{H-1}+1)\kappa_H(t-s)^{H} \\
				&\leq C_{H} (t-s)^{H}\,,
		\end{split}
	\end{equation*}
	because $2^{H-1}t^{H}+2^{H-1}s^{H}-(t+s)^{H}\leq 0$ when $t\geq s$. Again from \eqref{NMPlancherel}, the second term
	on the right hand side of \eqref{UConstant}
	is given by
    \begin{equation*}
      \begin{split}
         d_1^2((s,x),(s,y))
           &=\int_{0}^{s}\int_{\RR} \exp[-2(s-r)\xi^2]\cdot|\xi|^{1-2H}|1-\cos(\xi|x-y|)| d\xi dr \\
           &=C_H|x-y|^{2H}\int_{\RR_+}\lt[1-\exp\lc-\frac{2s\xi^2}{|x-y|^2}\rc\rt]\cdot \xi^{-1-2H}[1-\cos(\xi)] d\xi\,,
      \end{split}
    \end{equation*}
    which can be controlled by $C_H |x-y|^{2H}$.  On the other hand, we have
    \begin{align*}
      d_1^2((s,x),(s,y))
      =&\EE[|u_\affine (s,x)-u_\affine (s,y)|^2]\\
      \leq& 2\blc\EE[|u_\affine (s,x)|^2]+\EE[|u_\affine (s,y)|^2]\brc \leq C_H s^H\,.
    \end{align*}
    Thus,   the quantity of $d_1^2((s,x),(s,y))$ is bounded by  the minimum of $C_H |x-y|^{2H}$ and $C_H s^H$. We can summarize the above argument as
    \begin{equation}\label{eq.UMetric2}
      d_1((t,x),(s,y)) \leq C_{H}(|x-y|^H\wedge s^{\frac H2}+(t-s)^{\frac H2}) \,,
    \end{equation}
    which is the upper bound part of \eqref{LUofNaturalMetric}.

    Now we turn to  show the lower bound part of \eqref{LUofNaturalMetric}.  From Plancherel's identity    it is sufficient to bound   the   first summand in \eqref{NMPlancherel}  from below by  $c_H (|x-y|^H\wedge s^{\frac H2})$ for some constant $c_H>0$. We denote this first summand by $I$:
    \begin{align}
     I := &  \int_{\RR} [1-\exp( -2s\xi^2)][1+\exp(-2(t-s)\xi^2) \nonumber\\
      &-2\exp(-(t-s)\xi^2)\cos(|x-y|\xi)]|\xi|^{1-2H}d\xi \nonumber\\
       =&c |x-y|^{2H}\int_{\RR_+}\lk 1-\exp\lc-\frac{2s\xi^2}{|x-y|^2}\rc\rk\cdot \xi^{-1-2H} \label{eq.M2.1} \\
      &\qquad\qquad\qquad \cdot \lk 1-\exp\lc-\frac{(t-s)\xi^2}{|x-y|^2}\rc\cos(\xi) \rk^2 d\xi \,. \nonumber
    \end{align}
    To bound it from below, we divide our argument into  two cases:
    \[
     |x-y|>\sqrt{s}\qquad{\rm and}\qquad |x-y|\leq\sqrt{s}\,.
    \] 
    When $|x-y|\leq\sqrt{s}$, we can bound \eqref{eq.M2.1} from 
    below by
	\begin{align}
      I   &\geq c_H |x-y|^{2H} \sum_{n=1}^{\infty} \int_{2n\pi+\frac{\pi}{2}}^{2n\pi+\frac{3\pi}{2}}\blk1-\exp\lc-2\xi^2\rc\brk\cdot \xi^{-1-2H} d\xi \nonumber\\
        & \geq c_H |x-y|^{2H}\,, \label{eq.LM2.1}
	\end{align}
	since $1-\exp(-2s\xi^2/|x-y|^2)\geq 1-\exp(-2\xi^2)$ by the assumption and $\cos(\xi)$ is negative on the intervals
    $\bigcup_{n=1}^{\infty}[2n\pi+\frac{\pi}{2},2n\pi+\frac{3\pi}{2}]$.

    The case $|x-y|>\sqrt{s}$ is a little bit more involved. Denote
	\[
	 n_0:=\inf\left\{n\in \bN_0:2n\pi+\frac{\pi}{2}\geq \sqrt{\frac{-\ln(1-c^*)}{2s}}|x-y|\right\}
	\]
    with the choice   $c^*=1-\exp(-\pi^2/2)$ such that
	\[
    \sqrt{\frac{-\ln(1-c^*)}{2s}}|x-y|\geq \frac{\pi}{2}\,.
    \]
    It is then easy to see that $n_0$ is a  well defined finite positive integer. This way,  we   have the lower bound   for
    \eqref{eq.M2.1}:
	\begin{equation*}
	\begin{split}
    I\ge &\sum_{n=n_0}|x-y|^{2H}\int_{2n\pi+\frac{\pi}{2}}^{2n\pi+\frac{3\pi}{2}}\lt[1-\exp\lc-\frac{2s\xi^2}{|x-y|^2}\rc\rt]
    \cdot\xi^{-1-2H}d\xi \\
	 \geq& c^*|x-y|^{2H}\sum_{n\geq n_0}\int_{2n\pi+\frac{\pi}{2}}^{2n\pi+\frac{3\pi}{2}} \xi^{-1-2H} d\xi\geq \frac{c^*}{2}|x-y|^{2H}
    \int_{2n_0\pi+\frac{\pi}{2}}^{\infty} \xi^{-1-2H} d\xi,
	\end{split}
	\end{equation*}
    where 	the last inequality follows from the fact that   $\xi^{-1-2H}$ is a decreasing function  on $(0, \infty)$.
    From the definition of $n_0$,   it follows
 	\begin{align}
	 I  &\geq c_H |x-y|^{2H}\lc\sqrt{\frac{-\ln(1-c^*)}{2s}}|x-y|+2\pi\rc^{-2H}\geq c_H s^H \label{eq.LM2.2}
	\end{align}
    since $|x-y|>\sqrt{s}$ and consequently
    \begin{multline*}
      |x-y|^{2H}\lc \sqrt{\frac{-\ln(1-c^*)}{2s}}|x-y|+2\pi \rc^{-2H} \\
      = \lc\sqrt{\frac{-\ln(1-c^*)}{2s}}+\frac{2\pi}{|x-y|}\rc^{-2H} \geq \lc\sqrt{\frac{-\ln(1-c^*)}{2s}}+\frac{2\pi}{\sqrt{s}} \rc^{-2H}=c_H s^H\,.
    \end{multline*}
    Thus,  \eqref{eq.LM2.1} together with \eqref{eq.LM2.2} imply  
    \begin{equation}\label{eq.LMetric2}
       { d_1((t,x),(s,y))\geq c_H(|x-y|^H\wedge s^{\frac H2}+(t-s)^{\frac H2}).}
    \end{equation}
	Combining \eqref{eq.UMetric2} and \eqref{eq.LMetric2}, we   complete  the proof of this lemma.
\end{proof}

\medskip
Now we are ready  to prove  Theorem \ref{LUEESup}, which gives a sharp bound for
$$ \EE \left[ \sup_{{0\le t\le T\atop  -L\le x\le L}} | u_\affine (t,x)|\right]\,.$$

\begin{proof}[Proof   of the first part of Theorem \ref{LUEESup}]
  To simplify notation we denote
\[
 \bT=[0,T]\quad{\rm and}\quad   \bL=\LL\,.
\]
Since $u_\affine (t,x)$ is a symmetric and  centred Gaussian process   Lemma \ref{equiv} states that
\begin{equation}\label{supE}
   \EE \lk\sup_{(t,x)\in\bT\times\bL} |u_\affine (t,x)|\rk\es \EE \lk\sup_{(t,x)\in\bT\times\bL} u_\affine (t,x)\rk\,.
\end{equation}
Hence, to show \eqref{SupEEasmp} it is equivalent to
show
\begin{equation}
 { c_H\, \Uppsi(T, L)\le
\EE \lk\sup_{t\in \bT, x\in \bL}  u_\affine (t,x) \rk \le C_H\,  \Uppsi(T, L)\,,}
 \label{SupEEasmp.a}
    \end{equation}
    where $\Uppsi(T, L)$ is defined by \eqref{e.def_rho}.
We shall prove the upper  and lower bound parts of \eqref{SupEEasmp.a} separately.
 Let us first  consider the upper bound part  in \eqref{SupEEasmp.a}. We shall use the majorizing measure method (Theorem \ref{Talagrand})   and our bound for the natural distance (Lemma \ref{LemMetricProp}). Let us separate the proof into the cases  $L>\sqrt{T}$   and $L\leq \sqrt{T}$.  First, we assume $L>\sqrt{T}$.
 We choose the admissible sequences  $(\cA_n)$ as uniform partition of $\bT\times\bL=[0,T]\times\LL$ such that ${\rm card}(\cA_n)\leq 2^{2^n}$. More precisely, we partition
 $[0, T]\times [-L, L]$ as
\begin{equation*}
 \begin{cases}  [0,T] =\bigcup\limits_{j=0}^{2^{2^{n-1}}-1}\lt[j\cdot2^{-2^{n-1}}T,(j+1)\cdot2^{-2^{n-1}}T\rt)\,, \\ \\
   [-L,L] =\bigcup\limits_{k=-2^{2^{n-2}}}^{2^{2^{n-2}}-1}\lt[k\cdot2^{-2^{n-2}}L,(k+1)\cdot2^{-2^{n-2}}L\rt)\,.
   \end{cases}
 \end{equation*}
Theorem \ref{Talagrand}  states
  \begin{equation}\label{UEESup1}
  	\EE \lk\sup_{(t,x)\in\bT\times\bL} u_\affine (t,x)\rk\leq C\gamma_2(T,d)\leq C\sup_{(t,x)\in\bT\times\bL} \sum_{n\geq 0}2^{n/2} \diam(A_n(t,x))\,.
  \end{equation}
 Here $A_n(t,x)$ is the  element  of uniform partition $\cA_n$ that contains $(t,x)$, i.e.
 \[
  A_n(t,x)=\lt[j\cdot 2^{-2^{n-1}}T,(j+1)\cdot 2^{-2^{n-1}}T\rt)\times \lt[k\cdot 2^{-2^{n-2}}L,(k+1)\cdot 2^{-2^{n-2}}L\rt)
 \]
 such that $j\cdot 2^{-2^{n-1}}T\leq t< (j+1)\cdot 2^{-2^{n-1}}T$ and $k\cdot 2^{-2^{n-2}}L \leq x<(k+1)\cdot 2^{-2^{n-2}}L$.
We only need to estimate diameter of each $A_n(t,x)$. Since $(\cA_n)$ is an uniform partition,  the diameter of $A_n(t,x)$ with respect to    $d_{1,H}((t,x), (s,y))$ defined in \eqref{EqvMetric} can be estimated as
  \[
   \diam(A_n(t,x))\leq C_{H}\lc T^{\frac H2}\wedge (2^{-H 2^{n-2}}L^H) \rc+C_H 2^{-H2^{n-2}}T^{\frac H2}.
  \]
 
 For $L\geq\sqrt{T}$, we can split it into two cases: $\sqrt{T}\leq L<2\sqrt{T}$ and $L\geq 2\sqrt{T}$. It is clear that the case  $L\geq 2\sqrt{T}$ is  more complicated. We consider it at first.  Let $N_0$ be the   smallest integer  such that $2^{-2^{n-2}}L\leq \sqrt{T}$, i.e. $\log_2(\log_2(L/\sqrt{T}))+2\leq N_0< \log_2(\log_2(L/\sqrt{T}))+3$. By  \eqref{UEESup1} we have
  \begin{align}\label{eq.SupEE.aU}
  	&\EE \lk\sup_{(t,x)\in\bT\times\bL} u(t,x)\rk \nonumber\\
  { \leq}& C_H \sup_{(t,x)\in\bT\times\bL} \lk \sum_{n=0}^{N_0}2^{n/2} \diam(A_n(t,x))+ \sum_{n= N_0+1}^{\infty}2^{n/2} \diam(A_n(t,x))\rk \nonumber\\
  		\leq& C_{H} \, T^{\frac H2}\lk \sum_{n= 0}^{N_0}2^{n/2}+\sum_{n= N_0+1}^{\infty}2^{n/2} \lc\frac{2^{2^{N_0-2}}}{2^{2^{n-2}}}\rc^{H} \rk +C_H T^{\frac H2}\nonumber\\
  		\leq&C_H\, T^{\frac H2}\cdot 2^{N_0/2}+C_H T^{\frac H2} \leq C_H\, T^{\frac H2}\Psi_0(T,L)\,,
  \end{align}
  where $\Psi_0(T,L)=1+\sqrt{ \log_2 ( L/ \sqrt{T})}$ and $L\geq 2\sqrt{T}$. The case $\sqrt{T}\leq L<2\sqrt{T}$ is   easy because $\Psi_0(T,L)$ is bounded now. One can prove directly that $\EE \lk\sup_{(t,x)\in\bT\times\bL} u(t,x)\rk\leq C_H T^{\frac H2}$ same as \eqref{eq.SupEE.bU}. This  concludes proof of the upper bound in \eqref{SupEEasmp.a} when $L\geq \sqrt{T}$.

Now,  we   prove  the upper bound part in  \eqref{SupEEasmp.a} when  $L< \sqrt{T}$. The  same uniform partition discussed above is still applicable. We have
  \begin{align}
    &\EE \lk\sup_{(t,x)\in\bT\times\bL} |u(t,x)|\rk \nonumber\\
  \leq& C_H  \lk \sum_{n=0}^{\infty}2^{n/2} \sup_{(t,x)\in\bT\times\bL} \diam(A_n(t,x))\rk \nonumber\\
   \leq& C_H  T^{\frac H2} \sum_{n=0}^{\infty}2^{n/2}\cdot 2^{-H2^{n-1}}+ C_H T^{\frac{H}{2}}\leq C_H T^{\frac{H}{2}}\,,\label{eq.SupEE.bU}
  \end{align}
  because
  \[
   \sup_{(t,x)\in\bT\times\bL} \diam(A_n(t,x))\leq C_H\lk \lc2^{-2^{n-2}}L\rc^H+\lc 2^{-2^{n-1}}T\rc^{\frac H2} \rk\leq C_H 2^{-H2^{n-2}}T^{\frac H2}\,.
  \]
  This completes the upper bounds part of \eqref{SupEEasmp.a}.

  We will utilize Theorem \ref{Sudakov} (Sudakov minoration Theorem) to prove the lower bound  in \eqref{SupEEasmp.a}.  We also divide the proof into two cases:  $L\geq\sqrt{T}$ and $L< \sqrt{T}$.

First, we consider the case $L\geq\sqrt{T}$.  Select $\delta$ in Theorem \ref{Sudakov} as $c_H T^{\frac H2}$ with certain  relatively
 small $c_H>0$. For the sequence $\{u(T,x_i),i=0,1,\cdots,\pm N\}$,  where  $N=\lfloor L/\sqrt{T}\rfloor$ ($\geq 1$ by the assumption)   and
  \[
   x_0=0,x_{\pm1}=\pm \sqrt{T},\cdots, x_{\pm N}=\pm N \sqrt{T}\,,
  \]
   we have
   \[
   d_{1,H}((T,x_i),(T,x_j))\geq c_H T^{\frac H2}=\delta  \quad\hbox{if $i\neq j$}\,.
   \]
     Sudakov's  minoration theorem  implies
  \begin{equation}\label{SupEEasmp.aL}
  \begin{split}
     &\EE \lk\sup_{(t,x)\in\bT\times\bL} |u(t,x)|\rk \geq \EE \lk\sup_{i} u(T,x_i)\rk  \\
     \geq& c_H \delta \sqrt{\log_2(2N+1)}\geq c_H T^{\frac{H}{2}}\Psi_0(T,L)\,.
  \end{split}
  \end{equation}
  The lower bound in \eqref{SupEEasmp.a} is established when $L\geq \sqrt{T}$.

Now we   prove  the lower bound part in  \eqref{SupEEasmp.a} when  $L< \sqrt{T}$.  We choose  $\delta=c_H T^{\frac H2}$ as
above   and we choose $u(T/2,0),u(T,0)$ as our comparison set. We have $d_{1,H}((T/2,0),(T,0))\geq c_H(T/2)^{\frac H2}\geq \delta$.  Theorem \ref{Sudakov} gives
  \begin{equation}\label{eq.SupEE.bL}
    \EE \lk\sup_{(t,x)\in\bT\times\bL}  u(t,x) \rk \geq \EE[u(T/2,0)\vee u(T,0)]\geq c_H T^{\frac H2}\,.
  \end{equation}
Thus,   the proof of the lower bound part in \eqref{SupEEasmp.a}  is completed.
\end{proof}

Notice that  {from \eqref{EqvMetric}, it follows that } for any  fixed $t\in \RR_+$
    \begin{equation}\label{EqvMetric.x}
 {  	 d_1((t,x),(t,y))\asymp d_{t,H}(x,y):=t^{\frac H2}\wedge |x-y|^H\,,
}    \end{equation} 
and for fixed $x\in\RR$
    \begin{equation}\label{EqvMetric.t}
   	 d_1((t,x),(s,x))\asymp d_{1, H}(t,s):=|t-s|^{\frac H2}\,.
    \end{equation}
Using a similar   argument to that in the    proof of  inequality \eqref{SupEEasmp} we have   the following corollary.
\begin{corollary}\label{SupEEasmp.tx}
   Let the Gaussian field $u_\affine (t,x)$ be defined by \eqref{LinearMildSol}.
 There are positive universal constants $c_H$ and  $C_H$ such that
\begin{equation}\label{SupEEasmp.x}
    \begin{cases}
    \displaystyle c_H t^{\frac{H}{2}}\sqrt{\log_2(L)}
    \le    \EE \lk\sup_{-L\le x\le L} |u_\affine (t,x)|\rk\ \\ \qquad\qquad\qquad \quad\le      \EE \lk\sup\limits_{-L\le x\le L}  u_\affine (t,x) \rk \le C_H t^{\frac{H}{2}}\sqrt{\log_2(L)}\,;\\  \\
    \displaystyle c_H T^{\frac{H}{2}}\le      \EE \lk\sup_{0\le t\le T}  u_\affine (t,x) \rk  \le      \EE \lk\sup_{0\le t\le T} |u_\affine (t,x)|\rk\le C_H T^{\frac{H}{2}}\,.
    \end{cases}
\end{equation}
\end{corollary}

Next,   we  shall explain that the almost sure version of  Theorem  \ref{LUEESup}  is a consequence of \eqref{SupEEasmp} with the aid of  Borell's inequality  (Theorem \ref{Borell}).  

\medskip
\begin{proof}[Proof of the second part of Theorem \ref{LUEESup}]
    First,  we shall prove \eqref{Supasmp} for $T=n^\al$ for some $\al$ and for all sufficiently large  integer $n$.
	Denote $\bL:=[-L,L]$, $\bT^{\alpha}=[0,n^{\alpha}]$.   { Let  $\varepsilon>0$ and   let    
    $L\geq n^{\frac{(1+\varepsilon)\alpha}{2}}$ be sufficiently large.     We start with  the lower bound.  Theorem \ref{LUEESup} gives} 
	\[
	 \EE \lk\sup_{(t,x)\in\bT^{\alpha}\times\bL} u_\affine (t,x)\rk\geq c_H\left(
        n^{\frac{\alpha H}{2}}+n^{\frac{\alpha H}{2}}\sqrt{\log_2\lc\frac{L}{n^{\alpha/2}}\rc}\right)
	\]
    for some positive number $c_H$.       Denote
	\[
	 \lambda_{H}:=\lambda_{H}(\bT^\alpha\times\bL)=\frac{1}{2}\EE \lk\sup_{x\in\bT^\alpha\times\bL} u_\affine (t,x)\rk\,,
	\]
    and
	\[
     \sigma^2_{H}:=\sigma^2_{H}(\bT^\alpha\times\bL)=\sup_{(t,x)\in\bT^\alpha\times\bL}\EE[|u_\affine (t,x)|^2]=C_H n^{\frac{\alpha H}{2}} \,.
	\]
    Then,  Borell's inequality implies
	\begin{equation}\label{BorellConsqL}
	\begin{split}
		&\bP\lt\{\sup_{(t,x)\in\bT^{\alpha}\times\bL} u_\affine (t,x)<\frac 12 \EE\Blk\sup_{(t,x)\in\bT^{\alpha}\times\bL}u_\affine
    (t,x)\Brk\rt\}  \leq2\exp\lc -\frac{\lambda^2_{H}}{2\sigma^2_{H}}\rc\\
		& \leq 2\exp\lc{-c_{H}\lk 1+\log_2\lc\frac{L}{n^{\alpha/2}}\rc\rk}\rc
		\leq C_H\lk\frac{n^{\alpha}}{n^{\alpha(1+\varepsilon)}}\rk^{\frac{c_H}{2}}\leq C_H n^{-\alpha\varepsilon\cdot\frac{c_H}{2}}\,,
	\end{split}
	\end{equation}
	where $c_{H},C_H>0$ are some constants independent of  $n$.
	Select real number   $\alpha$ sufficiently large such that $\alpha\varepsilon\cdot\frac{c_H}{2}>1$ and define the events $F_n$
    \[
     F_n:=\lt\{\sup_{(t,x)\in\bT^{\alpha}\times\bL} u_\affine (t,x)<\frac 12 \EE\Blk\sup_{(t,x)\in\bT^{\alpha}\times\bL}u_\affine (t,x)\Brk\rt\}.
    \]
    The bound \eqref{BorellConsqL} means  $\sum_{n=1}^{\infty}\bP(F_n)<\infty$. An application of   Borel-Cantelli's  lemma yields that $\bP(\limsup_n F_n)=0$.
    This means that
	\begin{equation}\label{LSup_as}
		\sup_{(t,x)\in\bT^{\alpha}\times\bL} u_\affine (t,x)\ge \frac 12 \EE\Blk\sup_{(t,x)\in\bT^{\alpha}\times\bL}u_\affine
    (t,x)\Brk\ge
    c_H \,T^{\frac{H}{2}} \Psi_0(T,L) \,,
	\end{equation}
  almost surely for sufficiently large values of 	  $T=n^{\al}$.   Then letting $\ep\to 0$ proves lower bound part of \eqref{Supasmp}.

    The proof of the upper bound in \eqref{Supasmp} can be done  in   exactly the same manner  as  that in the proof of  the lower bound except now we replace  \eqref{BorellConsqL}   by
    \begin{equation}\label{BorellConsqU}
    \begin{split}
     &\bP\lt\{\sup_{(t,x)\in\bT^{\alpha}\times\bL} u_\affine (t,x)>\frac 32 \EE\Blk\sup_{(t,x)\in\bT^{\alpha}\times\bL}u_\affine (t,x)\Brk\rt\}
    \leq2\exp\lc -\frac{\lambda^2_{H}}{2\sigma^2_{H}}\rc\\
	&\leq 2\exp\lc{-c_{H}\lk 1+\log_2\lc\frac{L}{n^{\alpha/2}}\rc\rk}\rc
		\leq C_H\lk\frac{n^{\alpha}}{n^{\alpha(1+\varepsilon)}}\rk^{\frac{c_H}{2}}\leq C_H n^{-\alpha\cdot\frac{c_H\varepsilon}{2}}\,,
    \end{split}
    \end{equation}
    with some positive  constant $c_{H},C_H$ independent of $n$. Similar   to \eqref{LSup_as}  we  have   
    \begin{equation}\label{USup_as}
		\sup_{(t,x)\in\bT^{\alpha}\times\bL} u_\affine (t,x)\leq \frac 32 \EE\Blk\sup_{(t,x)\in\bT^{\alpha}\times\bL}u_\affine
    (t,x)\Brk\leq
    C_H \,T^{\frac{H}{2}} \Psi_0(T,L)
	\end{equation}
    almost surely for sufficiently large    $T=n^{\al}$. And then $\ep\to 0$ implies the upper bound in \eqref{Supasmp}. 

    Finally, we   conclude the proof of   \eqref{Supasmp} for   $\sup_{(t,x)} u_\affine (t,x)$   by combining \eqref{LSup_as}, \eqref{USup_as}  and  the property that $\sup_{(t,x)\in\bT\times\bL} u_\affine (t,x)$ is an increasing function of $L$ and $T$  almost surely.    On the other hand,  it is easy to see
    \[
     \sup_{x}|f(x)|\leq \sup_{x}[f(x)]+\sup_{x}[-f(x)]
    \]
    since $|f(x)|\leq \sup_{x}[f(x)]+\sup_{x}[-f(x)]$ for any function $f(x)$.     Since  $u_\affine(t,x)$ is symmetric, we see that   $\sup_{t,x}[-u_\affine(t,x)]$ and  $\sup_{t,x}[u_\affine(t,x)]$ have the same law. Then,   we have
\begin{equation}
     \sup_{t,x}|u_\affine(t,x)|\leq 2\sup_{t,x}[u_\affine(t,x)]\,.
    \label{e.3.31aa}
    \end{equation} 
This completes  the proof of \eqref{Supasmp}.
\end{proof}

    One can show the following asymptotic \eqref{Supasmp.x} by combining \eqref{SupEEasmp.x} and Borell's inequality  and   we omit the details.

\begin{corollary}\label{LUSupasmp}
	Let $u_\affine (t,x)$ be defined  by	 \eqref{LinearMildSol} and let $T$ satisfy  $T\le L^2$.   Then, there are two positive random constants $c_H$ and $C_H$ such that   for  any fixed $t\in [0,T]$ we have
\begin{align}
 c_H\,   t^{\frac{H}{2}}\sqrt{\log_2(L)}     &\le \sup_{-L\le x\le L}  u_\affine (t,x) \nonumber\\
 &\le \sup_{-L\le x\le L} |u_\affine (t,x)| \le
      C_H\, t^{\frac{H}{2}}\sqrt{\log_2(L)}  \quad \hbox{
almost surely}\,. 
\label{Supasmp.x}
\end{align} 
\end{corollary}
\begin{remark}
  As in \cite{DMD2013,DMDS2013},  the inequality 
  \eqref{Supasmp.x} implies that  there exist  some constants $c,C>0$ such that
  \begin{equation}\label{LimSupasmp.x}
    ct^{\frac H2}\leq \liminf_{|x|\to \infty} \frac{u_\affine (t,x)}{\sqrt{\log_2(|x|)}}\leq \limsup_{|x|\to \infty} \frac{u_\affine (t,x)}{\sqrt{\log_2(|x|)}} \leq Ct^{\frac H2}\,,
  \end{equation}
  for any $t\in\RR_+$ almost surely.
\end{remark}

\medskip
We now turn to show   Theorem \ref{LUHolderEESup}.

\begin{proof}[Proof of Theorem \ref{LUHolderEESup}]
  $\Delta _h u_\affine(t,x)$ is   centered symmetric and  stationary Gaussian process.     As before, we   only need to find appropriate bounds     for $\Delta _h u_\affine(t,x)$.
 The conclusion with respect to  $|\Delta _h u_\affine(t,x)|$ 
  follows from \eqref{e.3.31aa}.  
  Our  strategy  to  prove Theorem \ref{LUHolderEESup} 
  for   $\Delta _h u_\affine(t,x)$  is also   to  apply Talagrand's majorizing measure theorem and   Sudakov's   minoration theorem  to the following Gaussian process
  \begin{equation}\label{Diff}
  \begin{split}
     \Delta _h u_\affine(t,x):=&u_\affine (t,x+h)-u_\affine(t,x)\\
       =&\int_{0}^{t}\int_{\RR} [G_{t-s}(x+h-z)-G_{t-s}(x-z)] W(ds,dz),
  \end{split}
  \end{equation}
  with fixed $t>0$ and fixed $h\neq 0$.  Without loss of generality, we   assume $h>0$.
  The natural metric  is given  by
  \begin{equation*}
    d_{2, t,h}(x,y):=\lc\EE| \Delta _h u_\affine(t,x)- \Delta _h u_\affine (t,y))|^2\rc^{\frac 12} \,.
  \end{equation*}
  We need  to obtain   good upper and lower bounds of  $ d_{2, t,h}(x,y)$.
  Let us first  focus on the upper bound.      Similar to \eqref{NMPlancherel}  Plancherel's identity yields
  \begin{equation*}
  \begin{split}
     d_{2, t,h}^2(x,y)&=C_H \int_{\RR_+}[1-\exp(-2t\xi^2)][1-\cos(|x-y|\xi)][1-\cos(h\xi)]\cdot\xi^{-1-2H}d\xi\,.
  \end{split}
  \end{equation*}
  By the same argument as that in the proof of  the upper bound of $d_1((s,x),(s,y))$ in Lemma \ref{LemMetricProp}  it is easy to see  that for any  $0\leq\theta\leq 1$ 
  \begin{align*}
    d_{2, t,h}^2(x,y)&\le C_H \int_{\RR_+} [1-\exp(-2t\xi^2)][1-\cos(h\xi)]\cdot\xi^{-1-2H}d\xi \\
    &\leq C_{H} t^{H}\wedge h^{2H}\leq  C_{H} t^{H-\theta}h^{2\theta}\,. 
  \end{align*}
   On the other hand, an application of the  elementary inequality $1-\cos(x)\leq C_{\theta}x^{2\theta}$,  where $\theta\in(0,H)$   is  as above,   and a substitution $\xi\rightarrow \xi/|x-y|$   yield 
  \begin{equation*}
  \begin{split}
     d_{2, t,h}^2(x,y)&\leq C_{\theta,H}h^{2\theta}|x-y|^{2H-2\theta} \int_{\RR_+}  [1-\cos(\xi)]\xi^{2\theta-1-2H}d\xi \\
     &\leq C_{\theta,H}h^{2\theta}|x-y|^{2H-2\theta}\,.
  \end{split}
  \end{equation*}
  In conclusion, we have the following bound analogous to upper bound part of \eqref{EqvMetric}:
  \begin{equation}\label{U|x-y|Holder}
    d_{2, t,h}(x,y)\leq C_{ H,\theta}h^{\theta}(|x-y|^{H-\theta}\wedge t^{\frac{H-\theta}{2}})\,,
  \end{equation}
  for any $\theta\in(0,H)$.

{ Now we can follow  the same argument (\eqref{eq.SupEE.aU} in particular) as that  in the proof of Theorem \ref{LUEESup} by   invoking Talagrand's majorizing measure theorem   (Theorem \ref{Talagrand})  to  prove the upper bound part of  \eqref{SupHolderEEasmp}:
  \[
  \EE \lk\sup_{x\in\bL} \Delta _h u_\affine(t,x)\rk\leq C_{ H,\theta} |h|^{\theta} t^{\frac{H-\theta}{2}} \Psi_0(t,L)\,,
  \]
  if $L\geq \sqrt{t}$.
  Now we turn to prove the lower  bound  part of \eqref{SupHolderEEasmp}}.
To this end, we  need   the inverse   part  of \eqref{U|x-y|Holder} and we shall use  again   the  Sudakov minoration theorem.  
      Observe that we only need to consider   the  case  when $|x-y|\geq \sqrt{t}$. We claim
  \[
   d_{2, t,h}^2(x,y)
   \geq c_{H}  h^{2H}\quad   \hbox{when $|x-y|\ge \sqrt{t}$  ~and~  $h\leq \sqrt{\frac{t\pi^2}{8\ln2}}\wedge1$}  \,.
  \]
 In fact,   notice that
 \[
 1-\exp\lc-\frac{2t\xi^2}{|x-y|^2} \rc\geq \frac 12  \qquad \hbox{  $\forall \ \ \xi\ge \frac{|x-y|\pi}{4h}$ and $h\leq \sqrt{\frac{t\pi^2}{8\ln2}}\wedge1$}\,.
 \]
 The
 simple inequality
 \[
 1-\cos(x)\geq x^2/4 \quad \hbox{if $|x|\leq \pi/2$}\,
 \]
 implies 
 \[
 1-\cos \left(\frac{h\xi}{|x-y|}\right)\ge
  \frac{h^2\xi^2}{4|x-y|^2}  \quad \hbox{ if $\xi\le \frac{|x-y|\pi}{2h}$}\,.
 \]
  Therefore, a substitution $\xi\rightarrow \xi/|x-y|$ yields
  \begin{equation*}
    \begin{split}
  &     d_{2, t,h}^2(x,y)\\
        =&c_H |x-y|^{2H} \int_{\RR_+}\lt[1-\exp\lc-\frac{2t\xi^2}{|x-y|^2}\rc\rt]\lt[1-\cos\lc \frac{h\xi}{|x-y|}\rc\rt] \\
        &\qquad\qquad\qquad\qquad\qquad\qquad\qquad\qquad\qquad\cdot[1-\cos(\xi)]\xi^{-1-2H}d\xi \\
        \geq& c_H |x-y|^{2H}\int_{\frac{|x-y|\pi}{4h}}^{\frac{|x-y|\pi}{2h}} \lt[1-\cos\lc \frac{h\xi}{|x-y|}\rc\rt][1-\cos(\xi)]\cdot\xi^{-1-2H}d\xi \\
        \geq& c_H h^2|x-y|^{2H-2} \int_{\frac{|x-y|\pi}{4h}}^{\frac{|x-y|\pi}{2h}} [1-\cos(\xi)]\cdot\xi^{1-2H}d\xi\,. 
    \end{split}
  \end{equation*}
  Set
  \[
   k_0=\inf \left\{k\in\bN_0:\frac{(2k+1)\pi}{2}\geq \frac{|x-y|\pi}{4h}\rt\}\,;
  \]
  and
  \[
   k_1=\sup\lt\{k\in\bN_0:\frac{(2k+3)\pi}{2}\leq \frac{|x-y|\pi}{2h}\rt\}\,.
  \]
  If $h$  is  sufficiently small,  then
  \begin{equation*}
  	\begin{split}
  		&\int_{\frac{|x-y|\pi}{4h}}^{\frac{|x-y|\pi}{2h}} [1-\cos(\xi)]\cdot\xi^{1-2H}d\xi=\sum_{k\geq 0} \int_{I_k\cap [\frac{|x-y|\pi}{4h},\frac{|x-y|\pi}{2h}]} [1-\cos(\xi)]\cdot\xi^{1-2H}d\xi \\
  	   \geq &\sum_{k=k_0}^{k_1}  \int_{I_k} [1-\cos(\xi)]\cdot\xi^{1-2H}d\xi \geq \frac 12 \int_{\frac{(2k_0+1)\pi}{2}}^{\frac{(2k_1+3)\pi}{2}} \xi^{1-2H}d\xi \\
  	   =& c_H\lk \lc\frac{(2k_1+3)\pi}{2}\rc^{2-2H}- \lc\frac{(2k_0+1)\pi}{2}\rc^{2-2H}\rk\geq c_H \lc\frac{|x-y|}{h}\rc^{2-2H},
  	\end{split}
  \end{equation*}
  due to the fact that $\xi^{1-2H}$ is an increasing function. Thus,  we have for $|x-y|\geq \sqrt{t}$
  \begin{equation}\label{L|x-y|Holder}
  	d_{2, t,h}(x,y)\geq c_{H}h^{H}  
  \end{equation}
  if $h\leq C(\sqrt{t}\wedge1)$ for some small positive  quantity $C$.  On the interval $\bL=\LL$ for $L$ large enough, let us select $x_j=jL/\sqrt{t}$ for $j=0,\pm 1,\cdots,\pm \lfloor L/\sqrt{t}\rfloor$. Similar to \eqref{SupEEasmp.aL}, applying the Sudakov minoration
  theorem  (Theorem \ref{Sudakov}) with $\delta=c_H |h|^H$
  yields
  \begin{equation*}
  	\EE\lk\sup_{x\in\bL}\Delta _h u_\affine(t,x)\rk\geq \EE\lk\sup_{x_i}\Delta _h u_\affine(t,x)\rk\geq c_H |h|^H\Psi_0(t,L)\,.
  \end{equation*}
  The proof of \eqref{SupasmpHolder}  follows  from exactly the same argument  as that in the proof of \eqref{Supasmp}   by   Borel-Cantelli's  lemma. The only difference is that now we have
  \[
   \begin{cases}  \sigma_t^2(h)=\sup\limits_{x\in\bL^{\alpha}}\EE[|\Delta_h u_\affine(t,x)|^2]\leq C_{H,\theta} t^{H-\theta}|h|^{2\theta}\,;\\ \\
   \lambda_L:=\frac 12 \EE\lt[\sup\limits_{x\in\bL^{\alpha}} \Delta_h u_\affine(t,x)\rt]\,;\\ \\
   \exp\lc-\frac{\lambda_L^2}{2\sigma_t^2(h)}\rc\leq C_{H,\theta}\exp\lc-\lk\frac{h^2}{t}\rk^{H-\theta}\log_2\lt[\frac{n^\alpha}{\sqrt{t}} \rt]\rc\,,
   \end{cases}
  \]
where $\bL^{\alpha}:=[-n^\alpha,n^\alpha]$. We can then complete  the proof  of the theorem by choosing $\alpha$ appropriately.  We omit the details here.
\end{proof}
%

 \medskip
\begin{proof}[Proof of Theorem \ref{LUtHolderEESup}]
{    We will use the same method as  that in the proof of Theorem 
\ref{LUHolderEESup}. } 
 The  natural metric associated with the time increment of the solution is
 \[
 d_{3, t, \tau} (x,y)=(\EE|\Delta_\tau u_\affine(t,x)-\Delta_\tau u_\affine (t,y)|^2)^{\frac 12}\,.
 \]
 Using
	\[
    \Delta_\tau u_\affine(t,x) =\int_{0}^{t+\tau}\int_{\RR} G_{t+\tau-s}(x-z)W(ds,dz)-\int_{0}^{t}\int_{\RR} G_{t-s}(x-z)W(ds,dz)\,,
	\]
 and using the isometric property of stochastic integral and Plancherel's identity  one   derives  
 \begin{align}
 { d_{3, t,\tau}^2(x,y)}&=2\int_{\RR_+} f(t, \tau, \xi) [1-\cos(|x-y|\xi)]\cdot \xi^{-1-2H}d\xi\,,  \label{eq.NMtHolder}
 \end{align}
 where
	\begin{equation*}
	\begin{split}
		&f(t,\tau,\xi) \\
	   =&[1-\exp(-2(t+\tau)\xi^2)]+[1-\exp(-2t\xi^2)]-2\exp(-\tau\xi^2)[1-\exp(-2t\xi^2)] \\
	   =&[1-\exp(-2\tau\xi^2)]+[1-\exp(-2t\xi^2)][1+\exp(-2\tau\xi^2)-2\exp(-\tau\xi^2)].
	\end{split}
	\end{equation*}
 Notice  that  when $x\geq 0$,  $1-e^{-x}\leq C_{\theta}x^\theta$  and $1+e^{-2x}-2e^{-x}=(1-e^{-x})^2\leq C^2_{\theta}x^{2\theta}$ for any $\theta \in (0, 1)$.   Then,
    we have
	\[
	 f(t,\tau,\xi)\leq C_{\theta}(\tau\xi^2)^\theta \,,\quad \forall \  \theta\in (0, 1)\,.
	\]
 Inserting this bound into \eqref{eq.NMtHolder}   yields
 \begin{align*}
 d_{3, t,\tau}^2(x,y)\le & C_{\theta} \tau^\theta  \int_{\RR_+}  [1-\cos(|x-y|\xi)]\cdot \xi^{-1-2H+2\theta}d\xi\nonumber\\
 \le&  C_{H,\theta} \tau^\theta |x-y|^{2H-2\theta}
 \qquad \hbox{ for  any $0<\theta<H$}.
 \end{align*}
 On the other hand,   a substitution $\xi\rightarrow \xi/\sqrt{\tau}$ yields
    \begin{equation*}
	\begin{split}
		d_{3, t,\tau}^2(x,y)\leq&  C \int_{\RR_+}[1-\exp(-2\tau\xi^2)]\xi^{-1-2H}d\xi \\
        &+\int_{\RR_+} [1-\exp(-2t\xi^2)] [1-\exp(-\tau\xi^2)]^2 \xi^{-1-2H}d\xi \\
		\le& C_H \tau^{H}+C_{H,\theta}\tau^{\theta}t^{H-\theta} \leq C_{H,\theta} \tau^{\theta}t^{H-\theta} 
	\end{split}
	\end{equation*}
 when $\tau\leq C t$.
 Thus,  we have
	\begin{equation}\label{U|x-y|tHolder}
		d_{3, t,\tau}(x,y)\leq C_{H,t,\theta} \tau^{\theta/2} (|x-y|^{H-\theta}\wedge t^{\frac{H-\theta}{2}})\,,
	\end{equation}
 where $0<\theta<H$, which is the bound needed for us
 to prove  the upper bound part of \eqref{SuptHolderEEasmp}.
	
  The Sudakov minoration Theorem \ref{Sudakov} will still be used to  prove  the lower bound.
  We need to obtain an appropriate lower bound of $d_{3, t,\tau}(x,y)$ for $|x-y|\ge \sqrt{t}$.  It is easy to see 
	\begin{equation}
	\begin{split}
		d_{3, t,\tau}^2(x,y)&\geq c\int_{\RR_+}[1-\exp(-2\tau\xi^2)][1-\cos(|x-y|\xi)]\cdot \xi^{-1-2H}d\xi \\
			&\geq c\tau |x-y|^{2H-2} \int_{\frac{|x-y|}{2\sqrt{\tau}}}^{\frac{|x-y|}{\sqrt{\tau}}}
        [1-\cos(\xi)]\xi^{1-2H}d\xi.
        \label{e.3.integral1}
	\end{split}
	\end{equation}
 Analogous to the obtention of  \eqref{L|x-y|Holder}   we can conclude that the integral in \eqref{e.3.integral1}  is bounded below by a multiple of $\left(\frac{|x-y|}{\sqrt{\tau}}\right)^{2-2H}$. Thus, we obtain
	\begin{equation}\label{L|x-y|tHolder}
		d_{3, t,\tau}(x,y)\geq c_{H}\tau^{H/2} 
	\end{equation}
 if $\tau\leq C(t\wedge1)$ for some constant  $C$. This is the bound needed to use Theorem \ref{Sudakov} to show the lower bound part of \eqref{SuptHolderEEasmp}.

Once again, Borell's  inequality  (Theorem \ref{Borell})      can be combined with   Borel-Cantelli's lemma to  show   the almost sure  asymptotics \eqref{SupasmpHolder}, and the proof Theorem \ref{LUtHolderEESup}  is completed.
\end{proof}

In \cite{HHLNT2017} (see also next section)   to show the existence and uniqueness of the solution to \eqref{SHE} (for Hurst parameter $H\in (1/4, 1/2)$)
it is extensively used the following quantity
\begin{equation}
     \cN_{\frac 12 -H}u (t,x)=\lc\int_{\RR}|u  (t,x+h)-u  (t,x)|^2\cdot|h|^{2H-2} dh\rc^{\frac 12}\,,
     \label{e.3.defN}
    \end{equation}
    which plays the role of fractional derivative of $u$.
 It is because of the difficulty to appropriately bound this quantity
(see \cite{HHLNT2017} or the next section)    it is assumed that $\sigma(0)=0$ in \cite{HHLNT2017}.  After our work on the bound of the solution $u_\affine (t,x)$  we   want to argue  that
\begin{equation}
\EE\lk\sup_{x\in \bL}\cN^2_{\frac 12 -H}u_\affine(t,x)\rk\geq c_{t,H} \log_2(L)\quad   \hbox{if $L$ is
sufficently large}.
\end{equation}
This fact illustrates  that the  argument in \cite{HHLNT2017}   for the pathwise uniqueness   (see Lemma 4.9 in \cite{HHLNT2017} for this argument)  is not applicable in the general setting when $\sigma(0)\neq 0$.
Here is the precise statement of our  result, which is also interesting for its own sake.
\begin{proposition}\label{Prop.SupN}
  Let $u_\affine(t,x)$ be defined by  \eqref{LinearMildSol} and let
  $\cN_{\frac 12 -H}u_\affine(t,x)$ be
defined   by \eqref{e.3.defN}.
\begin{enumerate}
\item[(i)]   For any fixed $t>0$ and $L\geq \sqrt{t}$  we have
    \begin{equation}\label{LSupNEE}
        \EE \lk\sup_{-L\le x\le L} \cN^2_{\frac 12 -H}u_\affine(t,x)\rk\geq  { c_{t,H}} \, \log_2(L)\,,
    \end{equation}
where $c_{t, H}$ is a positive constant.
\item[(ii)] Moreover, we have almost surely if $L\geq \sqrt{t}$
    \begin{equation}\label{USupN_as}
        \sup_{-L\le x\le L} \cN_{\frac 12 -H}u_\affine(t,x)\leq C_{H}\, 
        t^{2H-\frac 12} [1- \log(\sqrt{t}\wedge 1)]\Psi_0(t,L)\,,
    \end{equation}
    where  $C_{H}$ is a positive random constant.
    \end{enumerate}
\end{proposition}
\begin{proof}
 First, we    consider  the upper bound \eqref{USupN_as}.  Let    $0<\theta<\frac{1-2H}{2}$. 
Applying Theorem \ref{LUHolderEESup} when $|h|\leq \sqrt t \wedge 1$ and Theorem \ref{LUEESup} when $|h|>\sqrt t \wedge 1$, respectively,  {  and 
 using the notation $\Delta _h u_\affine(t,x):=u_\affine (t,x+h)-u_\affine(t,x)$} we obtain
  \begin{equation*}
    \begin{split}
       \sup_{x\in\bL} &\cN^2_{\frac 12 -H}u_\affine(t,x)=\sup_{x\in\bL} \int_{\RR}|\Delta _h u_\affine(t,x)|^2\cdot|h|^{2H-2} dh \\
       &\leq \int_{\RR} \lc\sup_{x\in\bL}|\Delta _h u_\affine(t,x)|\rc^2\cdot|h|^{2H-2} dh \\
       &\leq \int_{\{|h|\leq \sqrt t\wedge 1 \}} \lc\sup_{x\in\bL}|\Delta _h u_\affine(t,x)|\rc^2\cdot|h|^{2H-2} dh \\
       &\quad +\int_{\{|h|> \sqrt t\wedge 1\}} \lc\sup_{x\in\bL}|\Delta _h u_\affine(t,x)|\rc^2\cdot|h|^{2H-2} dh \\
       &\leq C_{H,\theta} t^{H-\theta}\Psi_0(t,L)\int_{\{|h|\leq \sqrt t\wedge 1\}} |h|^{2H-2+2\theta} dh +C_{H} t^{H}\Psi_0(t,L)\\
       &\quad \cdot\lk \int_{\{|h|> \sqrt t\wedge 1\}} |h|^{2H-2} dh+\int_{\{|h|>\sqrt t\wedge 1\}}\log_2(|h|/\sqrt t) |h|^{2H-2} dh\rk\,,
    \end{split}
  \end{equation*}
where we  applied an elementary inequality
\[
\left|\log_2\left|L+h\right|\right|\leq
\begin{cases}
\log_2(L)+ 1 &  \quad  \hbox{ when  $|h|\le 1$}\,;\\
\log_2(L)+\log_2(|h|)+1 & \quad
\hbox{  when  $|h|\ge 1$}\,. \\
\end{cases}
\]
Most of terms of above integrals can be evaluated easily except the one involving with $\log_2(|h|/\sqrt t)$, which equals to
\begin{align*}
	&\int_{\{|h|>\sqrt t\wedge 1\}}\blk\log_2(|h|)-\log_2(\sqrt t)\brk |h|^{2H-2} dh\\
	\leq& \int_{\{|h|>\sqrt t\wedge 1\}}\blk\log_2(|h|)-\log_2(\sqrt t\wedge 1)\brk |h|^{2H-2} dh\ls (\sqrt t\wedge 1)^{2H-1}[1- \log(\sqrt{t}\wedge 1)]\,.
\end{align*}
This yields \eqref{USupN_as}.

Now we turn to  the lower bound \eqref{LSupNEE}.
A simple observation  and an application of Jensen's inequality give
  \begin{equation}\label{SupNEELobs}
    \begin{split}
        &\EE\lk\sup_{x\in\bL} \cN^2_{\frac 12 -H}u_\affine(t,x)\rk \\
       \geq& c_H \EE\lk\lc\sup_{x\in\bL} \int_{\RR}\Delta _h u_\affine(t,x) \varrho(h) dh\rc^2\rk
       \geq c_H \lc\EE\lk\sup_{x\in\bL} \int_{\RR} \Delta _h u_\affine(t,x) \varrho(h) dh\rk\rc^2,
    \end{split}
  \end{equation}
  where $\varrho(h)=C_H\,\blk|h|^{2H-\frac 32}\1_{\{|h|\leq 1\}}+|h|^{2H-2}\1_{\{|h|> 1\}}\brk$ such that it is probability density.   Denote
  \begin{equation*}
    \begin{split}
       u_\varrho(t,x)&= \int_{\RR}\Delta _h u_\affine(t,x)  \varrho(h) dh 
         = \int_{0}^{t}\int_{\RR} \lc\int_{\RR}D_{t-s}(h,x-z) \varrho(h) dh\rc W(ds,dz)\,,
    \end{split}
  \end{equation*}
where $D_{t}(h,x)$ is defined in \eqref{FirDiff}. The above $u_\varrho(t,x)$  is a well-defined Gaussian random field
since $\varrho(h)$ is integrable for $\frac{1}{4}<H<\frac{1}{2}$.  Introduce  the induced natural metric
\[
d_{4,t}(x,y):=(\EE|u_\varrho(t,x)-u_\varrho (t,y)|^2)^{\frac 12}\,.
\]
We need to bound this  distance
 for $|x-y|\geq 1$. Applying Plancherel’s identity we can find
  \begin{equation*}
    \begin{split}
       d_{4,t}^2 (x,y) &=c_H \int_{\RR_+} [1-\exp(-2t\xi^2)][1-\cos(|x-y|\xi)] \\
         & \qquad\qquad\cdot \lc\int_{\RR_+}[1-\cos(h\xi)]\varrho(h) dh\rc^2\cdot \xi^{-1-2H}d\xi\,.
    \end{split}
  \end{equation*}
When  $\xi\geq 1$, we  have
  \[
   \int_{\RR_+}[1-\cos(h\xi)]\varrho(h) dh\geq \xi^{\frac 12-2H}\int_{0}^{\xi}[1-\cos(h)]\cdot h^{2H-\frac 32} dh\geq c \xi^{\frac 12-2H}.
  \]
Thus,  we   conclude  that  if $|x-y|\geq 1$, then
  \begin{align}\label{L|x-y|N}
    d_{4,t}^2 (x,y)&\geq c_{H}[1-\exp(-2t)] \int_{1}^{\infty}[1-\cos(|x-y|\xi)]\cdot \xi^{-6H}d\xi \nonumber\\
        &\geq c_{H}[1-\exp(-2t)]
  \end{align}
  by the same  argument as that  in proof of lower bound of $\EE[\sup_{x\in\bL} \Delta_h u_{\affine}(t,x)]$ in Theorem \ref{LUHolderEESup}.  An application of the Sudakov minoration Theorem \ref{Sudakov} implies the lower bound \eqref{LSupNEE}.
\end{proof}

\section{Weak Existence and Regularity of Solutions}
\label{weak}
 \subsection{Basic settings}
 This section is devoted to prove the existence of a  weak solution  to  $\eqref{SHE}$.
 Let us briefly recall some notations  and facts in  \cite{HHLNT2017}.   Let $(B,\| \cdot \|_B)$ be a Banach space  with the norm $\| \cdot \|_B$.  Let   $\beta\in(0,1)$ be a fixed number.
 For  any  function $f:\RR\rightarrow B$  denote
 \begin{equation}\label{NBNorm}
   \cN_{\beta}^{B}f(x)=\lt(\int_{\RR}\|f(x+h)-f(x)\|_B^2|h|^{-1-2\beta}dh\rt)^{\frac 12}\,,
 \end{equation}
 if the above quantity is finite.
 When $B=\RR$, we abbreviate the notation $\cN_{\beta}^{\RR}f$  as  $\cN_{\beta}f$ (see also \eqref{e.3.defN}).  
 As in \cite{HHLNT2017} throughout  this paper  we are particularly interested in the case $B=L^p(\Omega)$, and in this case we  denote $ \cN_{\beta}^{B}$ by $\cN_{\beta,p}$:
 \begin{equation}\label{NpNorm}
   \cN_{\beta,p}f( x)=\lt(\int_{\RR}\|f(x+h)-f(x)\|^2_{L^p(\Omega)} |h|^{-1-2\beta}dh\rt)^{\frac 12}.
 \end{equation}
The following    Burkholder-Davis-Gundy     inequality is well-known
(see  e.g.  \cite{HHLNT2017}).


 \begin{proposition}\label{HBDG}
   Let $W$ be the Gaussian noise defined by the covariance \eqref{CovStru}, and let    $f\in\Lambda_H$  be a predictable random field. Then for any $p\geq2$ we have
   \begin{equation}
        \lt\|\int_{0}^{t}\int_{\RR}f(s,y) W(ds,dy)\rt\|_{L^p(\Omega)}  
     \leq \sqrt{4p}\, c_{ H}\lt(\int_{0}^{t}\int_{\RR}\lk\cN_{\frac 12-H,p}f(s,y)\rk^2dyds\rt)^{\frac 12},\label{e.bdg}
   \end{equation}
   where $c_{ H}$ is a constant depending only on $H$ and
   $\cN_{\frac 12-H,p}f(s,y)$ denotes the application of $\cN_{\frac 12-H,p} $   with respect to  the space variable $y$.
 \end{proposition}

 In the work  \cite{HHLNT2017}, the authors have already proved the  existence and uniqueness result
 in a solution space $\cZ^p_T$
 (see \cite{HHLNT2017} or  formula  \eqref{ZNorm} in next paragraph for the definition of  $\cZ^p_T$)   under the condition $\sigma(t,x,0)=0$.  When $\sigma(t,x, 0)\not=0
$ or even in the simplest case $\sigma (t,x, u)=1$ (as we see from \eqref{LSupNEE})   we cannot expect that the solution  is still in $\cZ^p_T$.  So,  the method
 powerful in \cite{HHLNT2017} is no longer valid to solve   \eqref{SHE}  for  general
 $\sigma(t,x,u)$.  Our idea is to add an appropriate weight $\lambda(x)$ to the  space
 $\cZ^p_T$ to obtain a weighted space $\cZ^p_{\lambda, T}$.

 Let  $\lambda(x)\ge 0$ be a Lebesgue integrable positive function  with $\int_{\RR}\lambda(x)dx=1$.  Introduce a norm $\|\cdot\|_{\cZ_{\lambda,T}^p}$ for a random field $v(t,x)$ as follows:
 \begin{equation}\label{ZNorm}
   \|v\|_{\cZ_{\lambda,T}^p} :=\sup_{t\in[0,T]} \|v(t,\cdot)\|_{L^p_{\lambda}(\Omega\times\RR)}+\sup_{t\in[0,T]}\cN^*_{\frac 12-H,p}v(t),
 \end{equation}
 where $p\geq 2$, $\frac 14<H<\frac 12$,
 \[
  \|v(t,\cdot)\|_{L^p_{\lambda}(\Omega\times\RR)}=\lc\int_{\RR} \EE\lk|v(t,x)|^p\rk \lambda(x)dx\rc^{\frac 1p},
 \]
 and
 \begin{equation}\label{N*pNorm}
   \cN^*_{\frac 12-H,p}v(t)=\lc\int_{\RR}\|v(t,\cdot)-v(t,\cdot+h)\|^2_{L^p_{\lambda}(\Omega\times\RR)}|h|^{2H-2}dh\rc^{\frac 12}.
 \end{equation}
 Then $\cZ_{\lambda,T}^p$ is the function space consisting of  all the random fields $v=v(t,x)$ such that $\|v\|_{\cZ_{\lambda,T}^p}$ is finite.   When the function is independent of $t$, the corresponding space is denoted by  $\cZ_{\lambda,0}^p$.

 \subsection{Some bounds for stochastic convolutions} To  prove the   existence of weak  solution, we need some delicate estimates of stochastic integral with respect to the weight.
 \begin{proposition}\label{CharProp} Denote the weight function
   \begin{equation}\label{lamd}
     \lambda(x)=\lambda_H(x)= c_H(1+|x|^2)^{H-1}\,,
   \end{equation}
where $c_H$ is a constant such that $\int \lambda(x) dx=1$,      and  denote
   \begin{equation}\label{Phiprop}
     \Phi(t,x)=\int_{0}^{t}\int_{\RR} G_{t-s}(x-y)v(s,y)W(ds,dy).
   \end{equation}
We have the following estimates.
 (In the following $C_{ T,p,H, \gamma}$  denotes a constant,    depending only  on $T$, $p$, $H$ and   $\gamma$).
   \begin{enumerate}[leftmargin=*]
     \item[\textbf{(i)}] If $p>\frac{3}{H}$, then
     \begin{equation}\label{EstZl}
       \Big\|\sup\limits_{t\in[0,T], x\in \RR}\lambda^{\frac{1}{p}}(x) \Phi(t,x)\Big\|_{L^p(\Omega)}\leq C_{ T,p,H}\|v\|_{\cZ_{\lambda,T}^p}\,.
     \end{equation}
     \item[\textbf{(ii)}] If $p>\frac{6}{4H-1}$, then
     \begin{equation}\label{EstZlN}
       \Big\|\sup\limits_{t\in[0,T], x\in \RR}\lambda^{\frac{1}{p}}(x) \cN_{\frac 12-H}\Phi(t,x)\Big\|_{L^{ p}(\Omega)}\leq C_{ T,p,H}\|v\|_{\cZ_{\lambda,T}^p}\,.
     \end{equation}
     \item[\textbf{(iii)}] If $p>\frac{3}{H}$,
     and  $0<\gamma<\frac H2-\frac{3}{2p}$,  then
     \begin{equation}\label{EstZ3}
       \Big\|\sup\limits_{\substack{t,t+h\in[0,T] \\ x\in \RR}}\lambda^{\frac{1}{p}}(x) \blk\Phi(t+h,x)-\Phi(t,x)\brk\Big\|_{L^p(\Omega)}\leq C_{ T,p,H, \gamma} |h|^{\gamma}\|v\|_{\cZ_{\lambda,T}^p}\,.
     \end{equation}
     \item[\textbf{(iv)}] If $p>\frac{3}{H}$,
     and  $0<\gamma<H-\frac{3}{p}$,  then
  \begin{equation}\label{EstZ4}
       \lt\|\sup\limits_{\substack{t\in[0,T] \\ x,y\in \RR}}\frac{\Phi(t,x)-\Phi(t,y)}{\lambda^{-\frac{1}{p}}(x)+\lambda^{-\frac{1}{p}}(y)} \rt\|_{L^p(\Omega)}\leq  C_{ T,p,H, \gamma }|x-y|^{\gamma}\|v\|_{\cZ_{\lambda,T}^p}\,.
     \end{equation}
   \end{enumerate}
 \end{proposition}
 \begin{remark}
The method provided in the following proof  depends on the semigroup property of the heat kernel because we need to use the factorization method (e.g. \cite{GJ2014}. see also 
\eqref{e.4.12} below). This means that we can not apply our  approach directly to the stochastic wave equation since  the wave  kernel (the fundamental solution of the wave equation in \tcg{\cite{BJQ2015}}) lacks  the semigroup property.
 \end{remark}
 \begin{proof} For  any $\alpha\in(0,1)$ we set
 \begin{equation}
      J_{\alpha}(r,z):=\int_{0}^{r}\int_{\RR} (r-s)^{-\alpha} G_{r-s}(z-y)v(s,y)W(ds,dy).\label{e.def_J}
 \end{equation}
 A stochastic version of Fubini's theorem implies
    \begin{equation}
      \Phi(t,x)=\frac{\sin(\pi\alpha)}{\pi}\int_{0}^{t}\int_{\RR}(t-r)^{\alpha-1}G_{t-r}(x-z)J_{\alpha}(r,z)dzdr.
      \label{e.4.12}
    \end{equation}
 We are going to show the four different parts of the  proposition separately.  We divide our proof into six steps.  { Let us recall $D_t(x,h):=G_t(x+h)-G_t(x)$, and $\Box_t(x,y,h)=D_t(x+y,h)-D_t(x,h)$ defined in \eqref{FirDiff} and \eqref{SecDiff}.
 }

 \noindent \textbf{Step 1.} \   The first two steps are to
  prove part  \textbf{(i)}.
  In this step we will obtain the desired   growth estimate of  $\Phi(t,x)$ in term of $J_{\alpha}(r,z)$.   Applying  the bounds of   \eqref{DGcontra_t} and  \eqref{Glamd}  to   
     \eqref{e.4.12}   we have
     \begin{equation*}
       \begin{split}
       &\sup_{t,x} \lambda^{\theta}(x)\left| \Phi(t,x)\right| \\
         \es  &\sup_{t,x} \lambda^{\theta}(x)\left|\int_{0}^{t}\int_{\RR}(t-r)^{\alpha-1}G_{t-r}(x-z)J_{\alpha}(r,z)dzdr\right| \\
          \ls&\sup_{t,x} \lambda^{\theta}(x)\int_{0}^{t}(t-r)^{\alpha-1}\lc\int_{\RR}|G_{t-r}(x-z)\lambda^{-\frac 1p}(z)|^{q}dz \rc^{\frac 1q} \|J_{\alpha}(r,\cdot)\|_{L_{\lambda}^p(\RR)}dr  \\
           \ls&\sup_{t,x} \lambda^{\theta}(x)\int_{0}^{t}(t-r)^{\alpha-1}\lc\int_{\RR}
           (t-r)^{ \frac{1-q}{2}}
           G_{(t-r)/q}(x-z)  \lambda^{-\frac qp}(z)dz \rc^{\frac 1q} \|J_{\alpha}(r,\cdot)\|_{L_{\lambda}^p(\RR)}dr  \\
          \ls&\sup_{t,x} \lambda^{\theta}(x)\int_{0}^{t}(t-r)^{\alpha-1}\cdot (t-r)^{\frac{1-q}{2q}}\lambda^{-\frac 1p}(x)\cdot \|J_{\alpha}(r,\cdot)\|_{L_{\lambda}^p(\RR)}dr\,.
       \end{split}
     \end{equation*}
Setting  $\theta=\frac 1p$ and  then applying the H\"{o}lder
inequality  we obtain
     \begin{align}
          \sup_{t,x}\lambda^{\theta}(x)|\Phi(t,x)| &\lesssim \sup_{t\in[0,T]} \int_{0}^{t}(t-r)^{\alpha-\frac 32+\frac{1}{2q}}\cdot \|J_{\alpha}(r,\cdot)\|_{L_{\lambda}^p(\RR)}dr \nonumber\\
            &\lesssim \sup_{t\in[0,T]}\lk\int_{0}^{t}(t-r)^{q(\alpha-\frac 32+\frac{1}{2q})}dr\rk^{\frac 1q}\cdot \lk\int_{0}^{T} \|J_{\alpha}(r,\cdot)\|^p_{L_{\lambda}^p(\RR)} dr\rk^{\frac 1p} \nonumber\\
            &\lesssim   \lk\int_{0}^{T} \|J_{\alpha}(r,\cdot)\|^p_{L_{\lambda}^p(\RR)} dr\rk^{\frac 1p} \label{e.4.14}
     \end{align}
if  $q(\alpha-\frac 32+\frac{1}{2q})>-1$, i.e.  if
\begin{equation}
\alpha>\frac{3}{2p}\,.  \label{e.4.alpha1}
\end{equation}
This is possible if  $p>3/2$. Thus,   to prove part \textbf{(i)}, we only need to show that
there exists a constant $C $, independent of $r\in [0, T]$,
  such that
     \begin{equation}\label{CharEst1}
      \EE \|J_{\alpha}(r,\cdot)\|^p_{L_{\lambda}^p(\RR)}\leq C  \|v\|^p_{\cZ_{\lambda,T}^p}.
     \end{equation}

\noindent \textbf{Step 2.} \      We shall   prove
the above bound \eqref{CharEst1}
in this step and
   to do this let us introduce the following two notations
     \begin{align*}
      \cD_1(r,z) &:=\bigg(\int_{0}^{r}\int_{\RR^2}(r-s)^{-2\alpha} \big|D_{r-s}(y,h)\big|^2  \|v(s,y+z)\|_{L^p(\Omega)}^2|h|^{2H-2}dhdyds\bigg)^{\frac p2}\,,
     \end{align*}
 and  
     \begin{align*}
      \cD_2(r,z) &:=\bigg( \int_{0}^{r}\int_{\RR^2}(r-s)^{-2\alpha}|G_{r-s}(y)|^2 \|\Delta_h v(s,y+z)\|_{L^p(\Omega)}^2 |h|^{2H-2} dhdyds\bigg)^{\frac p2}\,,
     \end{align*}
     where  $\Delta_h v(t,x):=v(t,x+h)-v(t,x)$. 
From the definition \eqref{e.def_J} of $J$ and  by  Burkholder-Davis-Bundy's  inequality \eqref{e.bdg} stated in  Lemma \ref{HBDG}, we  have
  \begin{align*}
  \EE \|J_{\alpha}(r,\cdot)\|^p_{L_{\lambda}^p(\RR)}
   &\ls \int_\RR \bigg\{ \int_0^r \int_{\RR^2}
   (r-s)^{-2\alpha} \bigg[  \EE \big|G_{r-s}( y+h-z)v(s,y+h)\nonumber\\
    &\qquad -G_{r-s}(y-z)v(s,y)\big|^p \bigg]^{2/p}  h^{2H-2} dhdyds \bigg\}^{p/2}  \lambda(z) dz \nonumber\\
    & \lesssim  \int_{\RR}\blk\cD_1(r,z)+\cD_2(r,z)\brk\lambda(z) dz
 {    =:  D_1+ D_2 \,.}
\nonumber
  \end{align*}
   For the first term $I_1$, thanks to Minkowski's inequality, we  have
\begin{align} \label{CharEst1D1}
{      D_1 }
         \lesssim&  \bigg(\int_{0}^{r}\int_{\RR^2}(r-s)^{-2\alpha} \big|D_{r-s}(y,h)\big|^2 \cdot\|\Delta_y v(s,\cdot)\|_{L^p_{\lambda}(\Omega\times\RR)}^2 |h|^{2H-2}dhdyds\bigg)^{\frac p2} \nonumber\\
           +&\bigg(\int_{0}^{r}\int_{\RR^2}(r-s)^{-2\alpha} \big|D_{r-s}(y,h)\big|^2\cdot\|v(s,\cdot)\|_{L^p_{\lambda}(\Omega\times\RR)}^2|h|^{2H-2}dhdyds\bigg)^{\frac p2} \nonumber\\
        \lesssim&\bigg(\int_{0}^{r}\int_{\RR}(r-s)^{-2\alpha-\frac 12} \|\Delta_y v(s,\cdot)\|_{L^p_{\lambda}(\Omega\times\RR)}^2|y|^{2H-2}dyds\bigg)^{\frac p2} \nonumber\\
            &\qquad\quad +\bigg(\int_{0}^{r}(r-s)^{-2\alpha+H-1} \|v(s,\cdot)\|_{L^p_{\lambda}(\Omega\times\RR)}^2ds\bigg)^{\frac p2}\,,
     \end{align}
     where the last inequality follows from   inequalities
     \eqref{DGcontra_t} and \eqref{FirDiffIneq}.

     For the second term $I_2$, we can again use Minkowski's inequality,  Jensen's inequality
     with respect to $(r-s)^{1/2} G_{r-s}^2(y)dy\es G_{\frac{r-s}{2}}(y)dy$ (since when $p>2$, the function $\phi(x)=x^{2/p}$, $x>0$,   is  concave),  and then we use     \eqref{Glamd} to obtain
     \begin{align}\label{CharEst1D2}
  {   D_2   }    \ls & \int_0^r \int_{\RR ^2} (r-s)^{-2\alpha} G_{r-s}^2(y)
          \bigg(\int_\RR \|\Delta_h v(s,y+z)\|_{L^p(\Omega)}^p \lambda (z) dz \bigg)^{2/p} |h|^{2H-2} dydh ds \nonumber\\
          \lesssim&  \int_{0}^{r} \int_{\RR}(r-s)^{-2\alpha-\frac 12} \bigg(\int_{\RR}\int_{\RR}G_{\frac {r-s}{2}}(y)\|\Delta_h v(s,z)\|_{L^p(\Omega)}^pdz  \cdot\lambda(z-y) dy\bigg)^{\frac 2p} |h|^{2H-2}dhds \nonumber\\
          \lesssim&  \int_{0}^{r}\int_{\RR}(r-s)^{-2\alpha-\frac 12} \|\Delta_h v(s,\cdot)\|_{L^p_{\lambda}(\Omega\times\RR)}^2|h|^{2H-2}dhds \,.
     \end{align}
Recall that 
   \[
    \|v\|_{\cZ_{\lambda,T}^p} :=\sup_{s\in[0,T]} \|v(s,\cdot)\|_{L^p_{\lambda}(\Omega\times\RR)}+\sup_{s\in[0,T]}\cN^*_{\frac 12-H,p}v(s)\,,
   \]
   where $\cN^*_{\frac 12-H,p}v(t)$ is defined in (\ref{N*pNorm}).  The estimates obtained in  (\ref{CharEst1D1}) and (\ref{CharEst1D2})   imply
   \begin{equation}
     \begin{split}
        \EE \|J_{\alpha}(r,\cdot)\|^p_{L_{\lambda}^p(\RR)}&\ls \|v\|_{\cZ_{\lambda,T}^p}^p\bigg(\int_{0}^{r} (r-s)^{-2\alpha-\frac 12}         +(r-s)^{-2\alpha+H-1}dr \bigg)^{\frac p2}\,.
     \end{split}
   \end{equation}
   If we have $-2\alpha+H-1>-1$ and $-2\alpha-\frac 12>-1$, i.e. $\alpha<\frac H2$, then (\ref{CharEst1}) follows.

  However, the condition $\alpha<H/2$ should be combined
  with \eqref{e.4.alpha1}. This gives $
    \frac{3}{2p}<\alpha<\frac H2$ which implies
    $p>\frac 3H$.  Thus,  under the  condition of the proposition, the  inequality \eqref{CharEst1} holds true. This  finishes  the proof of \textbf{(i)}.

\noindent\textbf {Step 3.} \ In this and next steps  we   prove
     \textbf{(ii).}   The spirit of the proof is   similar to that of the proof of \textbf{(i)} but is  more involved. In order to obtain the desired decay    rate of $\cN_{\frac 12-H}\Phi(t,x)$, we still use the equation  \eqref{e.4.12} to express $\Phi(t,x)$   by $J$.
     \begin{equation*}
       \begin{split}
          & \Phi(t,x+h)-\Phi(t,x)\\
          =&\frac{\sin(\pi\alpha)}{\pi}\int_{0}^{t}\int_{\RR}(t-r)^{\alpha-1} D_{t-r}(x-z,h) J_{\alpha}(r,z)dzdr\\
          =&\frac{\sin(\pi\alpha)}{\pi}\int_{0}^{t}\int_{\RR}(t-r)^{\alpha-1}G_{t-r}(x-z)\Delta_h J_{\alpha}(r,z) dzdr\,,
       \end{split}
     \end{equation*}
     where $\Delta_h J_{\alpha}(t,x):=J_{\alpha}(t,x+h)-J_{\alpha}(t,x)$.
     
   Invoking   Minkowski's inequality and then H\"{o}lder's  inequality with $\frac 1p + \frac 1q=1$  we get
     \begin{align*}
      &\int_{\RR}|\Phi(t,x+h)-\Phi(t,x)|^2|h|^{2H-2}dh\\
      \es  &\int_{\RR}\bigg|\int_{0}^{t}\int_{\RR}(t-r)^{\alpha-1}G_{t-r}(x-z)\Delta_h J_{\alpha}(r,z) dzdr\bigg|^2\cdot |h|^{2H-2}dh \\
      \ls&\bigg(\int_{0}^{t}\int_{\RR}(t-r)^{\alpha-1}G_{t-r}(x-z)\Big[\int_{\RR}|\Delta_h J_{\alpha}(r,z)|^2|h|^{2H-2}dh\Big]^{\frac 12}dzdr\bigg)^2   \\
         \ls & \left(\int_{0}^{t}\int_{\RR}(t-r)^{q(\alpha-1)}G_{t-r}^q (x-z) \lambda^{-\frac qp}(z) dzdr\right)^{\frac 2p}  \\
         &\qquad\qquad\times
         \lc\int_{0}^{T}\int_{\RR}\Blk\int_{\RR}|\Delta_h J_{\alpha}(r,z)|^2|h|^{2H-2}dh\Brk^{\frac p2}\lambda(z)dz dr\rc^{\frac 2p}  \\
          \ls& \lambda(x)^{-\frac 2p}\lk\int_{0}^{t}(t-r)^{q(\alpha-\frac 32+\frac{1}{2q})}dr\rk^{\frac 2q}\times \lc\int_{0}^{T}\int_{\RR}\Blk\int_{\RR}|\Delta_h J_{\alpha}(r,z)|^2|h|^{2H-2}dh\Brk^{\frac p2}\lambda(z)dz dr\rc^{\frac 2p}\,,
     \end{align*}
where in the above last inequality we used
$G_{t-r}^q (x-z) =(t-r)^{\frac{1-q}2 } G_{\frac{t-r} {q}}  (x-z)$ and inequality \eqref{Glamd}.     If we take $\theta=\frac 1p$, and $q(\alpha-\frac 32+\frac{1}{2q})>-1$, i.e.
\begin{equation}
\alpha>\frac{3}{2p}\,, \label{e.4.alpha2}
\end{equation}  then
     \begin{align}
       &\sup_{t,x}\lambda(x)^{\theta}\lc\int_{\RR}|\Phi(t,x+h)-\Phi(t,x)|^2|h|^{2H-2}dh\rc^{\frac 12} \nonumber \\
       \lesssim& \lc\int_{0}^{T}\int_{\RR}\Blk\int_{\RR}|\Delta_h J_{\alpha}(r,z)|^2|h|^{2H-2}dh\Brk^{\frac p2}\lambda(z)dz dr\rc^{\frac 1p}.\nonumber
     \end{align}
Thus, to prove part \textbf{(ii)}  we only need
  to prove that there exists some constant $C_1$, independent of $r\in [0, T]$,  such that
     \begin{equation}\label{CharEst2}
     \cI:=  \EE \int_{\RR}\Blk\int_{\RR}|\Delta_h J_{\alpha}(r,z)|^2|h|^{2H-2}dh\Brk^{\frac p2}\lambda(z)dz\leq C_1  \|v\|^p_{\cZ_{\lambda,T}^p}.
     \end{equation}

\noindent\textbf{Step 4}. \ In this step we   show the above inequality \eqref{CharEst2}.  By the definition \eqref{e.def_J} of $J$ and by an application of   Minkowski's inequality 
we have
     \begin{equation*}
       \begin{split}
          \cI
         \lesssim& \lc\int_{\RR}\Blk\int_{\RR}\EE|\Delta_h J_{\alpha}(r,z)|^p\lambda(z)dz\Brk^{\frac 2p}|h|^{2H-2}dh\rc^{\frac p2} \\
         \lesssim& \bigg(\int_{\RR} \Blk\int_{\RR}\EE\Blc\int_{0}^{r}\int_{\RR^2}(r-s)^{-2\alpha}\big|D_{r-s}(z-y-l,h)v(s,y+l)\\
            &\qquad-D_{r-s}(z-y,h)v(s,y)\big|^2|l|^{2H-2}dldyds\Brc^{\frac p2}\lambda(z)dz\Brk^{\frac 2p} |h|^{2H-2}dh\bigg)^{\frac p2}.
       \end{split}
     \end{equation*}
We introduce two notations:
     \begin{align*}
     \cI_1(r,z,h)&:= \EE \bigg(\int_{0}^{r}\int_{\RR^2}(r-s)^{-2\alpha}\big| D_{r-s}(z-y,h)\big|^2 \times\big|\Delta_{l} v(s,y) \big|^2 |l|^{2H-2}dldyds\bigg)^{\frac p2},
     \end{align*}
     and
     \begin{align*}
     \cI_2(r,z,h)&:= \EE \bigg(\int_{0}^{r}\int_{\RR^2} (r-s)^{-2\alpha}\Big|\Box_{r-s}(z-y,l,h)\Big|^2 \times\big|v(s,y)\big|^2|l|^{2H-2} dldyds\bigg)^{\frac p2}\,.
     \end{align*}
Then,  we have
     \begin{align*}
      &\EE \int_{\RR}\Blk\int_{\RR}|\Delta_h J_{\alpha}(r,z)|^2|h|^{2H-2}dh\Brk^{\frac p2}\lambda(z)dz\\
      \lesssim& \lc\int_{\RR} \Blk\int_{\RR}\cI_1(r,z,h)\lambda(z)dz\Brk^{\frac 2p}|h|^{2H-2}dh\rc^{\frac p2} \\
      &\qquad \qquad+\lc\int_{\RR} \Blk\int_{\RR}\cI_2(r,z,h)\lambda(z)dz\Brk^{\frac 2p}|h|^{2H-2}dh\rc^{\frac p2}
      =:I_1^{p/2}+I_2^{p/2}.
     \end{align*}
We shall bound $I_1$ and $I_2$ one by one.
For the first term, a change of variable $y\to z-y$ and an application of  Minkowski's inequality yield
     \begin{align}\label{CharEst2I1.0}
        I_1 \lesssim& \int_\RR \bigg|\int_\RR\EE \bigg(\int_{0}^{r}\int_{\RR^2}(r-s)^{-2\alpha}\big| D_{r-s}(y,h)\big|^2 \nonumber\\
          &\qquad\quad\times\big|\Delta_{l} v(s,y+z) \big|^2 |l|^{2H-2}dldyds\bigg)^{\frac p2}\lambda(z) dz\bigg|^{\frac 2p}|h|^{2H-2} dh \nonumber\\
          \lesssim& \int_{0}^{r}\int_{\RR^3}(r-s)^{-2\alpha}\big| D_{r-s}(y,h)\big|^2|l|^{2H-2}|h|^{2H-2} \nonumber\\
          &\qquad\qquad\qquad\times \lc\int_{\RR}\EE\big|\Delta_{l} v(s,z) \big|^p\lambda(z-y) dz\rc^{\frac 2p} dydhdlds.
     \end{align}
By     \eqref{FirDiffIneq} with $\beta=\frac12-H$ we see that
   \[
   \int_{\RR^2}\big| D_{r-s}(y,h)\big|^2 |h|^{2H-2}dhdy \lesssim (r-s)^{H-1}\,,
   \]
   which is finite.  Since $ x ^{2/p}, x>0$ is a concave function  for $p\ge 2$  we can apply Jensen's inequality with respect to the probability measure $(r-s)^{1-H}\blk G_{r-s}(y)-G_{r-s}(y+h)\brk^2 |h|^{2H-2}dydh$. Thus,   we have  for $p\geq 2$:
     \begin{align}\label{CharEst2I1}
       I_1  \lesssim& \int_{0}^{r}\int_{\RR}(r-s)^{-2\alpha+H-1}\bigg(\int_{\RR^3}(r-s)^{1-H}\big| D_{r-s}(y,h)\big|^2 \nonumber\\
          &\qquad|h|^{2H-2}\int_{\RR}\EE\big|\Delta_{l} v(s,z) \big|^p\lambda(z-y) dzdydh\bigg)^{\frac 2p}\times |l|^{2H-2}dlds\nonumber\\
         \lesssim& \int_{0}^{r}\int_{\RR}(r-s)^{-2\alpha+H-1}\|\Delta_{l} v(s,\cdot)\|^2_{L^p_{\lambda}(\Omega\times\RR)}|l|^{2H-2}dldr
     \end{align}
     by the first inequality in  Lemma \ref{LemDGLcontra}.

     In order to bound $\cI_2(t,x,h)$, we make a change of variable $y\to z-y$ and then split it to two terms.  More precisely, we have
     \begin{align}\label{CharEst2I2.0}
      &\cI_2(r,z,h)\lesssim \cI_{21}(r,z,h)+\cI_{22}(r,z,h) \\
      :=& \EE\lc \int_{0}^{r}\int_{\RR^2} (r-s)^{-2\alpha}|\Box_{r-s}(y,l,h)|^2 |v(s,z)|^2 |l|^{2H-2}dldyds\rc^{\frac p2} \nonumber\\
       +& \EE\lc \int_{0}^{r}\int_{\RR^2} (r-s)^{-2\alpha}|\Box_{r-s}(y,l,h)|^2 |\Delta_{y} v(s,z)|^2 |l|^{2H-2}dldyds\rc^{\frac p2}\,.\nonumber
     \end{align}
Using   Minkowski's inequality, Lemma \ref{NGreen},
 and Lemma \ref{LemBGcontra}, one can check that
     \begin{equation}\label{CharEst2I21}
       \begin{split}
          &I_{21}:=\int_{\RR} \Blk\int_{\RR}\cI_{21}(r,z,h)\lambda(z)dz\Brk^{\frac 2p}|h|^{2H-2}dh\\
         \lesssim& \int_{0}^{r}\int_{\RR^3} (r-s)^{-2\alpha}|\Box_{r-s}(y,l,h)|^2\|v(s)\|^2_{L^p_{\lambda}(\Omega\times\RR)}|l|^{2H-2}|h|^{2H-2}dldhdyds\\
         \lesssim& \int_{0}^{r} (r-s)^{-2\alpha+2H-\frac 32} \|v(s)\|^2_{L^p_{\lambda}(\Omega\times\RR)} ds,
       \end{split}
     \end{equation}
     and
     \begin{equation}\label{CharEst2I22}
       \begin{split}
           &I_{22}:=\int_{\RR} \Blk\int_{\RR}\cI_{22}(r,z,h)\lambda(z)dz\Brk^{\frac 2p}|h|^{2H-2}dh \\
         \lesssim& \int_{0}^{r}\int_{\RR} (r-s)^{-2\alpha}\int_{\RR}\lc\int_{\RR^2}|\Box_{r-s}(y,l,h)|^2 |l|^{2H-2}|h|^{2H-2}dldh\rc\cdot\|\Delta_{y} v(s,\cdot)\|^2_{L^p_{\lambda}(\Omega\times\RR)}dyds\\
         \lesssim& \int_{0}^{r}\int_{\RR} (r-s)^{-2\alpha+H-1} \|\Delta_{y} v(s,\cdot)\|^2_{L^p_{\lambda}(\Omega\times\RR)}|y|^{2H-2}dyds.
       \end{split}
     \end{equation}
Recalling the definition of $\|\cdot\|^p_{\cZ_{\lambda,T}^p}$, and combining  (\ref{CharEst2I1}), (\ref{CharEst2I21}) and (\ref{CharEst2I22}), we obtain
     \begin{equation}
     \begin{split}
        &\EE \int_{\RR}\Blk\int_{\RR}|\Delta_{h}J_{\alpha}(r,z)|^2|h|^{2H-2}dh\Brk^{\frac p2}\lambda(z)dz \\
        \leq&C_2\,\|v\|_{\cZ_{\lambda,T}^p}^p\bigg(\int_{0}^{r} (r-s)^{-2\alpha+2H-\frac 32} +(r-s)^{-2\alpha+H-1}dr \bigg)^{\frac p2}.
        \label{e.4.25a} 
     \end{split}
     \end{equation}
     Once we have $-2\alpha+2H-\frac 32>-1$ and $-2\alpha+H-1>-1$, i.e. $\alpha<H-\frac 14$, we see  that  (\ref{CharEst2}) follows
     from  \eqref{e.4.25a}. This condition on $\alpha$ 
     is combined with \eqref{e.4.alpha2} to become $\frac{3}{2p}<\alpha<H-\frac 14$.
     Therefore,  we have proved
      that if  $p>\frac{6}{4H-1}$,
     then (\ref{CharEst2})  holds,   finishing  the proof of \textbf{(ii)}.

\noindent\textbf{Step 5.} \ We are going to prove part      \textbf{(iii).}   We   continue to use \eqref{e.4.12}.
Without loss of generality,   we can assume $h>0$ and $t\in[0,T]$ such that $t+h\leq T$.   We have
     \begin{equation*}
       \begin{split}
          &\Phi(t+h,x)-\Phi(t,x) \\
          =&\frac{\sin(\pi\alpha)}{\pi} \bigg[\int_{0}^{t+h}\int_{\RR}  (t+h-r)^{\alpha-1}G_{t+h-r}(x-z) J_{\alpha}(r,z)drdz\\
          &\qquad\qquad-\int_{0}^t \int_{\RR} (t-r)^{\alpha-1}G_{t-r}(x-z) \times J_{\alpha}(r,z)drdz\bigg]\\
          \lesssim&\sum_{i=1}^{3} \cJ_{i}(t,h,x),
       \end{split}
     \end{equation*}
     where
     \[
      \cJ_1(t,h,x):=\int_{0}^{t}\int_{\RR}\blk(t+h-r)^{\alpha-1}-(t-r)^{\alpha-1}\brk G_{t-r}(x-z) J_{\alpha}(r,z)drdz,
     \]
     \[
      \cJ_2(t,h,x):=\int_{0}^{t}\int_{\RR}(t+h-r)^{\alpha-1} \blk G_{t+h-r}(x-z)-G_{t-r}(x-z)\brk J_{\alpha}(r,z)drdz,
     \]
     and
     \[
      \cJ_3(t,h,x):=\int_{t}^{t+h}\int_{\RR}(t+h-r)^{\alpha-1}  G_{t+h-r}(x-z) J_{\alpha}(r,z)drdz.
     \]
As in  the  proof of \textbf{(i)} and \textbf{(ii)}, we insert additional factors of $\lambda^{-\frac 1p}(z)\cdot\lambda^{\frac 1p}(z)$ 
and apply H\"{o}lder's  inequality in the expression for $\cJ_1$.  Then, $\cJ_1$ is estimated as follows.
     \begin{equation}
       \begin{split}
          \cJ_1(t,h,x)\leq&\lambda^{-\frac 1p}(x) \int_{0}^{t} \big|(t+h-r)^{\alpha-1}-(t-r)^{\alpha-1}\big|(t-r)^{\frac{1-q}{2q}}\|J_{\alpha}(r,\cdot)\|_{L_{\lambda}^p(\RR)}dr\\
          \leq& \lambda^{-\frac 1p}(x) \lc\int_{0}^{t} \big|(t+h-r)^{\alpha-1}-(t-r)^{\alpha-1}\big|^q(t-r)^{\frac{1-q}{2}}dr\rc^{\frac 1q} \\
          &\quad\quad\times \lc\int_{0}^{T}\|J_{\alpha}(r,\cdot)\|^p_{L_{\lambda}^p(\RR)}dr\rc^{\frac 1p}.
       \end{split}\label{e.4.26aa} 
     \end{equation}
     Fix $\gamma\in(0,1)$. It is easy to see
\begin{equation}
\big|(t+h-r)^{\alpha-1}-(t-r)^{\alpha-1}\big|\lesssim \big|t-r\big|^{\alpha-1-\gamma}h^{\gamma}.\label{e.simple_difference}
     \end{equation}
Thus, we have
     \begin{equation*}
     \begin{split}
         \sup_{t,x} \lambda^{1/p}(x)|\cJ_1(t,h,x)|\lesssim& h^{\gamma}
          \sup_{t\in[0,T]}\lc\int_{0}^{t} (t-r)^{q(\alpha-1-\gamma)+\frac{1-q}{2}}dr\rc^{\frac 1q} \\
          &\qquad\quad \times
          \lc\int_{0}^{T}\|J_{\alpha}(r,\cdot)\|^p_{L_{\lambda}^p(\RR)}dr\rc^{\frac 1p}.
     \end{split}
     \end{equation*}
In other word,  if $\gamma+\frac{3}{2p}<\alpha<\frac H2$ or equivalently,  if $\gamma<\frac H2-\frac{3}{2p}$,   then  we have
     \begin{equation}\label{CharEst3J1}
       \EE \left|\sup_{t,x} \lambda^{\theta}(x)|\cJ_1(t,h,x)|\right|^p\lesssim |h|^{p\gamma}\|v\|^p_{\cZ_{\lambda,T}^p}.
     \end{equation}
  Let us proceed to bound $\cJ_2(t,h,x)$. One   finds easily
\begin{eqnarray}
 \cJ_2(t,h,x)
 & \leq& \bigg( \int_{0}^{t}\int_{\RR} (t+h-r)^{q(\alpha-1)}\big|G_{t+h-r}(x-z)\label{e.4.J2}\\
          &&  -G_{t-r}(x-z)\big|^q \lambda^{-\frac qp}(z)dz dr\bigg)^{\frac 1q}   \bigg(\int_{0}^{T}\|J_{\alpha}(r,\cdot)\|^p_{L_{\lambda}^p(\RR)}dr\bigg)^{\frac 1p}. \nonumber
       \end{eqnarray}
To bound the above first factor  we use the following  inequality
     \[
      \left|\exp\lc-\frac{x^2}{t+h}\rc-\exp\lc-\frac{x^2}{t}\rc\right|\leq C_{\gamma} h^{\gamma}t^{-\gamma}\exp\lc-\frac{\gamma x^2}{2(t+h)}\rc\quad \forall \ \gamma\in (0, 1)\,.
     \]
Combining  the above inequality with  \eqref{e.simple_difference} (with $\al=1/2$), we have
\begin{equation}
\begin{split}
&|G_{t+h-r}(x-z)-G_{t-r}(x-z)|
\\
\le &C_{\gamma} h^{\gamma}(t-r)^{-\gamma}
\left[G_{\frac{2}{\gamma}(t+h-r)}(x-z)+ { G_{\frac{2}{\gamma}(t-r)}(x-z)}\right]\,.
\end{split} \label{e.4.G_difference}
   \end{equation}
Thus,   the first factor in \eqref{e.4.J2} is bounded by
     \begin{equation*}
       \begin{split}
          &\int_{0}^{t}\int_{\RR} (t+h-r)^{q(\alpha-1)}\big|G_{t+h-r}(x-z)-G_{t-r}(x-z)\big|^q \lambda^{-\frac qp}(z)dz dr \\
         \lesssim& h^{q\gamma}\int_{0}^{t}\int_{\RR}(t-r)^{q(\alpha-1-\gamma)+\frac{1-q}{2}}G_{\frac{ 2(t+h-r) }{\gamma q}}(x-z)\lambda^{-\frac qp}(z)dzdr \\
         +& h^{q\gamma}\int_{0}^{t}\int_{\RR}(t-r)^{q(\alpha-1-\gamma)+\frac{1-q}{2}}G_{\frac{2(t-r)}{\gamma q}}(x-z)\lambda^{-\frac qp}(z)dzdr \\
         \lesssim& h^{q\gamma} \lambda^{-\frac qp}(x) \int_{0}^{t}(t-r)^{q(\alpha-1-\gamma)+\frac{1-q}{2}}dr\,,
       \end{split}
     \end{equation*}
where the last inequality follows from Lemma \ref{TechLemma1}.
Hence, if  $\gamma+\frac{3}{2p}<\alpha<\frac H2$, namely, if  $\gamma<\frac H2-\frac{3}{2p}$, then we have the following estimation:
     \begin{equation}\label{CharEst3J2}
       \EE \left|\sup_{t,x} \lambda^{\theta}(x)|\cJ_2(t,h,x)|\right|^p\lesssim |h|^{p\gamma}\|v\|^p_{\cZ_{\lambda,T}^p}.
     \end{equation}

Now we are going to bound $\cJ_3(t,x,h)$. {   Exactly in the same way as 
for  \eqref{e.4.26aa},  we have }
     \begin{equation*}
       \begin{split}
  \cJ_3(t,x,h)\le& \lambda^{-\frac 1p}(x) \lc\int_{t}^{t+h}(t+h-r)^{q(\alpha-1)}
  (t+h-r)^{ \frac{1-q}{2}}dr\rc^{\frac 1q}\\
  &\quad \lc\int_{0}^{T}\|J_{\alpha}(r,\cdot)\|^p_{L_{\lambda}^p(\RR)}dr\rc^{\frac 1p}\\
  =& C_p \lambda^{-\frac 1p}(x)  h^{\alpha -\frac{3}{2p}}\lc\int_{0}^{T}\|J_{\alpha}(r,\cdot)\|^p_{L_{\lambda}^p(\RR)}dr\rc^{\frac 1p}\,.  
       \end{split}
     \end{equation*}
If  $\frac{3}{2p}<\alpha<\frac H2$,  which is possible if  $\gamma<\alpha-\frac{3}{2p}<\frac H2-\frac{3}{2p}$,  then
     \begin{equation}\label{CharEst3J3}
       \EE \left|\sup_{t,x} \lambda^{\theta}(x)|\cJ_3(t,h,x)|\right|^p\leq C_3 |h|^{p\alpha-\frac{3}{2}}\|v\|_{\cZ_{\lambda,T}^p}=C_3\,|h|^{p\gamma}\|v\|^p_{\cZ_{\lambda,T}^p}\,.
     \end{equation}
Combining  \eqref{CharEst3J1}, \eqref{CharEst3J2} and \eqref{CharEst3J3}  we prove \eqref{EstZ3}.

 \noindent\textbf {Step 6.} \ We prove   part  \textbf{(iv)} of the proposition.  As before, we shall again use the representation formula
 \eqref{e.4.12}  and then we apply the H\"{o}lder inequality to find
     \begin{equation*}
       \begin{split}
          &\Phi(t,x)-\Phi(t,y) \\
          =&\frac{\sin(\pi\alpha)}{\pi}\int_{0}^{t}\int_{\RR} (t-r)^{\alpha-1}\blk G_{t-r}(x-z)-G_{t-r}(y-z)\brk J_{\alpha}(r,z) dzdr \\
          \lesssim&\lc\int_{0}^{t}\int_{\RR} (t-r)^{q(\alpha-1)}|G_{t-r}(x-z)-G_{t-r}(y-z)|^{q} \lambda^{-\frac qp}(z)dzdr\rc^{\frac 1q} \\
          &\times\lc\int_{0}^{T} \int_{\RR}|J_{\alpha}(r,z)|^{p} \lambda(z)dzdr\rc^{\frac 1p}.
       \end{split}
     \end{equation*}
Denote the above first factor by 
     \[
      \cK(t,x,y):=\int_{0}^{t}\int_{\RR} (t-r)^{q(\alpha-1)}|G_{t-r}(x-z)-G_{t-r}(y-z)|^{q} \lambda^{-\frac qp}(z)dzdr.
     \]
Fix  $\gamma\in(0,1)$.
{  Using H\"{o}lder's  inequality     we have
     \begin{equation}
       \begin{split}
          &\cK(t,x,y) \\
         \lesssim& \int_{0}^{t} (t-r)^{q(\alpha-1)} \lc\int_{\RR}|G_{t-r}(x-z)-G_{t-r}(y-z)|^{pq(1-\gamma)}\lambda^{-q}(z)dz\rc^{\frac 1p}\\
         &\qquad\qquad\qquad\times \lc\int_{\RR}|G_{t-r}(x-z)-G_{t-r}(y-z)|^{q^2 \gamma}dz\rc^{\frac 1q} dr  \,.  
       \end{split}\label{e.4.33aa} 
     \end{equation} 
  To bound the integral inside the above second bracket, we make the substitutions  $\tilde{x}=\frac{x}{\sqrt{t-r}}$, $\tilde{y}=\frac{y}{\sqrt{t-r}}$ and $\tilde{z}=\frac{z}{\sqrt{t-r}}$ to obtain for any    {$\rho>0$},  
\begin{align*}
  &\int_{\RR}|G_{t-r}(x-z)-G_{t-r}(y-z)|^{\rho }dz \\
  \simeq& (t-r)^{\frac{1-\rho}{2}}\int_{\RR} |\exp(-|\tilde{x}-\tilde{z}|^2)-\exp(-|\tilde{y}-\tilde{z}|^2)|^{\rho}d\tilde{z} \\
  \ls&(t-r)^{\frac{1-\rho}{2}}|\tilde{x}-\tilde{y}|^{\rho}= (t-r)^{\frac{1-2\rho}{2}}|x-y|^{\rho} \,. 
\end{align*}
Substituting this  bound into \eqref{e.4.33aa}
with $\rho=q^2 r$   we have
\begin{equation*}
       \begin{split}
          &\cK(t,x,y) \\
         \lesssim& |x-y|^{q\gamma}\cdot\int_{0}^{t} (t-r)^{q(\alpha-1)+\frac{1-pq(1-\gamma)}{2p}+\frac{1-2q^2 \gamma}{2q}} \\
         &\qquad\qquad\qquad\times\lc\int_{\RR}\blk G_{\frac{t-r}{pq(1-\gamma)}}(x-z)+G_{\frac{t-r}{pq(1-\gamma)}}(y-z)\brk \lambda^{-q}(z)dz\rc^{\frac 1p} dr \\
         \lesssim& |x-y|^{q\gamma}\cdot \blk\lambda^{-\frac qp}(x)+\lambda^{-\frac qp}(y)\brk \cdot\int_{0}^{t} (t-r)^{q(\alpha-\frac 32+\frac{1}{2q})-\frac{q\gamma}{2}} dr\,,
       \end{split}
     \end{equation*}
 where the last inequality follows from Lemma \ref{TechLemma1}. 
}

If $q(\alpha-\frac 32+\frac{1}{2q})-\frac{q\gamma}{2}>-1$ and $\alpha<\frac{H}{2}$, namely,  if 
$\frac{3}{2p}+\frac{\gamma}{2}<\alpha<\frac{H}{2}$, then with $\theta=\frac 1p$ we have
     \begin{align}\label{CharEst4K}
        &\EE \bigg|\sup\limits_{\substack{t\in[0,T] \\ x,y\in \RR}}\lt(\lambda^{-\theta}(x)+\lambda^{-\theta}(y)\rt)^{-1}|\cK(t,x,y)|^{\frac 1q}\times\Blc\int_{0}^{T} \int_{\RR}|J_{\alpha}(r,z)|^{p}\lambda(z)dzdr\Brc^{\frac 1p}\bigg|^p \nonumber\\
       \lesssim& |x-y|^{p\gamma}\cdot\int_{0}^{T} \int_{\RR}\EE|J_{\alpha}(r,z)|^{p}\lambda(z)dzdr\leq C_4 |x-y|^{p\gamma}\|v\|^p_{\cZ_{\lambda,T}^p}\,.
     \end{align}
This proves   \eqref{EstZ4}.  The proof of the proposition is then completed.
 \end{proof}

\subsection{Weak existence of the solution}
In this subsection  we   show the weak existence of a solution with paths in $\cC([0, T]\times\RR)$, the space of all continuous real valued functions on $[0, T]\times\RR$, equipped with a metric
\begin{equation}
  d_{\cC}(u,v):=\sum_{n=1}^{\infty}\frac{1}{2^n} \max_{0\le t\le T,|x|\leq n}(|u(t,x)-v(t,x)|\wedge 1).\label{e.4.metric}
 \end{equation}
We state   a tightness criterion  of probability measures on $(\cC([0, T]\times\RR),\sB(\cC([0, T]\times\RR)))$   that we are going to use  (see  Section 2.4 in \cite{KS1998} for  the  case where $[0, T]\times\RR$ is replaced by $[0,\infty)$. It is also true for our case  as indicated there).
 \begin{theorem}\label{TightMoment}
   A sequence  $\{\bP_n\}_{n=1}^{\infty}$ of probability measures on $(\cC([0, T]\times\RR),\sB(\cC([0, T]\times\RR)))$ is tight if and only if
   \begin{enumerate}[leftmargin=*]
     \item[(1)] $\lim_{\lambda\uparrow\infty} \sup_{n\geq 1} \bP_n
     \left( \left\{\omega\in \cC([0, T]\times\RR):\ |\omega(0,0)|>\lambda\right\}\right) =0$,
     \item[(2)]  for any $T>0$, $R>0$ and $\varepsilon>0$
    \[
     \lim_{\delta\downarrow 0} \sup_{n\geq 1} \bP_n
     \left( \left\{\omega\in \cC([0, T]\times\RR):\ m^{T,R}(\omega,\delta)>\varepsilon\right\}\right) =0
    \]
    where
    \[
    m^{T,R}(\omega,\delta):=\max\limits_{\substack{|t-s|+|x-y|\leq \delta\\0\leq t,s\leq T;0\leq |x|,|y|\leq R}} |\omega(t,x)-\omega(s,y)|
    \]
     is the modulus of continuity on $[0,T]\times[-R,R]$.
   \end{enumerate}
 \end{theorem}
We approximate  the noise $ W$ with respect to the space variable by the following smoothing of the noise.
 That   is, for $\ep>0$ we define
 \begin{equation}\label{Regu}
     \frac{\partial }{\partial x} W_{\ep}(t,x) = \int_{\RR}  G_{\ep} (x-y)W(t,dy)\,.
 \end{equation}
The noise $W_\ep$ induces an approximation to mild solution
 \begin{equation}\label{MildSolRegu}
   u_\ep(t,x)=G_t\ast u_0(x)+\int_{0}^{t}\int_{\RR} G_{t-s}(x-y)\sigma(s,y,u_\ep(s,y))W_{\ep}(ds,dy),
 \end{equation}
 where the stochastic
 integral is understood in the It\^{o} sense.
 As in \cite{HHLNT2017} due to the regularity in space, the existence and uniqueness of the solution $u_\ep(t,x)$ to above equation is well-known.

 The lemma below asserts that the approximate solution
 $u_\ep(t,x)$  is uniformly bounded in the space $\cZ_{\lambda,T}^p$.  More  precisely, we have

 \begin{lemma}\label{UniBExist}
 Let $H\in(\frac 14,\frac 12)$ and let $\lambda(x)$ be defined by  \eqref{lamd}.
   Assume $\sigma(t,x,u)$ satisfies hypothesis \textbf{(H1)}.
Assume  also that  the initial value $u_0(x)\in\cZ_{\lambda,0}^p$.
   Then the approximate solutions $u_\ep$ satisfy
   \begin{equation}\label{ReguBdd}
     \sup_{\ep>0}\|u_\ep\|_{\cZ_{\lambda,T}^p}:=\sup_{\ep>0}\sup_{t\in[0,T]}\|u_\ep(t,\cdot)\|_{L^p_{\lambda}(\Omega\times\RR)} +\sup_{\ep>0}\sup_{t\in[0,T]}\cN^*_{\frac 12-H,p}u_\ep(t)<\infty.
   \end{equation}
 \end{lemma}
\begin{proof}
For notational simplicity  we   assume $\sigma(t,x,u)=\sigma(u)$ without loss of generality because of hypothesis \textbf{(H1)}.  We shall use some   similar thoughts   to    that in \cite{HHLNT2017} but now with   special attention to the weight $\lambda(x)$.
To this end, we define the Picard iteration  as follows: 
   \[
    u_\ep^0(t,x)=G_t\ast u_0(x)\,,
   \]
   and recursively  for $n=0, 1, 2, \cdots$,
\begin{equation}
    u_\ep^{n+1}(t,x)=G_t\ast u_0(x)+\int_{0}^{t}\int_{\RR} G_{t-s}(x-y)\sigma(u_\ep^{n}(s,y))W_{\ep}(ds,dy)\,. \label{e.4.39} 
   \end{equation} 
From \cite[Lemma 4.12]{HHLNT2019}  it follows that  for  any fixed $\ep>0$  
 when $n$ goes to infinity,  the  sequence $u_{\ep}^n(t,x)$ converges to $u_{\ep}(t,x)$ a.s. 
{ In the following steps 1 and 2, we shall first  bound $\|u_\ep^n\|_{\cZ_{\lambda,T}^p}$ uniformly in $n$,    and $\ep$.
Then,  in step 3 we   
use Fatou's lemma to show \eqref{ReguBdd}. 
}

   \noindent{\textbf{Step 1. }}\
   In this step,     we   derive a Gronwall-type  inequality to bound  the $L^p_{\lambda}(\Omega\times\RR)$ norm  of $u^{n+1}_\ep(t,x)$ by the   $\cZ_{\lambda,T}^p$  norm  of $u^{n}_\ep(t,x)$.  
    Rewrite \eqref{e.4.39} as  
   \[
    u^{n+1}_\ep(t,x)=G_t\ast u_0(x)+\int_{0}^{t}\int_{\RR} \Blk\Blc G_{t-s}(x-\cdot)\sigma(u^{n}_\ep(s,\cdot))\Brc\ast G_\ep\Brk(y)W(ds,dy)\,.
   \]
   In the following, we will continue to  use the notations $D_t(x,h)$ and $\Box_{t-s}(x,y,h)$ defined in \eqref{FirDiff} and \eqref{SecDiff} previously. Applying  the Burkholder-Davis-Gundy  inequality (Proposition \ref{HBDG}) and the isometry equalities \eqref{hinner.1}-\eqref{hinner.3}  
 and  then noting $|\sigma(u)|\lesssim |u|+1$,  we have
   \begin{align}\label{uepLp}
        &\EE\blk|u^{n+1}_\ep(t,x)|^p\brk \nonumber\\
     \leq& C_p |G_t\ast u^{n}_0(x)|^p
       +C_p\EE\lc\int_{0}^{t}\int_\RR \Big|\cF\blk G_{t-s}(x-\cdot)\sigma(u_\ep(s,\cdot)) \brk (\xi)\Big|^2  e^{-\varepsilon |\xi|^2}| \xi|^{1-2H} d\xi ds\rc^{\frac p2}  \nonumber\\
       \leq& C_p|G_t\ast u_0(x)|^p
       + C_p\EE\bigg(\int_{0}^{t}\int_{\RR^2}  \Big|G_{t-s}(x-y-h)\sigma(u^{n}_\ep(s,y+h)) \nonumber\\
       &\qquad\qquad\qquad\qquad\qquad\qquad-G_{t-s}(x-y)\sigma(u^{n}_\ep(s,y))\Big|^2 |h|^{2H-2}dhdyds\bigg)^{\frac p2} \nonumber\\
       \leq & C_p\lc |G_t\ast u_0(x)|^p+\cD^{\ep,n}_1(t,x)+\cD^{\ep,n}_2(t,x)\rc\,,
   \end{align}
  where the constant  $C_p$  
is independent  of $\ep$ because $e^{-\varepsilon |\xi|^2}\leq 1$,  
and where we denote 
   \begin{align*}
     \cD^{\ep,n}_1(t,x) &:=\bigg(\int_{0}^{t}\int_{\RR^2} \big|D_{t-s}(y,h)\big|^2 \left(1+\|u^{n}_\ep(s,x+y )\|_{L^p(\Omega)}^2\right) |h|^{2H-2}dhdyds\bigg)^{\frac p2} \,, 
\end{align*}
   and $\Delta_h u_{\ep}^n(t,x):=u_{\ep}^n(t,x+h)-u_{\ep}^n(t,x)$,
   \begin{align*}
     \cD^{\ep,n}_2(t,x) &:=\bigg( \int_{0}^{t}\int_{\RR^2}|G_{t-s}(y)|^2\|\Delta_h u_{\ep}^n(t,x+y)\|_{L^p(\Omega)}^2
             |h|^{2H-2} dhdyds\bigg)^{\frac p2}\,.
   \end{align*}
   This means
   \begin{align}
      \|u^{n+1}_{\ep}(t,\cdot)\|_{L^p_{\lambda}(\Omega\times\RR)}^2 =& \lc \int_\RR \EE\blk|u^{n}_\ep(t,x)|^p\brk \lambda(x) dx \rc^{\frac 2p} \nonumber\\
     \leq& C_p\lc \|u_0(x)\|_{L^p_{\lambda}(\RR)}+ I^{\ep,n}_1+I^{\ep,n}_2\rc\,,
     \label{e.4.38}
   \end{align}
where $I^{\ep,n}_1$ and $I^{\ep,n}_2$ are defined and bounded as follows.
\begin{equation}\label{D1Bdd}
     \begin{split}
 	I^{\ep,n}_1:= \lc\int_\RR \cD^{\ep,n}_1(t,x) \lambda(x)dx \rc^{\frac 2p}\leq& C_{p,H}\int_{0}^{t} (t-s)^{H-1}\Blc1+\|u^n_{\ep}(s,\cdot)\|^2_{L^p_{\lambda}(\Omega\times\RR)}\Brc ds\,,
     \end{split}
\end{equation}
{  and  }
\begin{equation}\label{D1Bdd2}
    \begin{split}
    I^{\ep,n}_2:= &\lc\int_\RR \cD^{\ep,n}_2(t,x) \lambda(x)dx \rc^{\frac 2p}\leq C_{p,H} \int_{0}^{t}\frac{\lt[\cN^*_{\frac 12-H,p}u^n_\ep(s)\rt]^2}{\sqrt{t-s}} ds\,.
    \end{split}
\end{equation} 
The  above bounds on $I_1^{\ep, n} , I_2^{\ep, n}$ together with \eqref{e.4.38} yield  
\begin{align}
  \|u^{n+1}_{\ep}(t,\cdot)\|_{L^p_{\lambda}(\Omega\times\RR)}^2
 &\leq C_{p,H}\bigg( \|u_0\|^2_{L^p_{\lambda}(\omega\times\RR)}+ \int_{0}^{t} (t-s)^{H-1}\  \|u^n_{\ep}(s,\cdot)\|^2_{L^p_{\lambda}(\Omega\times\RR)}  ds \nonumber\\
 &\qquad\qquad +\int_{0}^{t}(t-s)^{-1/2} \lt[\cN^*_{\frac 12-H,p}u^n_\ep(s)\rt]^2  ds\bigg)\,.
     \label{e.4.42}
\end{align}

\noindent\textbf{Step 2.}\
Next, we obtain a bound for  $\cN^*_{\frac 12-H,p}u^{n+1}_\ep(t)$ analogous to \eqref{e.4.42}.  Similar to \eqref{uepLp} we have
   \begin{align*}
      &\EE\blk|u^{n+1}_\ep(t,x)-u^{n+1}_\ep(t,x+h)|^p\brk \\
     \leq& C_p \big|G_t\ast u_0(x)-G_t\ast u_0(x+h)\big|^p \\
       &+C_p\EE\bigg( \int_{0}^{t}\int_{\RR^2} \Big|D_{t-s}(x-y-z,h)\sigma(u^{n}_\ep(s,y+z)) \\
       &\qquad\qquad\quad -D_{t-s}(x-z,h)\sigma(u^{n}_\ep(s,z))\Big|^2|y|^{2H-2} dzdyds\bigg)^{\frac p2} \\
     \leq& C_p\lc \cI_0(t,x,h)+\cI^{\ep,n}_1(t,x,h)+\cI^{\ep,n}_2(t,x,h)\rc\,,
   \end{align*}
   where
   \[
    \cI_0(t,x,h):=\big|G_t\ast u_0(x)-G_t\ast u_0(x+h)\big|^p\,,
   \]
   \begin{align*}
     \cI^{\ep,n}_1(t,x,h)&:= \EE \bigg(\int_{0}^{t}\int_{\RR^2}\Big| D_{t-s}(x-y-z,h)\Big|^2 \big|\sigma(u_\ep(s,y+z))-\sigma(u_\ep(s,z)) \big|^2 |y|^{2H-2}dzdyds\bigg)^{\frac p2}\,,
   \end{align*}
   \begin{align*}
     \cI^{\ep,n}_2(t,x,h)&:= \EE \bigg(\int_{0}^{t}\int_{\RR^2} \Big|\Box_{t-s}(x-z,y,h)\Big|^2 \times\big|\sigma(u_\ep(s,z))\big|^2|y|^{2H-2} dzdyds\bigg)^{\frac p2}\,.
   \end{align*}
By   Minkowski's  inequality we have
   \begin{align}
     \lk\cN^*_{\frac 12-H,p}u^{n+1}_\ep(t)\rk^2
     &\leq C_{p} \sum_{j=0}^{2}\int_{\RR}\lc \int_\RR \cI^{\ep,n}_j(t,x,h) \lambda(x) dx\rc^{\frac 2p} |h|^{2H-2} dh\nonumber\\
     &{  =:J_0+J_1+J_2. }   \label{e.4.43}
   \end{align} 
  Our strategy is to control the  above  three quantities by using the
   ideas similar  to  those when we deal  with  the terms $\cI_1$ and $\cI_{2}$   
  in the step 4 of the  proof of Proposition \ref{CharProp} \textbf{(ii)}.  First, from Lemma \ref{TechLemma1} 
  it follows. 
   \begin{equation}\label{I0Bdd}
     \begin{split}
   J_0  \leq& C_p\int_\RR \lc \int_{\RR} \lk\int_\RR G_t(x-y) \lambda(x)dx\rk \lt| \Delta_h u_0(y)\rt|^p dy  \rc^{\frac 2p} |h|^{2H-2} dh \\
     \leq& C_p\int_\RR \lc \int_{\RR} \lt|\Delta_h u_0(y)\rt|^p \lambda(y) dy  \rc^{\frac 2p} |h|^{2H-2} dh=C_p\lk\cN^*_{\frac 12-H,p}u_0\rk^2\,.
     \end{split}
   \end{equation}
   For the   term $J_1$, we can use the   method similar to   that when we obtain \eqref{CharEst2I1.0} and \eqref{CharEst2I1}. This is, a change of variable $y \to z-y$,  
    and   applications of Minkowski's inequality, Jensen's inequality and Lemma \ref{LemDGLcontra} 
   give 
   \begin{equation}\label{I1Bdd}
     \begin{split}
 {   J_1}   \leq& C_{p,H}\int_{0}^{t}\int_\RR (t-s)^{H-1} \bigg(\int_{\RR^3} (t-s)^{1-H}\Big| D_{t-s}(z,h)\Big|^2 |h|^{2H-2} \\
      &\qquad\qquad \times \EE\Blk\big|\Delta_y u_{\ep}^n(t,x)\big|^p\Brk\lambda(x-z) dxdzdh\bigg)^{\frac 2p} |y|^{2H-2} dyds \\
      \leq& C_{p,H}\int_{0}^{t} (t-s)^{H-1}\lk\cN^*_{\frac 12-H,p}u^{n}_{\ep}(s)\rk^2 ds\,.
     \end{split}
   \end{equation} 
Next,  we obtain a bound for $J_2$.  Similar to the obtention of \eqref{CharEst2I2.0}, \eqref{CharEst2I21} and \eqref{CharEst2I22} we  also make a  change of variable 
$y\to z-y$,
and  then split it to two terms to obtain 
   \begin{align*}
     &\cI^{\ep,n}_2(t,x,h)\leq C_{p}\lc \cI^{\ep,n}_{21}(t,x,h)+\cI^{\ep,n}_{22}(t,x,h)\rc \\
     :=& C_{p}\EE\lc \int_{0}^{t}\int_{\RR^2} |\Box_{t-s}(y,z,h)|^2 |\sigma(u^{n}_\ep(s,x))|^2 |y|^{2H-2}dydzds\rc^{\frac p2} \\
     +& C_{p}\EE\lc \int_{0}^{t}\int_{\RR^2} |\Box_{t-s}(y,z,h)|^2|\sigma(u^n_\ep(s,x+z))-\sigma(u^n_\ep(s,x))|^2 |y|^{2H-2}dydzds\rc^{\frac p2}  \,.
   \end{align*}
   Applying Minkowski's inequality,   the condition   $|\sigma(u)|\lesssim |u|+1$,   and Lemma \ref{NGreen}  one  has
   \begin{equation}\label{I21Bdd}
     \begin{split}
   J_{21}:=   &\int_{\RR}\lt| \int_\RR \cI^{\ep,n}_{31}(t,x,h) \lambda(x) dx\rt|^{\frac 2p} |h|^{2H-2} dh \\
     \leq& C_{p,H} \int_{0}^{t} (t-s)^{2H-\frac 32}\lc 1+\|u^n_\ep(s,\cdot)\|_{L^p_{\lambda}(\Omega\times\RR)}^{2}\rc ds\,.
     \end{split}
   \end{equation}
   Again by Minkowski's inequality,  the Lipschitz condition \eqref{LipSigam} on $\sigma$, 
and  Lemma \ref{LemBGcontra} 
we obtain
   \begin{equation}\label{I22Bdd}
     \begin{split}
    J_{22}:=  &\int_{\RR}\lt| \int_\RR \cI^{\ep,n}_{32}(t,x,h) \lambda(x) dx\rt|^{\frac 2p} |h|^{2H-2} dh \\
    \leq & C_{p,H} \int_{0}^{t} (t-s)^{H-1}\lk\cN^*_{\frac 12-H,p}u^n_{\ep}(s)\rk^2 ds\,. 
     \end{split}
   \end{equation} 
{Using that fact that $J_3\le J_{31}+J_{32}$ and 
using \eqref{e.4.43}-\eqref{I22Bdd}}     we obtain 
   \begin{equation}\label{e.4.48}
   \begin{split}
     \lk\cN^*_{\frac 12-H,p}u^{n+1}_\ep(t)\rk^2  \leq& C_{p,H} \lk\cN^*_{\frac 12-H,p}u_0\rk^2+C_{p,H}\int_{0}^{t} (t-s)^{H-1}\lk\cN^*_{\frac 12-H,p}u^n_{\ep}(s)\rk^2 ds\\
     +& C_{p,H}\int_{0}^{t} (t-s)^{2H-\frac 32}\lc 1+\|u^n_\ep(s,\cdot)\|_{L^p_{\lambda}(\Omega\times\RR)}^{2}\rc ds\,.
     \end{split}
   \end{equation}

\noindent\textbf{Step 3.}\
Set
   \[
    \Psi_{\ep}^{n}(t):=\|u^{n}_\ep(t,\cdot)\|_{L^p_{\lambda}(\Omega\times\RR)}^{2}+\lk\cN^*_{\frac 12-H,p}u^{n}_\ep(t)\rk^2.
   \]
   Thus,  {  combining all the estimates (\ref{I0Bdd}), (\ref{I1Bdd}), (\ref{I21Bdd}) and (\ref{I22Bdd}) yields}  
   \begin{align*}
     \Psi_{\ep}^{n+1}(t)&\leq C_{p,H,T}\lc \|u_0\|^2_{L^p_{\lambda}(\Omega\times\RR)} +\lk\cN^*_{\frac 12-H,p}u_0\rk^2+\int_{0}^{t} (t-s)^{2H-\frac 32} \Psi_{\ep}^{n}(s) ds\rc\,.
   \end{align*}
Now  it is relatively 
 easy to see by fractional Gronwall lemma (e.g.   \cite[Lemma 1]{LHH2020})
   \[
    \sup_{n\geq 1}\sup_{t\in[0,T]}\Psi_{\ep}^{n}(t)\leq C_{T,p,H}<\infty\,.
   \] 
For any fixed $\ep>0$   since $u^n_\ep$ converges to $u_{\ep}$  a.s. as $n\to\infty$, we have by Fatou's lemma
   \begin{align*}
     \|u_{ {\ep}}(t,\cdot)\|_{L_{\lambda}^p(\Omega\times\RR)} &=\lc\int_{\RR} \EE\lk\lim_{n\rightarrow \infty}|u^n_\ep(t,x)|^p\rk \lambda(x)dx\rc^{\frac{1}{p}} \\
       &\leq \varliminf \limits_{n\rightarrow \infty}\lc\int_{\RR} \EE\lk|u^n_\ep(t,x)|^p\rk \lambda(x)dx\rc^{\frac{1}{p}}\leq  \sup_{n\geq 1}\sup_{t\in[0,T]}\Psi_{\ep}^{n}(t)<\infty\,.
   \end{align*}
   Thus,   we   conclude that $\sup\limits_{\ep>0}\sup\limits_{t\in[0,T]} \|u_{\ep}(t,\cdot)\|_{L_{\lambda}^p(\Omega\times\RR)}$ is finite. 
 {   On the other hand, for  any $t, x$ and $h$ we have $|u^n_\ep(t,x+h)-u^n_\ep(t,x)|^2\rightarrow |u_{\ep}(t,x+h)-u_{\ep}(t,x)|^2$ a.s. So,  on the domain $|h|\leq 1$} 
   \begin{align*}
      &\int_{|h|\leq 1} \|u_{\ep}(t,\cdot+h)-u_{\ep}(t,\cdot)\|^2_{L_{\lambda}^p(\Omega\times\RR)} |h|^{2H-2}dh \\
     \leq & \varliminf \limits_{n\rightarrow \infty} \int_{|h|\leq 1} \|u^n_\ep(t,\cdot+h)-u^n_\ep(t,\cdot)\|^2_{L_{\lambda}^p(\Omega\times\RR)} |h|^{2H-2}dh.
   \end{align*}
   For $|h|\geq 1$, we simply bound $\|u_{\ep}(t,\cdot+h)-u_{\ep}(t,\cdot)\|^2_{L_{\lambda}^p(\Omega\times\RR)}$ by $2\|u_{\ep}^n(t,\cdot)\|^2_{L_{\lambda}^p(\Omega\times\RR)}$, {  which is uniform bounded with respect to $t,\ep$ and $n$. }  When $ H<\frac 12$, $\int_{|h|>1}|h|^{2H-2}<\infty$. Thus,  we have  that
   \begin{align}
   \sup_{t\in[0,T]} \cN^*_{\frac 12-H,p}u_{\ep}(t)=&\sup_{t\in[0,T]}\lc\int_{\RR} \|u_{\ep}(t,\cdot+h)-u_{\ep}(t,\cdot)\|^2_{L_{\lambda}^p(\Omega\times\RR)} |h|^{2H-2}dh\rc^{\frac 12}\nonumber \\
   	\leq& C_H\sup_{n\geq 1}\sup_{t\in[0,T]}\Psi_{\ep}^{n}(t)<\infty\,.
   	\label{e.4.53a} 
   \end{align}
   Therefore,    $\sup\limits_{\ep>0}\sup\limits_{t\in[0,T]} \cN^*_{\frac 12-H,p}u_{\ep}(t)$ is finite. 
      
   In conclusion, {   we have proved $\sup_{\ep>0}\|u_\ep\|_{\cZ_{\lambda,T}^p}:=\sup\limits_{\ep>0}\sup\limits_{t\in[0,T]} \|u_{\ep}(t,\cdot)\|_{L_{\lambda}^p(\Omega\times\RR)}+\sup\limits_{\ep>0}\sup\limits_{t\in[0,T]}\cN^*_{\frac 12-H,p}u_{\ep}(t)$ is finite.   }
\end{proof}
 Recall that $(\cC([0, T]\times\RR),d_{\cC})$ is the metric space with the metric $d_{\cC}$ defined by \eqref{e.4.metric}.
 \begin{lemma}\label{ConveInC}  Let $  u_\ep\in \cZ_{\lambda,T}^p$.
   If $u_\ep\rightarrow u$  almost surely
   in $(\cC([0, T]\times\RR),d_{\cC})$   as $\ep\rightarrow0$, then $u$ is also in $\cZ_{\lambda,T}^p$.
 \end{lemma}
 \begin{proof}
   Since $u_\ep$ converges to $u$ in $(\cC([0, T]\times\RR),d_{\cC})$ almost surely, we have $u_\ep(t,x)\rightarrow u(t,x)$ for each $(t,x)\in[0,T]\times\RR$ almost surely. Thus
   \begin{align}
     \|u(t,\cdot)\|_{L_{\lambda}^p(\Omega\times\RR)} 
       &\lesssim \varliminf \limits_{\ep\rightarrow 0}\lc\int_{\RR} \EE\lk|u_\ep(t,x)|^p\rk \lambda(x)dx\rc^{\frac{1}{p}}<\infty.\label{e.4.49}
   \end{align}
   This means  that $\sup\limits_{t\in[0,T]} \|u(t,\cdot)\|_{L_{\lambda}^p(\Omega\times\RR)}$ is finite.

   On the other hand, for any $  x,h$ we have $|u_\ep(t,x+h)-u_\ep(t,x)|^2\rightarrow |u(t,x+h)-u(t,x)|^2$ almost surely.   So, 
    on the domain $|h|\leq 1$
   and $|h|\geq 1$, we can simply repeat  the same procedure as in  the 
  {Step 3}  of  the proof of Lemma \ref{UniBExist} 
   but replacing $\varliminf \limits_{n\rightarrow \infty}$ by $\varliminf \limits_{\ep\rightarrow 0}$, and bound $\|u(t,\cdot+h)-u(t,\cdot)\|^2_{L_{\lambda}^p(\Omega\times\RR)}$ by $2\|u(t,\cdot)\|^2_{L_{\lambda}^p(\Omega\times\RR)}$,   which is finite. 
Thus, similar to \eqref{e.4.53a} we   have 
   \[
    \sup_{t\in[0,T]} \cN^*_{\frac 12-H,p}u(t)=\sup_{t\in[0,T]}\lc\int_{\RR} \|u(t,\cdot+h)-u(t,\cdot)\|^2_{L_{\lambda}^p(\Omega\times\RR)} |h|^{2H-2}dh\rc^{\frac 12}<\infty.
   \]
   Together with \eqref{e.4.49}, this implies  that $u\in\cZ^p_{\lambda,T}$.
\end{proof}

 \begin{lemma}\label{TimeSpaceRegBdd}
   Let $u_\ep$ be   the  {  approximate  mild solution }  defined by  \eqref{MildSolRegu} and assume   that $u_0(x)$ belongs to $\cZ_{\lambda,0}^p$.
   Then, we have the following statements. 
   \begin{enumerate}[leftmargin=*]
     \item[\textbf{(i)}] If $p>\frac{6}{4H-1}$, then
     \begin{equation}
       \Big\|\sup\limits_{t\in[0,T], x\in \RR}\lambda^{\frac{1}{p}}(x) \cN_{\frac 12-H}u_\ep(t,x)\Big\|_{L^{p}(\Omega)}\leq C_{T,H}(\|u_\ep\|_{\cZ_{\lambda,T}^p}+1).
 \label{e.4.55a}     \end{equation}
     \item[\textbf{(ii)}] If $p>\frac{3}{H}$, then
     \begin{equation}
       \Big\|\sup\limits_{\substack{t,t+h\in[0,T] \\ x\in \RR}}\lambda^{\frac{1}{p}}(x) \blk u_\ep(t+h,x)-u_\ep(t,x)\brk\Big\|_{L^p(\Omega)}\leq C_{T,H}|h|^{\gamma}(\|u_\ep\|_{\cZ_{\lambda,T}^p}+1),
  \label{e.4.56a}     \end{equation}
     for all $0<\gamma<\frac H2-\frac{3}{2p}$.
     \item[\textbf{(iii)}] If $p>\frac{3}{H}$, then
     \begin{equation}
       \lt\|\sup\limits_{\substack{t\in[0,T] \\ x,y\in \RR}}\frac{u_\ep(t,x)-u_\ep(t,y)}{\lambda^{-\frac{1}{p}}(x)+\lambda^{-\frac{1}{p}}(y)} \rt\|_{L^p(\Omega)}\leq C_{T,H}|x-y|^{\gamma}(\|u_\ep\|_{\cZ_{\lambda,T}^p}+1),
   \label{e.4.57a}    \end{equation}
     for all $0<\gamma<H-\frac{3}{p}$.
   \end{enumerate}
 \end{lemma}
 \begin{proof}  Denote  for $\alpha\in[0,1]$
    \[
      J^\ep_{\alpha}(r,\xi)=\int_{0}^{r}\int_{\RR}\int_{\RR} (r-s)^{-\alpha} G_{r-s}(\xi-z)\sigma(u_\ep(s,z))G_\ep(z-y)dz W(ds,dy).
     \]
     Then,  Fubini's theorem    implies
\begin{align*}
      u_\ep(t,x)
      =&G_t\ast u_0(x)+\frac{\sin(\pi\alpha)}{\pi}\int_{0}^{t}\int_{\RR}(t-r)^{\alpha-1}G_{t-r}(x-\xi)J^\ep_{\alpha}(r,\xi)d\xi dr\\
      =&u_{1 }(t,x)+u_{2,\ep}(t,x)\,. 
     \end{align*}  
 Applying Proposition \ref{CharProp}
      \textbf{(ii)}, \textbf{(iii)}, \textbf{(iv)}   to 
      $u_{2,\ep}(t,x)$ yields \eqref{e.4.55a}-\eqref{e.4.57a} without the constant term $1$. However, from the assumption    that $u_0(x)$ belongs to $\cZ_{\lambda,0}^p$ we see that left hand sides of \eqref{e.4.55a}-\eqref{e.4.57a}
      are finite when $u_\ep(t,x) $ is  replaced by $u_{1 }(t,x)$. 
Combining the bounds for  $u_{1 }(t,x)$ and $u_{2, \ep  }(t,x)$  proves the lemma. 
 \end{proof}
%
 
\begin{proof}[Proof of Theorem \ref{WeakSol}]
 We still assume $\sigma(t,x,u)=\sigma(u)$ to simplify the notations.
 From   Lemma \ref{UniBExist} and  Lemma \ref{TimeSpaceRegBdd}
  \textbf{(ii)} and \textbf{(iii)} it follows    that the two conditions of Theorem \ref{TightMoment} are satisfied. Hence,  the  probability measures on the space $(\cC([0, T]\times\RR),\sB(\cC([0, T]\times\RR)),d_{\cC})$ corresponding to  the processes $\{u_{\ep}\,, \ep\in (0, 1]\}$   are tight. Thus, there is a subsequence $\ep_n\downarrow 0$ such that    $u_n=u_{\ep_n}$ convergence weakly. By Skorohod representation theorem, there is a probability space $(\widetilde{\Omega},\widetilde{\cF},\widetilde{\bP})$ carrying the subsequence $\widetilde{u}_{n_j}$ and noise $\widetilde{W}$ such that the finite dimensional distributions of $(\widetilde{u}_{n_j},\widetilde{W})$ and $(u_{n_j},W)$ coincide. Moreover,   we have
 \begin{equation}\label{Skorohod}
   \widetilde{u}_{n_j}(t,x)\rightarrow \widetilde{u}(t,x)\,~\text{in}~(\cC([0, T]\times\RR),d_{\cC}) \quad
   \hbox{$\widetilde{\bP}$-almost surely}
 \end{equation}
 for a certain stochastic process $\widetilde{u}$ as $j\rightarrow \infty$. By Lemma \ref{ConveInC} we see that $\widetilde{u}$ belongs to space $\widetilde \cZ_{\lambda,T}^p$ { with respect to   the new probability $\widetilde{\bP}$. }  We want to show that  $\widetilde{u}$ is a weak solution to   \eqref{SHE}.

  Define the filtration $\widetilde{\cF}_t$ 
  to  be  the filtration generated by $\widetilde{W}$.  We claim that $\widetilde{u}_{n_j}$ satisfies \eqref{SHE} with $W$ replaced by $\widetilde{W}$, namely,
 \begin{equation}\label{WeakSHE}
   \widetilde{u}_{n_j}(t,x)=G_t\ast u_0(x)+\int_{0}^{t}\int_{\RR} G_{t-s}(x-\cdot)\sigma(\widetilde{u}_{n_j}(s,\cdot))\ast G_{\ep_j}(y)\widetilde{W}(ds,dy)\,.
 \end{equation}
To show the above identity  it is sufficient to prove that  for any $Z\in L^2(\widetilde{\Omega},\widetilde{\bP})$ one has
\begin{align}
\widetilde{\EE}[\widetilde{u}_{n_j}(t,x) Z]= \widetilde{\EE}&\bigg[ G_t\ast u_0(x)Z\nonumber\\
&+\int_{0}^{t}\int_{\RR} G_{t-s}(x-\cdot)\sigma(\widetilde{u}_{n_j}(s,\cdot))\ast G_{\ep_j}(y)\widetilde{W}(ds,dy) Z\bigg]\,,
\label{e.4.55}
\end{align}
where $\widetilde{\EE}$ means the expectation under $\widetilde{\bP}$.

For any $\phi\in \cD(\RR)$, denote
\[
\widetilde{W}_t(\phi)=\int_\RR \phi(x) \tilde W(t, dx)\,;\quad
 {W}_t(\phi)=\int_\RR \phi(x)   W(t, dx)\,. 
\]
It is routine to argue that  the set
\[
\cS:=\bigg\{ f(\widetilde{W}_{t_1}(\phi),\cdots,\widetilde{W}_{t_n}(\phi))\,, \ 0\leq t_1< \cdots < t_n\leq T\,\ f\in \cC_0(\RR^n)\
\bigg\}
\]
 are dense in $L^2(\widetilde{\Omega},\widetilde{\bP},\widetilde{\cF}_T)$. This means that  it is  sufficient  to choose
$ Z=f(\widetilde{W}_{t_1}(\phi),\\ \cdots,\widetilde{W}_{t_n}(\phi))$  in \eqref{e.4.55},
 which is true because we have the following identities:
 \[
  \widetilde{\EE}[\widetilde{u}_{n_j}(t,x) f(\widetilde{W}_{t_1}(\phi),\cdots,\widetilde{W}_{t_n}(\phi))]={\EE}[{u}_{n_j}(t,x) f({W}_{t_1}(\phi),\cdots,{W}_{t_n}(\phi))]\,;
 \]
 \[
  \widetilde{\EE}\lk G_t\ast u_0(x) f(\widetilde{W}_{t_1}(\phi),\cdots,\widetilde{W}_{t_n}(\phi))\rk={\EE}\lk G_t\ast u_0(x) f({W}_{t_1}(\phi),\cdots,{W}_{t_n}(\phi))\rk\,;
 \]
 and
 \begin{align*}
   &\widetilde{\EE}\lk \int_{0}^{t}\int_{\RR} G_{t-s}(x-\cdot)\sigma(\widetilde{u}_{n_j}(s,\cdot))\ast G_{\ep_j}(y)\widetilde{W}(ds,dy) f(\widetilde{W}_{t_1}(\phi),\cdots,\widetilde{W}_{t_n}(\phi))\rk \\
  =& {\EE}\lk \int_{0}^{t}\int_{\RR} G_{t-s}(x-\cdot)\sigma({u}_{n_j}(s,\cdot))\ast G_{\ep_j}(y){W}(ds,dy) f({W}_{t_1}(\phi),\cdots,{W}_{t_n}(\phi))\rk
 \end{align*}
 due to the fact that the finite dimensional distributions of $(\widetilde{u}_{n_j},\widetilde{W})$ coincide with that of $(u_{n_j},W)$.
 Therefore,    $\tilde u_{n_j}(t,x)$ satisfies 
 \eqref{e.4.55}, and hence it satisfies (\ref{WeakSHE}).

From \eqref{Skorohod} and \eqref{WeakSHE}  it follows  that $\widetilde{u}$ is   a mild solution to \eqref{SHE} with  $W$ replaced by  $\widetilde{W}$. Therefore,  we have proved 
 the existence of a  weak solution to \eqref{SHE}.

Moreover,  for any $\gamma\in(0,H-\frac{3}{p})$ and for
any compact set $\TT\subseteq [0,T]\times\RR$, Lemma \ref{TimeSpaceRegBdd}  (parts \textbf{(ii)} and \textbf{(iii)})  implies that  there exists constant $C$ such that
   \begin{equation}\label{STRegSol}
     \widetilde{\EE}\lc \sup\limits_{(t,x),(s, y)\in \TT} \bigg|\frac{\tilde{u}(t,x)-\tilde{u}(s,y)}{|t-s|^{\frac{\gamma}{2}}+|x-y|^{\gamma}}\bigg|^p\rc \leq C\|\tilde{ u}\|^p_{\cZ^p_{\lambda,T}}\,.
   \end{equation}
This combined with the Kolmogorov lemma implies the desired
H\"older continuity.
\end{proof}

\section{Pathwise Uniqueness and Strong Existence of solutions}
\label{strong}


In this section we prove   the pathwise uniqueness and  the  existence of strong solution for the equation \eqref{SHE}. {   It is well known that once pathwise uniqueness is achieved, together with  the  existence of weak solution proved in previous section,  we can conclude the   existence of the unique  strong  solutions to \eqref{SHE} by,   for example,   the Yamada-Watanabe theorem  (\cite{IW}). }  Therefore,  we only need to focus on the proof of pathwise uniqueness.
\begin{proof}[Proof of Theorem \ref{ModUniq}]
The proof follows the   strategy in the proof of Theorem 4.3 of  \cite{HHLNT2017} combined with   Proposition \ref{CharProp}  (part \textbf{(ii)}).

Define  the following  stopping times
  \begin{equation*}
    \begin{split}
        T_k:= \inf \Big\{&t\in[0,T]:  \sup_{0\leq s\leq t,x\in\RR} \lambda^{\frac 2p}(x)\cN_{\frac 12-H}u(s,x)\geq k, \\
         &\text{or}~\sup_{0\leq s\leq t,x\in\RR} \lambda^{\frac 2p}(x)\cN_{\frac 12-H}v(s,x)\geq k  \Big\}\,,\quad k=1, 2, \cdots
    \end{split}
  \end{equation*}
  Proposition \ref{CharProp},  part \textbf{(ii)}  implies that $T_k \uparrow T$ almost surely as $k\rightarrow\infty$. We need to find appropriate bounds for   the following two quantities:
  \[
   I_1(t)=\sup_{x\in\RR}\EE\blk\1_{\{t<T_k\}}|u(t,x)-v(t,x)|^2\brk
  \]
  and
  \[
   I_2(t)=\sup_{x\in\RR}\EE\lt[\int_{\RR}\1_{\{t<T_k\}}|u(t,x)-v(t,x)-u(t,x+h)+v(t,x+h)|^2|h|^{2H-2}dh\rt].
  \]
First, it is easy to see
  \begin{equation*}
    \begin{split}
       &\1_{\{t<T_k\}}(u(t,x)-v(t,x)) \\
      =&\1_{\{t<T_k\}} \int_{0}^{t}\int_{\RR}G_{t-s}(x-y)\1_{\{s<T_k\}}  [ {\sigma}(s,y, u(s,y))- {\sigma}(s,y , v(s,y))] W(ds,dy).
    \end{split}
  \end{equation*}
  Recall $D_t(x,h)$ defined in \eqref{FirDiff} and denote $\triangle(t,x,y)=\sigma(t,x,u(t,y))-\sigma(t,x,v(t,y))$. We can decompose
  \begin{align}
       &\EE\blk\1_{\{t<T_k\}}|u(t,x)-v(t,x)|^2\brk \nonumber\\
     \lesssim& \EE\bigg(\int_{0}^{t}\int_{\RR^2}\1_{\{s<T_k\}} |D_{t-s}(x-y,h)|^2 [ \triangle(s,y,y)]^2|h|^{2H-2}dhdyds\bigg) \nonumber\\
     +&\EE\bigg(\int_{0}^{t}\int_{\RR^2}\1_{\{s<T_k\}} G^2_{t-s}(x-y-h)  [\triangle(s,y+h,y)-\triangle(s,y,y) ]^2|h|^{2H-2}dhdyds\bigg) \nonumber\\
     +&\EE\bigg(\int_{0}^{t}\int_{\RR^2}\1_{\{s<T_k\}}  G_{t-s}^2(x-y)  [\triangle(s,y,y+h)-\triangle(s,y,y) ]^2|h|^{2H-2}dhdyds\bigg)\nonumber\\
     &{  =: J_1+J _2+J _3} \,.
     \label{e.5.1}
  \end{align}
The assumption    \eqref{DuSigam} of $\sigma$   and the  equality  \eqref{FirDiffIneq}   can be used to dominate the above
first term $J_1$. This is,
  \begin{align*}
  {   J_1 }\lesssim& \EE\bigg(\int_{0}^{t}\int_{\RR^2}\1_{\{s<T_k\}} |D_{t-s}(x-y,h)|^2 |u(s,y)-v(s,y)|^2|h|^{2H-2}dhdyds\bigg)\\
    \lesssim& \int_{0}^{t}(t-s)^{H-1} \sup_{y\in\RR}\EE\lk\1_{\{s<T_k\}}|u(s,y)-v(s,y)|^2\rk ds=\int_{0}^{t}(t-s)^{H-1} I_1(s)ds\,.
  \end{align*}
  Using the properties \eqref{DuSigam} of $\sigma$,   we
  have if $|h|>1$
  \begin{align*}
     &[\triangle(s,y+h,y)-\triangle(s,y,y)]^2\lesssim |u(s,y)-v(s,y)|^2 \\
    =&\lt|\int_{u}^{v} [\sigma'_{\xi}(s,y+h,\xi)-\sigma'_{\xi}(s,y,\xi)] d\xi\rt|^2\lesssim |u(s,y)-v(s,y)|^2\,,
  \end{align*}
  and if $|h|\leq 1$  (with the help of  additional properties \eqref{DxuSigam})
  \begin{align*}
    &[\triangle(s,y+h,y)-\triangle(s,y,y)]^2 \\
   =&\lt|\int_{u}^{v} [\sigma'_{\xi}(s,y+h,\xi)-\sigma'_{\xi}(s,y,\xi)] d\xi\rt|^2\lesssim |h|^2|u(s,y)-v(s,y)|^2\,.
  \end{align*}
  Thus,  the second term  $J_2$  in \eqref{e.5.1}
  is bounded by
  \begin{align*}
   J_2= &\EE\bigg(\int_{0}^{t}\int_{\RR}\int_{|h|>1}\1_{\{s<T_k\}} G^2_{t-s}(x-y-h) |u(s,y)-v(s,y)|^2|h|^{2H-2}dhdyds\bigg)\\
    &+\EE\bigg(\int_{0}^{t}\int_{\RR}\int_{|h|\leq 1}\1_{\{s<T_k\}} G^2_{t-s}(x-y-h) |u(s,y)-v(s,y)|^2|h|^{2H}dhdyds\bigg)\\
   \lesssim& \int_{0}^{t}I_1(s)\lc\int_{\RR} G^2_{t-s}(x-y)dy\rc ds\lesssim\int_{0}^{t} (t-s)^{-\frac 12} I_1(s)ds\,.
  \end{align*}
  For the last term  $J_3$ 
  in \eqref{e.5.1} we have by \eqref{DuSigam},  \eqref{DuSigamAdd}
  \begin{align*}
    &\big|\triangle(s,y,y+h)-\triangle(s,y,y)\big|^2 \\
    =&\Big|\int_{0}^{1}[u(s,y+h)-v(s,y+h)]\sigma'_{\xi}(s,y,\theta u(s,y+h)+(1-\theta) v(s,y+h))d\theta \\
    &\quad-\int_{0}^{1}[u(s,y)-v(s,y)]\sigma'_{\xi}(s,y,\theta u(s,y)+(1-\theta) v(s,y))d\theta \Big|^2 \,. 
  \end{align*}
  Noticing the additional uniform decay assumption
 \eqref{DuSigam}, we have
 \begin{align*}
    &\big|\triangle(s,y,y+h)-\triangle(s,y,y)\big|^2 \\
    \lesssim& |u(s,y+h)-v(s,y+h)-u(s,y)+v(s,y)|^2 \\
    +&\lambda^{\frac{2}{p}}(y)|u(s,y)-v(s,y)|^2\cdot\blk|u(s,y+h)-u(s,y)|^2+|v(s,y+h)-v(s,y)|^2\brk\,.
  \end{align*}
Thus,   we can dominate  the last term in \eqref{e.5.1} by
  \begin{equation*}
    J_3\lesssim   k\int_{0}^{t} (t-s)^{-\frac 12} \blk I_1(s)+I_2(s)\brk ds\,.
  \end{equation*}
 Summarizing  the above  estimates   we have
  \begin{equation*}
    I_1(t)\lesssim k\int_{0}^{t} (t-s)^{H-1}\blk I_1(s)+I_2(s)\brk ds\,.
  \end{equation*}
  The similar procedure can be applied to estimate  the term $I_2(t)$  to obtain 
  \begin{equation*}
    I_2(t)\lesssim k\int_{0}^{t} (t-s)^{2H-\frac{3}{2}}\blk I_1(s)+I_2(s)\brk ds\,.
  \end{equation*}
  As a consequence,
  \begin{equation*}
    I_1(t)+I_2(t)\lesssim k\int_{0}^{t} (t-s)^{2H-\frac{3}{2}}\blk I_1(s)+I_2(s)\brk ds.
  \end{equation*}
Now  Gronwall's lemma implies $I_1(t)+I_2(t)=0$ for all $t\in[0,T]$.
In   particular,  we have
  \[
   \EE\blk\1_{\{t<T_k\}}|u(t,x)-v(t,x)|^2\brk=0\,.
  \]
  Thus,  we have $u(t,x)=v(t,x)$ almost surely on $\{t<T_k\}$ for all $k\geq 1$, and the fact $T_k \uparrow \infty$ a.s as $k$ tends to infinity necessarily indicate $u(t,x)=v(t,x)$ a.s. for every $t\in[0,T]$ and $x\in\RR$.

  It is clear   that the hypothesis \textbf{(H2)} implies the hypothesis  \textbf{(H1)}. So the existence of a H\"{o}lder continuous modification version of the solution follows from Theorem \ref{WeakSol}. We have
  then  completed the proof  of Theorem \ref{ModUniq}.
\end{proof}

\bibliographystyle{amsplain}

\end{document}